\documentclass[smallextended]{svjour3} 
\usepackage[english]{babel}
\usepackage[utf8]{inputenc}

\usepackage{fullpage,comment,placeins}
\usepackage{subcaption}
\usepackage{amsmath,amssymb,graphicx,psfrag,color}

\def\C{{\mathbb C}}

\def\P{{\mathbb P}}
\def\R{{\mathbb R}}

\def\N{{\mathbb N}}

\def\T{{\mathbb T}}

\def\ZZ{\mathcal{Z}}
\def\bf{\boldsymbol{f}}

\def\bx{{\boldsymbol{x}}}
\def\by{{\boldsymbol{y}}}
\def\bz{{\boldsymbol{z}}}
\def\bv{{\boldsymbol{v}}}

\def\bq{{\boldsymbol{q}}}
\def\br{{\boldsymbol{r}}}

\def\KK{{\mathcal K}}

\def\NN{{\mathcal N}}

\def\PP{{\mathcal P}}

\newcommand{\dist}[3][]{{\rm d\!l}[\ifthenelse{\equal{#1}{}}{}{#1;}#2,#3]}
\newcommand{\eff}[3][]{{\rm osc}\ifthenelse{\equal{#1}{}}{}{_{#1}}(#2;#3)}

\newcommand{\norm}[3][]{#1\|#2#1\|_{#3}}

\def\set#1#2{\big\{#1\,:\,#2\big\}}

\def\eps{\varepsilon}

\def\normL2#1#2{\|#1\|_{L^2(#2)}}
\newtheorem{algorithm}{Algorithm}
\newtheorem{assumption}{Assumption}

\begin{document}

\begin{abstract}
We use the $H$-matrix technology to compute the approximate square root of a 
covariance matrix in linear cost. This allows us to generate normal and 
log-normal random fields
on general point sets with optimal cost. We derive rigorous error 
estimates which show convergence of the method. Our approach requires only mild assumptions on the covariance function and on the point set. 
Therefore, it might be also a nice alternative to the circulant embedding approach which applies only to regular 
grids and stationary covariance functions.
\end{abstract}
\title{Fast random field generation with $H$-matrices}


\author{Michael Feischl \and Frances Y. Kuo \and Ian H. Sloan}


\institute{M. Feischl, F.Y. Kuo, I.H. Sloan \at
              School of Mathematics and Statistics, UNSW Sydney, NSW 2052 \\
              Tel.: +61-2-93857076\\
              \email{m.feischl@unsw.edu.au, f.kuo@unsw.edu.au, i.sloan@unsw.edu.au}                  
}

\date{Received: date / Accepted: date}
\maketitle

\section{Introduction}
Generating samples of random fields is a common bottleneck in simulation and modeling of real life phenomena as, e.g., structural vibrations~\cite{ap2}, groundwater flow~\cite{flow},
and composite material behavior~\cite{ap1}.
 A standard approach is to truncate the Karhunen-Lo\`eve expansion of the random  field.
This can, particularly for rough fields with short correlation length, be very expensive, as many summands of the expansion have to be evaluated to compute a decent approximation.
Often, it suffices to evaluate the random field only on some particular (quadrature) nodes. If the random field $\ZZ(\bx,\omega)$ is Gaussian with given covariance function $\varrho(\cdot,\cdot)$, it is well-known that 
the evaluation at the quadrature nodes $\bx_1,\ldots,\bx_N$ can be done by computing the square-root of the corresponding covariance matrix 
$\boldsymbol{C}=(\varrho(\bx_i,\bx_j))_{i,j\in\{1,\ldots,N\}}\in\R^{N\times N}$, i.e.,
\begin{align*}
 \ZZ(\bx_i,\omega)=\big(\boldsymbol{C}^{1/2}\bz(\omega)\big)_i\quad\text{for all }i\in\{1,\ldots,N\},
\end{align*}
where $\bz(\omega)\in \R^N$ is a vector of i.i.d. standard normal random numbers.
Since each evaluation requires a matrix-vector multiplication with $\boldsymbol{C}^{1/2}$, a direct approach requires $\mathcal{O}(N^2)$ operations for the multiplication plus $\mathcal{O}(N^3)$ operations
for 
computing the square-root itself  and thus is
prohibitively expensive. An efficient method first proposed in~\cite{circ1,circ2} is \emph{circulant embedding}, which employs fast FFT techniques to realize the factorization and the 
matrix-vector multiplication in $\mathcal{O}(N\log(N))$
operations.
This approach, however, works solely for stationary covariance functions $\varrho(\bx,\by)=\rho(|\bx-\by|)$ and
regular grids of quadrature nodes. Since non-stationary covariance functions are of great interest for the modeling of natural structures (e.g., porous rock, wood,\ldots), and since finite element methods often
use irregular grids,
we propose a new method which removes both restrictions. 

The idea is to approximate the covariance matrix $\boldsymbol{C}$ by an $H^2$-matrix, as described in, e.g,~\cite{h2mat}, and to use an iterative method to compute an approximation $\ZZ_{k,p}(\bz)$ 
($k$ and $p$ are parameters of the methods, see below) to $\boldsymbol{C}^{1/2}\bz$ for any $\bz\in\R^N$. We therefore obtain the approximation to 
the random field by feeding the algorithms with i.i.d. standard normal
random vectors $\bz(\omega)\in\R^N$, i.e.,
\begin{align*}
 \ZZ(\bx_i,\omega)\approx \ZZ_{k,p}(\bz(\omega))_i\quad\text{for all }i\in\{1,\ldots,N\}.
\end{align*}%
This is feasible since matrix-vector multiplication with $H^2$-matrices can be done in $\mathcal{O}(N)$ operations. The only assumption on the covariance function of the random field
is that it is asymptotically smooth. We propose two iterative algorithms, each with individual advantages for smooth or rough random fields.
This algorithms might also be of interest for the approximation of random fields with covariance kernels of random solutions of certain stochastic operator equations, as considered in~\cite{schwab}.

The idea to use $H$-matrices for random field approximation has already been used indirectly in~\cite{h2KL,harbrecht}, 
where the authors efficiently compute eigenfunctions of the covariance operator by use of $H$-matrix techniques.


\subsection{Notation}
Throughout the text, $\alpha\lesssim\beta$ denotes $\alpha\leq C\beta$ for some generic constant $C>0$ and $\alpha\simeq\beta$ means $\alpha\lesssim\beta$ and $\beta\lesssim\alpha$. The notation $|\cdot|$ has several unambiguous meanings: for vectors, it
denotes the euclidean norm,  while for sets, $|\cdot|$ is the natural measure, which is the Lebesgue measure (volume, area) for continuous sets and the counting measure (cardinality) for finite sets.
The notation $\norm{\cdot}{2}$ is used for the spectral matrix norm and $|\bz|_p:=(\sum_{j=1}^N|\bz_i|^p)^{1/p}$ for all $\bz\in\R^N$ denotes the $\ell_p$-norm.
By $\PP^k$ we denote the set of polynomials of maximal degree $k$. For brevity, we write $|\cdot|:=|\cdot|_2$.
We denote the maximal and minimal eigenvalues of a positive definite and symmetric matrix $\boldsymbol{M}\in\R^{N\times N}$ by
\begin{align*} 
\lambda_{\rm max}(\boldsymbol{M}):=\sup_{\bz\in\R^N\setminus\{0\}}\frac{|\boldsymbol{M}\bz|}{|\bz|}
\quad\text{and}\quad \lambda_{\rm min}(\boldsymbol{M}):=\inf_{\bz\in\R^N\setminus\{0\}}\frac{(\boldsymbol{M}\bz)^T\bz}{|\bz|^2}.
\end{align*}
We denote the $k$-th component of a vector $\bv\in\R^N$  by $\bv_k$, whereas sequences of vectors are denoted by $\bv^1,\bv^2,\ldots$.
\section{Model Problem}\label{sec:modprob}
Let $(\Omega,\Sigma  ,\P)$ be a probability space and let $D\subseteq \R^d$, $d\in\N$ be a Lipschitz domain. We consider a random
field  which is normal or log-normal,
\begin{align*}
 \ZZ(\bx,\omega)\quad\text{or}\quad \exp(\ZZ(\bx,\omega))\quad\text{for all }\omega\in\Omega,\,\bx\in D
\end{align*}
for some zero-mean Gaussian random field $\ZZ(\cdot,\cdot)$ (note that the assumption on the mean is purely for brevity of presentation). The covariance 
function $\varrho\colon D\times D\to \R$ of $\ZZ(\cdot,\cdot)$ is assumed asymptotically smooth: that is, $\varrho\in C^\infty\big(\set{(\bx,\by)\in D\times D}{\bx\neq \by}\big)$ and
there
exist constants $c_1,c_2>0$ such that
\begin{align}\label{eq:as}
|\partial_\bx^\alpha \partial_\by^\beta \varrho(\bx,\by)|\leq 
c_1(c_2|\bx-\by|)^{-|\alpha|_1-|\beta|_1}|\alpha+\beta|_1!\quad\text{for all }\bx\neq \by\in 
D,
\end{align}
for all multi-indices $\alpha,\beta\in\N_0^d$ with $|\alpha|_1+|\beta|_1\geq 1$. 
(The expert reader will notice that the original definition of asymptotically smooth 
includes a singularity order. As our
covariance functions are always finite in value, we do not consider this.)
The goal of this work is to derive an efficient method which evaluates the random field at certain (quadrature) points $\NN\subseteq D$, where $\NN=\{\bx_1,\ldots,\bx_N\}$ is a finite set, i.e., we aim to approximate
\begin{align*}
 \Big(\ZZ(\bx,\omega)\Big)_{\bx\in\NN}\in\R^N \quad\text{or}\quad  \Big(\exp(\ZZ(\bx,\omega))\Big)_{\bx\in\NN}\in\R^N
\end{align*}
for given $\omega\in\Omega$.
\subsection{Examples of valid covariance functions}
The condition above includes the important class of isotropic  stationary 
covariance functions of Mat\'ern form, e.g.,
\begin{align}\label{eq:matern}
 \varrho(\bx,\by)= \sigma^2\frac{2^{1-\mu}}{\Gamma(\mu)}\Big(\sqrt{2\mu}\frac{|\bx-\by|_p}{\lambda}\Big)^\mu K_\mu\Big(\sqrt{2\mu}\frac{|\bx-\by|_p}{\lambda}\Big),
\end{align}
where $\Gamma(\cdot)$ is the gamma function, $K_\mu$ is the modified Bessel function of second kind, and $\lambda,\sigma>0$, $\mu\in (0,\infty]$, $p\in\N$ are parameters. 
For $\mu=1/2$, the above function takes the form
\begin{align*}
 \varrho(\bx,\by)=\sigma^2\exp\Big(-\frac{|\bx-\by|_p}{\lambda}\Big)
\end{align*}
and the limit case $\mu=\infty$ satisfies
\begin{align*}
 \varrho(\bx,\by)=\sigma^2\exp\Big(-\frac{|\bx-\by|_p^{2}}{2\lambda^2}\Big).
\end{align*}

Also much more general non-stationary, non-isotropic covariance functions, e.g., 
\begin{align}\label{eq:nonstat}
 \varrho(\bx,\by):=\sigma^2\frac{{\rm det}(\boldsymbol{\Sigma}_\bx)^{1/4}{\rm 
det}(\boldsymbol{\Sigma}_\by)^{1/4}}{\sqrt{2}{\rm det}(\boldsymbol{\Sigma}_\bx+\boldsymbol{\Sigma}_\by)^{1/2}} \exp\Big( 
-(\bx-\by)^T\frac{(\boldsymbol{\Sigma}_\bx+\boldsymbol{\Sigma}_\by)^{-1}}{2}(\bx-\by)\Big).
\end{align}
satisfy the assumptions.
Here, $\boldsymbol{\Sigma}_{(\cdot)}\colon D\to \R^{d\times d}$ is a smooth mapping into the  
symmetric positive definite matrices and $\sigma>0$ is a parameter. This 
covariance function was first suggested in~\cite{cov}
to model spatially dependent anisotropies in a material. 

\begin{lemma}\label{lem:cov}
 The covariance functions from~\eqref{eq:matern} 
satisfy~\eqref{eq:as}. Assume the mapping $\bx\mapsto \boldsymbol{\Sigma}_\bx$ satisfies (for any matrix norm $\norm{\cdot}{}$)
\begin{align}\label{eq:boundderiv}
 \sup_{\alpha\in \N^d}\sup_{\bx\in D}\norm{\partial_\bx^\alpha \boldsymbol{\Sigma}_\bx}{}<\infty.
\end{align}
Then, the covariance function from~\eqref{eq:nonstat} is asymptotically smooth~\eqref{eq:as}.
\end{lemma}
We postpone the proof of the lemma to Appendix~\ref{app:lemma1}.

\section{Sampling the random field}\label{sec:sample}
By 
definition,
$\ZZ(\bx,\cdot)$, $\bx\in\NN$ is a Gaussian random field with covariance matrix 
$\boldsymbol{C}\in \R^{N\times N}$, $N=|\NN|$, and $\boldsymbol{C}_{ij}=\varrho(\bx_i,\bx_j)$, where we write
$\NN:=\{\bx_1,\ldots,\bx_N\}$. 
The main goal of this section is to establish a new way to 
efficiently approximate $\boldsymbol{C}^{1/2}\bz$ for given $\bz\in\R^N$.
Roughly, the strategy is to approximate $\boldsymbol{C}$ by an $H^2$-matrix and to benefit from 
the fast matrix-vector multiplication provided by it. This allows us to efficiently approximate $\boldsymbol{A}\bz$ (without actually factorizing the matrix $\boldsymbol{C}$).

\subsection{$H^2$-matrix approximation of the covariance matrix}\label{sec:h2}
Given the finite set of evaluation points $\NN:=\{\bx_1,\ldots,\bx_N\}\subset 
D$, we approximate the covariance matrix $\boldsymbol{C}\in \R^{N\times N}$, 
$\boldsymbol{C}_{ij}:=\varrho(\bx_i,\bx_j)$
by an $H^2$-matrix $\boldsymbol{C}_p$ via interpolation of order $p\in\N$.

In the following, we recall the definition of $H^2$-matrices  and the 
approximation process as laid out in, e.g.,~\cite{h2mat}. The rough idea is to partition the index set of the 
covariance matrix into \emph{far-field} blocks, which can be approximated
efficiently by interpolation of the covariance function, and \emph{near-field} blocks, 
which are stored exactly.

\subsubsection{Block partitioning}
For each subset $X\subseteq \NN$, we denote by $B_X\subseteq \R^d$, the 
smallest axis-parallel box such that $X \subseteq B_X$.
We build a binary tree of clusters in the following way. Let $X_{\rm 
root}:=\NN= \{\bx_1,\ldots,\bx_N\}$ denote the root of the tree which has level 
zero ${\rm level}(X_{\rm root})=0$ by definition.
For each node of the tree $X$ with $|X| > C_{\rm leaf}$ for some cut-off 
constant $C_{\rm leaf}{\geq 2}$
(usually $C_{\rm leaf}\approx 20$), we define two sons of $X$ as follows: Split $B_X$ in half along its longest edge into $B_0\cup 
B_1= B_X$.
Define ${\rm sons}(X):=\{X_0,X_1\}$ with $X_0:=X\cap B_0$ and 
$X_1:= X\setminus X_0$ and set ${\rm level}(X_i)={\rm level}(X)+1$ for $i=0,1$. For a node $X$ with $|X|\leq C_{\rm leaf}$, we define ${\rm sons}(X):=\emptyset$.
This procedure generates a binary tree denoted by $\T_{\rm cl}$ (where ${\rm cl}$ stands for \emph{cluster}) and guarantees that its leaves satisfy $|X|\leq C_{\rm leaf}$.

For a parameter $\eta>0$, we consider the {\rm admissibility condition} for 
axis parallel boxes $B,B^\prime\subseteq \R^d$
\begin{align}\label{eq:adm}
 \max\{{\rm diam}(B),{\rm diam}(B^\prime)\}\leq \eta \,{\rm dist}(B,B^\prime),
\end{align}
where the euclidean distance between the bounding boxes is defined by
\begin{align*}
 {\rm dist}(B,B^\prime):=\inf_{\bx\in B,\by\in B^\prime}|\bx-\by|.
\end{align*}
The condition~\eqref{eq:adm} will be used to build the block-cluster tree $\T\subseteq 
\T_{\rm cl}\times \T_{\rm cl}$ as follows. The root of $\T$ is $(X_{\rm root},X_{\rm root})$.
For each node $(X,Y)\in\T$ of the tree, define ${\rm sons}(X,Y)$, the set of sons,  
as:
\begin{align*}
\left\{\;\qquad\parbox{0.8\textwidth}{\begin{itemize}
 \item[\textbf{if}] $B_X$ and $B_Y$ 
satisfy~\eqref{eq:adm} or if ${\rm sons}(X)=\emptyset={\rm sons}(Y)$ \textbf{set} ${\rm sons}(X,Y)=\emptyset$
 \item[\textbf{else if}]${\rm sons}(Y)\neq 
\emptyset$ and ${\rm sons}(X)=\emptyset$ \textbf{set}  ${\rm sons}(X,Y)=\{X\}\times {\rm sons}(Y)$
  \item[\textbf{else if}] ${\rm sons}(X)\neq 
\emptyset$ and ${\rm sons}(Y)=\emptyset$ \textbf{set} ${\rm sons}(X,Y)={\rm sons}(X)\times \{Y\}$
   \item[\textbf{else}]${\rm 
sons}(X)\neq \emptyset$ and ${\rm sons}(Y)\neq\emptyset$ \textbf{set} ${\rm sons}(X,Y)={\rm sons}(X)\times{\rm sons}(Y)$
\end{itemize}}\right.
\end{align*}
We also define the level as ${\rm level}(X_{\rm root}, X_{\rm root})=0$ and ${\rm 
level}(X, Y)={\rm level}(X^\prime, Y^\prime)+1$ for $(X, Y)\in {\rm sons}(X^\prime, 
Y^\prime)$.
Further, we define 
\begin{align*}
\T_{\rm far}:=\set{(X,Y)\in\T}{{\rm sons}(X,Y)=\emptyset\text{ and } 
B_X,B_Y\text{ satisfy }~\eqref{eq:adm}}
\end{align*}
as well as
\begin{align*}
\T_{\rm near}:=\set{(X,Y)\in\T}{{\rm sons}(X,Y)=\emptyset\text{ and } 
B_X,B_Y\text{ do not satisfy }~\eqref{eq:adm}}.
\end{align*}
Note that by definition of the block-cluster tree $\T$, the set $\T_{\rm near}\cup\T_{\rm far}$ contains all the leaves of $\T$. Moreover, we see that for each $(X,Y)\in \T\setminus(\T_{\rm near}\cup\T_{\rm far})$,
there holds
\begin{align*}
 X\times Y=\bigcup_{(X^\prime,Y^\prime)\in{\rm sons}(X,Y)}X^\prime\times Y^\prime
\end{align*}
Therefore, $\T_{\rm near}\cup \T_{\rm far}$ is a partition of $\NN\times \NN$
in the sense that each pair of points $(\bx_i,\bx_j)\in \NN\times 
\NN$ for $1\leq i,j\leq N$ is contained in exactly one $(X,Y)\in \T_{\rm 
near}\cup \T_{\rm far}$.

\subsubsection{Interpolation}
The blocks $(X,Y)\in\T_{\rm far}$ satisfy~\eqref{eq:adm} and hence interpolation of the kernel function is 
highly accurate. This allows us to store the matrix very efficiently.
Let $I(X):=\set{i\in\N}{\bx_i\in X}$ denote the index set of $X$. 
The basic idea now is to replace $\boldsymbol{C}|_{I(X)\times I(Y)}$ by a low-rank 
approximation $\boldsymbol{V}^X \boldsymbol{M}^{XY} (\boldsymbol{V}^Y)^T$ with $\boldsymbol{V}^X\in \R^{|X|\times p^d}$, $\boldsymbol{M}^{XY}\in 
\R^{p^d\times p^d}$, and $\boldsymbol{V}^Y\in \R^{|Y|\times p^d}$, where $p$ is 
the interpolation order. The three matrices are defined by {Chebychev} interpolation 
of the covariance function. To that end, let $\{q_1^X,\ldots,q_{p^d}^X\}$ denote 
transformed, tensorial {Chebychev} nodes in $B_X$
with the corresponding Lagrange basis functions $L_1^X,\ldots,L_{p^d}^X\colon 
B_X\to \R$.
Given $(X,Y)\in \T_{\rm far}$, we may approximate
\begin{align*}
 \varrho(\bx,\by)\approx c_p^{XY}(\bx,\by):=\sum_{n,m=1}^{p^d}\varrho(q_n^X,q_m^Y) 
L_n^X(\bx)L_m^Y(\by)\quad\text{for all }\bx\in X,\by\in Y.
\end{align*}
For $i,j\in\{1,\ldots,N\}$ and $n,m\in\{1,\ldots,p^d\}$, this leads to
\begin{align*}
 \boldsymbol{V}^X_{in}:=L^X_n(\bx_i),\quad  \boldsymbol{V}^Y_{jm}:=L^Y_m(\bx_j),\text{ and}\quad  \boldsymbol{M}^{XY}_{nm}:= 
\varrho(q_n^X,q_m^Y)
\end{align*}
and hence
\begin{align*}
 \boldsymbol{C}|_{I(X)\times I(Y)}\approx \boldsymbol{V}^X \boldsymbol{M}^{XY} (\boldsymbol{V}^Y)^T.
\end{align*}
The admissibility condition~\eqref{eq:adm} guarantees that the approximation 
error converges to zero exponentially in $p$, as we prove in Proposition~\ref{prop:h2} below.
Further note that the {Chebychev} interpolation described above is exact on polynomials of degree $p$. Thus, for $X\in \T_{\rm cl}$ and 
$\bx_i\in X^\prime\in {\rm sons}(X)$, there holds with the transfer matrices $\boldsymbol{T}^{X^\prime X}:= (L^X_n(q_m^{X^\prime}) )_{mn}\in 
\R^{p^d\times p^d}$
\begin{align*}
 \boldsymbol{V}^{X}_{in}:= L^X_n(\bx_i)= \sum_{m=1}^{p^d} 
L^X_n(q_m^{X^\prime})L^{X^\prime}_m(\bx_i)=\sum_{m=1}^{p^d} L^X_n(q_m^{X^\prime}) 
\boldsymbol{V}^{X^\prime}_{im}=(\boldsymbol{V}^{X^\prime}\boldsymbol{T}^{X^\prime X})_{in}.
\end{align*}
Thus, it suffices to store $\boldsymbol{V}^X$ only for the leaves of $\T_{\rm cl}$ together with the 
transfer matrices $\boldsymbol{T}^{X^\prime X}$. This enables very efficient storage and arithmetics
for $H^2$ matrices.

The capabilities of $H^2$-matrices which we employ in this work are summarized 
below in Proposition~\ref{prop:h2}.
To that end, we assume that the points $\NN$ are approximately uniformly 
distributed, in the following sense.

\begin{assumption}[quasi-uniform distribution]\label{ass:approxu}
 We say that $\NN$ is quasi-uniformly distributed if there exists a constant $C_{\rm u}>0$ such that
{ \begin{align*}
  C_{\rm u}^{-1}N^{-1/d}\leq \min_{\bx,\bx^\prime\in \NN}|\bx-\bx^\prime|\leq 
  \sup_{\bx\in D}\min_{\bx^\prime\in\NN}|\bx-\bx^\prime|\leq C_{\rm u}N^{-1/d}.
 \end{align*}}
%
\end{assumption}

\begin{proposition}\label{prop:h2}
Suppose we have a covariance matrix $\boldsymbol{C}\in \R^{N\times N}$ and an asymptotically smooth 
kernel $\varrho(\cdot,\cdot)$ and recall Assumption~\ref{ass:approxu} on approximate 
uniform distribution of $\NN$. 
Then, there exists a constant $C_H>0$ such that, for all $p\in\N_0$,
the $H^2$-matrix  $\boldsymbol{C}_p\in \R^{N\times N}$ constructed as above satisfies
\begin{align}\label{eq:h2err}
 \norm{\boldsymbol{C}-\boldsymbol{C}_p}{2}\leq\norm{\boldsymbol{C}-\boldsymbol{C}_p}{F}:=\Big(\sum_{i,j=1}^N|\boldsymbol{C}-\boldsymbol{C}_p|_{ij}^2\Big)^{1/2}\leq C_HN 
(\log(p)+1)^{2d-1}\Big(\frac{\eta}{4c_2}\Big)^{p}.
\end{align}
(The constant $c_2$ is defined in~\eqref{eq:as}.)
The $H^2$-matrix $\boldsymbol{C}_p$ is symmetric and can be stored using less than $C_Hp^{2d} N$ memory units. 
Moreover, given any vector $\bx\in \R^N$, it is possible to compute
$\boldsymbol{C}_px\in \R^N$ in less than $C_H p^{2d} N$ arithmetic operations.
The constant $C_H$ depends only on $C_{\rm leaf}$ and $d$.
The matrix $\boldsymbol{C}_p$ is positive definite if $p$ is sufficiently large such that
\begin{align}\label{eq:ppos}
C_HN (\log(p)+1)^{2d-1}\Big(\frac{\eta}{4c_2}\Big)^{p}<\lambda_{\rm min}(\boldsymbol{C}).
\end{align}
\end{proposition}
We postpone the proof of the lemma to Appendix~\ref{app:proph2}.
\subsection{Computing the square-root (Method 1)}
Since $\boldsymbol{C}$ is positive definite in our case, a standard method is to compute the 
Cholesky factorization $\boldsymbol{L}\boldsymbol{L}^T=\boldsymbol{C}$. This can be done using $H^2$-matrices
in {\emph{almost}} linear cost (analyzed in~\cite{hmatrices} for $H$-matrices, but the method transfers to $H^2$-matrices). However, to the authors'
best knowledge, there is no complete error analysis available, and due to 
the complicated structure of the algorithm, 
the worst-case error estimate may be overly pessimistic.
Therefore, we propose an iterative algorithm based on a variant of the Lanczos iteration. Note that polynomial or rational approximations of
the square root (as pursued in, e.g.,~\cite{rational}) are doomed to fail since smooth random fields result in very badly conditioned covariance matrices $\boldsymbol{C}$ (see also the numerical experiments below).
This implies that a polynomial approximation of the square root over the spectrum of $\boldsymbol{C}$ is very costly, whereas a rational approximation requires the inverse of $\boldsymbol{C}$ which is hard to compute due to
the bad condition number.

The idea behind the algorithm below is as follows.  Given a positive
definite symmetric matrix $\boldsymbol{M}\in\R^{N\times N}$ and a vector
$\bz\in\R^N$, the aim is to compute efficiently an approximation to
$\boldsymbol{M}^{1/2}\bz$. For arbitrary $k\leq N$ define the order-$k$ Krylov subspace of $\boldsymbol{M}$ and $\bz$ as
\begin{align}\label{eq:bigZ}
\KK_k:={\rm span}\{\bz,\boldsymbol{M}\bz,\boldsymbol{M}^2\bz, \dots,\boldsymbol{M}^{k-1}\bz).
\end{align}		
Assuming $\KK_k$ is $k$-dimensional, consider the orthogonal matrix $\boldsymbol{Q}\in\R^{N\times k}$
whose columns are the orthonormal basis vectors of the Krylov subspace, i.e., $\boldsymbol{Q}^T \boldsymbol{Q} =\boldsymbol{I}_k$ and ${\rm range}(\boldsymbol{Q})=\KK_k$.
Now define $\boldsymbol{U}\in\R^{k\times k}$ by
\[
\boldsymbol{U}:=\boldsymbol{Q}^T \boldsymbol{M} \boldsymbol{Q}.
\]
If $k=N$ then $\boldsymbol{Q}\boldsymbol{Q}^T=\boldsymbol{I}_N$ and $\boldsymbol{Q}\boldsymbol{U}\boldsymbol{Q}^T = \boldsymbol{M}$, from
which it follows that
\begin{equation}\label{eq:ideal}
\boldsymbol{M}^{1/2}\bz =\boldsymbol{Q}\boldsymbol{U}^{1/2}\boldsymbol{Q}^T \bz.
\end{equation}
The algorithm relies on explicit matrix multiplication to construct
$\boldsymbol{U}$ and then a direct factorization of $\boldsymbol{U}$, thus for large $N$ it
is feasible only when $k\ll N$, in which case \eqref{eq:ideal} does not hold
exactly. However, as we show later it may hold to a good enough
approximation. The following Lanczos type algorithm builds up progressively the columns of
$\boldsymbol{Q}$ without fully computing $\KK_k$ first.

\begin{remark}\label{rem:qr}
In the following, we make frequent use of the $QR$-factorisation of matrices and therefore recall the most important facts: For a matrix $\boldsymbol{A}\in\R^{n\times k}$
with $k\leq n\in \N$, there
exists a
$QR$-factorization $\boldsymbol{A}=\boldsymbol{Q}\boldsymbol{R}$ such that $\boldsymbol{Q}\in\R^{n\times k}$ and $\boldsymbol{R}\in\R^{k\times k}$. 
The columns of $\boldsymbol{Q}$ are orthonormal and for $1\leq j\leq {\rm rank}(\boldsymbol{A})$, 
the first $j$
columns of $\boldsymbol{Q}$ span the same linear space as the first $j$ columns of $\boldsymbol{A}$. Moreover, $\boldsymbol{R}$ is upper triangular. If we restrict to positive diagonal entries of $\boldsymbol{R}$,
the factorization is unique if ${\rm rank}(\boldsymbol{A})=k$.
\end{remark}

\begin{algorithm}\label{alg1}
	\textbf{ Input:} positive definite symmetric matrix $\boldsymbol{M}\in\R^{N\times N}$, vector $\bz\in\R^N$, and
	maximal number of iterations $k\in\N$.
	\begin{enumerate}
		\item Compute Krylov subspace: Set $\boldsymbol{Q}_1:=\bz/|\bz|\in\R^{N\times 1}$ and $k_0=k$. For $j=2,\ldots,k$ do:
		\begin{enumerate}
			\item Compute $\widetilde{\bq}:= \boldsymbol{M}\bq^{j-1}\in\R^N$, where $\bq^{j-1}$ is the $(j-1)$-th column of $\boldsymbol{Q}_{j-1}\in\R^{N\times (j-1)}$.
			\item Compute $QR$-factorization $\boldsymbol{Q}_j\in\R^{N\times j}$ (with orthonormal columns), $\boldsymbol{R}_j\in\R^{j\times j}$ (upper triangular) such that $\boldsymbol{Q}_j\boldsymbol{R}_j=(\boldsymbol{Q}_{j-1},\widetilde{\bq})\in\R^{N\times j}$.
			\item If $(\boldsymbol{R}_{j})_{jj}=0$, set $k_0=j-1$ and goto Step 2. 
		\end{enumerate}
		\item Compute $\boldsymbol{U}_{k_0}:= \boldsymbol{Q}_{k_0}^T \boldsymbol{M} \boldsymbol{Q}_{k_0}\in \R^{{k_0}\times {k_0}}$.
		\item Compute $\boldsymbol{U}_{k_0}^{1/2}$ directly.
		\item Return $\by=\boldsymbol{Q}_{k_0} \boldsymbol{U}_{k_0}^{1/2} \boldsymbol{Q}_{k_0}^T \bz$.
	\end{enumerate}
	\textbf{ Output:} Approximation $\by\approx \boldsymbol{M}^{1/2}\bz$ and number of steps $k_0$.
\end{algorithm}

\begin{remark}
	Obviously, the orthogonal basis $\bq^1,\ldots,\bq^k$ could also be generated by Gram-Schmidt orthogonalization. However, numerical experiments show that this
	is not stable with respect to roundoff errors. Moreover, also the classical Lanczos algorithm seems to be prone to rounding errors, especially for ill-conditioned matrices.
	Therefore, we propose to use the $QR$-factorization as above. 
\end{remark}
\begin{remark}
As proved in Lemma~\ref{lem:qq} below (and as is easily verified), a generic $QR$-algorithm produces $\boldsymbol{Q}_j$ which coincides with the first $j$ columns of $\boldsymbol{Q}$ up to signs.
For simplicity, we assume in the following that the $QR$-algorithm ensures that the diagonal entries of $\boldsymbol{R}_j$ are always non-negative.
This guarantees that the first $j$ columns of $\boldsymbol{Q}_{j+1}$ coincide with $\boldsymbol{Q}_j$.
Thus, it suffices
to store only the new column $\bq^j$.
\end{remark}

\begin{theorem}\label{thm:sampleerror2}
	Let $0<\eta<4c_2$ and let $p$ be sufficiently large such that $\boldsymbol{C}_p$ constructed from $\boldsymbol{C}$ as in Section~\ref{sec:h2} is positive definite (condition~\eqref{eq:ppos} is sufficient),
	and suppose Assumption~\ref{ass:approxu} holds.
	Given $\bz\in\R^{N}$, call Algorithm~\ref{alg1} with $\boldsymbol{M}=\boldsymbol{C}_p$, $\bz$, and a maximal number of iterations $k\in\N$.
	The output of Algorithm~\ref{alg1} contains the approximation $\ZZ_{k,p}(\bz):=\by\in \R^{N}$ to $\boldsymbol{C}^{1/2}\bz$ and the step number $k_0\leq k$.
	\begin{itemize}
	 \item[(i)] There holds with Kronecker's delta $\delta_{i,j}$
	 \begin{align*}
	  \frac{|\boldsymbol{C}^{1/2} \bz -\ZZ_{k,p}(\bz)|}{|\bz|}
	  \leq  \delta_{k_0,k} {\sqrt{2\norm{\boldsymbol{M}}{2}}}\frac{4r^{{2}}}{r-1}r^{-k}+\frac{2C_{\rm H}N (\log(p)+1)^{2d-1}\Big(\frac{\eta}{4c_2}\Big)^{p}}{\max\{\lambda_{\rm min}(\boldsymbol{C}),\lambda_{\rm min}(\boldsymbol{C}_p)\}^{1/2}},
	 \end{align*} 
	 where $C_{\rm H}$, $\eta$, $c_2$, and $p$ are as in Proposition~\ref{prop:h2}, and
	  \begin{align*}
 r:= \frac{\lambda_{\rm max}(\boldsymbol{C}_p)+\lambda_{\rm min}(\boldsymbol{C}_p)}{\lambda_{\rm max}(\boldsymbol{C}_p)-\lambda_{\rm min}(\boldsymbol{C}_p)}>1.
\end{align*}
	 \item[(ii)] Let $\lambda_{\rm max}(\boldsymbol{C}_p)=\lambda_1> \lambda_2> \ldots> \lambda_M>0$ denote the distinct eigenvalues of $\boldsymbol{C}_p$ for some $M\leq N$ and assume 
	\begin{align*}
	 |\lambda_i-\lambda_j|\leq \lambda_{\rm max}(\boldsymbol{C}_p)C_\kappa \kappa^{\min\{i,j\}} \quad\text{for all }1\leq i,j\leq M
	\end{align*}
	for some $C_\kappa>0$ and $0<\kappa<1$, then
	\begin{align*}
	\frac{|\boldsymbol{C}^{1/2} \bz -\ZZ_{k,p}(\bz)|}{|\bz|}\leq \delta_{k_0,k}3\sqrt{\lambda_{\rm max}(\boldsymbol{C}_p)C_\kappa}\;\kappa^{k/4}+
	3\sqrt{2C_{\rm H}N} (\log(p)+1)^{d-1/2}\Big(\frac{\eta}{4c_2}\Big)^{p/2}.
	\end{align*}
	      \end{itemize}
	The algorithm completes in $\mathcal{O}(k^3p^{2d}N)$ arithmetic operations and uses less than $\mathcal{O}(kN)$ storage.
\end{theorem}
\begin{remark}
 The theorem covers two regimes of covariance matrices. Whereas case~(i) is the classical Lanczos convergence analysis for well-conditioned matrices, case~(ii) 
 considers ill-conditioned matrices with rapidly decaying
 eigenvalues. The numerical examples in Section~\ref{sec:numerics} suggest that the error estimates might be more or less sharp, since Algorithm~\ref{alg1} performs remarkably well for
 smooth random fields (with rapidly decaying eigenvalues) and very rough random fields (with well-conditioned covariance matrices).
 Note that $k_0<k$ (hence $\delta_{k_0,k}=0$) implies that the condition in the if-clause~1(c) is true. This however is an exotic case, meaning that {$\bz$ lies some non-trivial invariant subspace} of
 $\boldsymbol{C}_p$ with fewer
 than $k$ dimensions.
 In this situation the algorithm computes $\boldsymbol{C}_p^{1/2}\bz$ exactly and only the $H$-matrix approximation error remains.
 We note that {by use of~\eqref{eq:sqstandard} instead of~\eqref{eq:sqimproved} in the proof below,}
 it is possible to replace $(k+1)/4$ by $(k+1)/2$ and $p/2$ by $p$ in the exponents in~(ii) at the price of including the square-root of the minimal eigenvalue in the denominator as in~(i).
 
\end{remark}

\begin{proof}[Proof of Theorem~\ref{thm:sampleerror2}]
The cost estimate is proved as follows. The Krylov subspace loop of Algorithm~\ref{alg1} completes at most $k$ iterations. In each iteration, we have one $H^2$-matrix-vector multiplication which needs
	$\mathcal{O}(p^{2d}N)$ operations. Moreover, the $QR$-factorization needs $\mathcal{O}(Nk^2)$ arithmetic operations. After the matrix $\boldsymbol{Q}_k$ is set up,
	we have $k$ $H^2$-matrix-vector multiplications to compute $\boldsymbol{M}\boldsymbol{Q}_k$ and $k^2$ scalar products to compute $\boldsymbol{U}_{k_0}$. In total, this needs $\mathcal{O}(N(k+k^2))$ arithmetic operations.
	The computation of $\boldsymbol{U}_{k_0}^{1/2}$ can be done in $\mathcal{O}(k^3)$ operations (see, e.g.,~\cite{sqrtm} for the algorithm and the corresponding analysis). Finally, to compute $\by$, we have $k$ scalar products, 
	a matrix vector multiplication with a $(k\times k)$ matrix and a matrix-matrix multiplication of $(N\times k)$ and $(k\times k)$ matrices, all of which can be done in $\mathcal{O}(Nk^2)$ arithmetic operations.
	
	To see~(i), we employ the triangle inequality 
	\begin{align}\label{eq:split}
	\begin{split}
	\frac{|\boldsymbol{C}^{1/2} \bz -\ZZ_{k,p}(\bz)|}{|\bz|}&\leq \frac{|\boldsymbol{C}_p^{1/2} \bz -\ZZ_{k,p}(\bz)|}{|\bz|}+\frac{|\boldsymbol{C}^{1/2}\bz-\boldsymbol{C}_p^{1/2} \bz|}{|\bz|}\\
	&\leq \frac{|\boldsymbol{C}_p^{1/2} \bz -\ZZ_{k,p}(\bz)|}{|\bz|}
	+\norm{\boldsymbol{C}_p^{1/2}-\boldsymbol{C}^{1/2}}{2}.
	\end{split}
	\end{align}
	For the first term on the right-hand side, Lemma~\ref{lem:qr1.5} below proves
	\begin{align*}
	 \frac{|\boldsymbol{C}_p^{1/2} \bz -\ZZ_{k,p}(\bz)|}{|\bz|}&
	  \leq \delta_{k_0,k}
	  {\sqrt{2\norm{\boldsymbol{M}}{2}}}\frac{4r^{{2}}}{r-1}r^{-k}. 
	\end{align*}
	As shown in~\eqref{eq:sqstandard} of Lemma~\ref{lem:sqrtm} below, the second term on the right-hand side of~\eqref{eq:split} is bounded by
	\begin{align}\label{eq:triangle}
	\norm{\boldsymbol{C}_p^{1/2}-\boldsymbol{C}^{1/2}}{2}\leq 2\max\{\lambda_{\rm min}(\boldsymbol{C}),\lambda_{\rm min}(\boldsymbol{C}_p)\}^{-1/2}\norm{\boldsymbol{C}_p-\boldsymbol{C}}{2}.
	\end{align}
	Hence,~(i) follows from Proposition~\ref{prop:h2}. 
	For~(ii), we note that the combination of both estimates in Proposition~\ref{prop:alg2} below
	shows for $\boldsymbol{U}_j:=\boldsymbol{Q}_j^T\boldsymbol{M}\boldsymbol{Q}_j$
	\begin{align*}
	\min_{1\leq j\leq k}\frac{|\boldsymbol{C}_p^{1/2} \bz -\boldsymbol{Q}_j(\boldsymbol{U}_j^{1/2})\boldsymbol{Q}_j^T\bz|}{|\bz|}\leq \delta_{k_0,k} 
	3\sqrt{\lambda_{\rm max}(\boldsymbol{C}_p)C_\kappa}\;\kappa^{k/4}.	
	\end{align*}%
	We may eliminate the minimum in the error estimate since Algorithm~\ref{alg1} is essentially (up to roundoff errors) of Lanczos type, and for this algorithm,~\cite[Example~5.1]{monotone} shows that the approximation error
	$|\boldsymbol{C}_p^{1/2} \bz -\boldsymbol{Q}_j(\boldsymbol{U}_j^{1/2})\boldsymbol{Q}_j^T\bz|$ decreases monotonically in $j$. 
	Since $\boldsymbol{Q}_{k_0}(\boldsymbol{U}_{k_0}^{1/2})\boldsymbol{Q}_{k_0}^T\bz=\ZZ_{k,p}(\bz)$, the remainder of the proof then follows as for~(i) but we use~\eqref{eq:sqimproved}
	instead of~\eqref{eq:sqstandard} of Lemma~\ref{lem:sqrtm} below.
\end{proof}

\subsection{Computing the square-root (Method 2)}
The main drawback of Algorithm~\ref{alg1} is the additional storage requirements due to the necessity to store the matrix $\boldsymbol{Q}_k$. 
For this reason, we here follow a different approach, proposing a second algorithm that improves this situation.

The matrix sign function is defined for all square matrices $\widetilde{\boldsymbol{M}}$ with no pure imaginary eigenvalues as
\begin{align*}
 {\rm sgn}(\widetilde{\boldsymbol{M}}):=\widetilde{\boldsymbol{M}}(\widetilde{\boldsymbol{M}}^2)^{-1/2}.
\end{align*}
The sign function ${\rm sgn}(\widetilde{\boldsymbol{M}})$ can be computed using the Schultz iteration via
\begin{align}\label{eq:schultz}
 \boldsymbol{M}_{k+1} = \frac12 \boldsymbol{M}_k(3\boldsymbol{I}-\boldsymbol{M}_k^2),\quad \boldsymbol{M}_0=\widetilde{\boldsymbol{M}}.
\end{align}
The iterates $\boldsymbol{M}_k$ converge quadratically towards ${\rm sgn}(\widetilde{\boldsymbol{M}})$ if $\norm{\boldsymbol{I}-\widetilde{\boldsymbol{M}}^2}{2}<1$ in any matrix norm  
(see~\cite[Theorem~5.2]{sqit2}).
It is observed in~\cite{sqit1}, that all matrices $\boldsymbol{M}\in \R^{N\times N}$ with only positive real 
eigenvalues satisfy
\begin{align*}
{\rm sgn}\begin{pmatrix} 0 & \boldsymbol{M}\\ \boldsymbol{I} &0 \end{pmatrix} = 
\begin{pmatrix} 0 & \boldsymbol{M}^{1/2}\\ \boldsymbol{M}^{-1/2} & 0 \end{pmatrix},
\end{align*}
where $\boldsymbol{I}\in\R^{N\times N}$ denotes the identity matrix, which opens the possibility to compute $\boldsymbol{M}^{1/2}$ via the sign function of the matrix
By inserting 
\begin{align*}
 \widetilde{\boldsymbol{M}}:=\begin{pmatrix} 0 & \boldsymbol{M}\\ \boldsymbol{I} &0 \end{pmatrix}.
\end{align*}
By inserting this choice of $\widetilde{\boldsymbol{M}}$ into~\eqref{eq:schultz}, we see that all iterates have the form
\begin{align*}
 \boldsymbol{M}_k:=\begin{pmatrix} 0 & \boldsymbol{A}_k\\ \boldsymbol{B}_k & 0 \end{pmatrix}.
\end{align*}
As already observed in~\cite{sqit1}, this leads to the iteration
\begin{align}\label{eq:it2}
 \boldsymbol{A}_{k+1} = \frac12 \boldsymbol{A}_k(3\boldsymbol{I}-\boldsymbol{B}_k\boldsymbol{A}_k),\quad \boldsymbol{B}_{k+1}=\frac12\boldsymbol{B}_k(3\boldsymbol{I}-\boldsymbol{A}_k\boldsymbol{B}_k),
\end{align}
starting with $\boldsymbol{A}_0=\boldsymbol{M}$ and $\boldsymbol{B}_0=\boldsymbol{I}\in \R^{N\times N}$. The iterates $\boldsymbol{A}_k$ converge towards $\boldsymbol{M}^{1/2}$, which is what we aim to compute.
The considerations above lead us to the following recursive form of the Schulz 
algorithm above, which uses only matrix vector multiplication. {The subroutines \texttt{PartA} and \texttt{PartB} compute $\boldsymbol{A}_k\bz$ and $\boldsymbol{B}_k\bz$ respectively.}

\begin{algorithm}\label{alg2}
\textbf{ Input:} positive definite symmetric matrix $\boldsymbol{M}\in\R^{N\times N}$, vector $\bz\in\R^N$, 
maximal number of iterations $k\in\N$, temporary storage vectors $\bz^j\in \R^N$, $j\in\{1,\ldots,k\}$, and scaling factor $0<s<2\norm{\boldsymbol{C}_p}{2}^{-1}$ (the scaling factor ensures convergence of the algorithm).\\
\textbf{ Main:}
\begin{enumerate}
\item Compute $\by=\texttt{PartA}(s\boldsymbol{M},\bz,(\bz^j)_{j=1}^k,k)$.
\item Return $\by/\sqrt{s}$.
\end{enumerate}
\textbf{ Output:} the approximation $\by\approx \boldsymbol{M}^{1/2}\bz$.\\
\textbf{ Subroutines:}\\
$\texttt{PartA}(\boldsymbol{M},\bz,(\bz^j),k)$:
\begin{itemize}
 \item[(i)] If $k=0$, return $\boldsymbol{M}\bz$.
 \item[(ii)] Compute $\bz^k:= \texttt{PartA}(\boldsymbol{M},\bz,(\bz^j)_{j=1}^{k-1},k-1)$ and $\bz^k:= 
\texttt{PartB}(\boldsymbol{M},\bz^k,(\bz^j)_{j=1}^{k-1},k-1)$.
 \item[(iii)] Compute $\bz:=3\bz-\bz^k$.
 \item[(iv)] Return $\frac12\texttt{PartA}(\boldsymbol{M},\bz,(\bz^j)_{j=1}^{k-1},k-1)$.
\end{itemize}
$\texttt{PartB}(\boldsymbol{M},\bz,(\bz^j),k)$:
\begin{itemize} 
 \item[(i)] If $k=0$, return $\bz$.
 \item[(ii)] Compute $\bz^k:= \texttt{PartB}(\boldsymbol{M},\bz,(\bz^j)_{j=1}^{k-1},k-1)$ and $\bz^k:= 
 \texttt{PartA}(\boldsymbol{M},\bz^k,(\bz^j)_{j=1}^{k-1},k-1)$.
 \item[(iii)] Compute $\bz:=3\bz-\bz^k$.
 \item[(iv)] Return $\frac12\texttt{PartB}(\boldsymbol{M},\bz,(\bz^j)_{j=1}^{k-1},k-1)$.
\end{itemize}

\end{algorithm}

\begin{remark}
The extra storage vectors $(\bz^j)_{j=1}^k$ are needed to avoid allocation of a new temporary storage vector in each call of either $\texttt{PartA}$ are $\texttt{PartB}$.
This would result in $\mathcal{O}(3^k)$ additional allocations. By supplying the additional storage vectors, we can exploit the fact that each level of recursion can share a single
storage vector.
\end{remark}

\begin{theorem}\label{thm:sampleerror}
	Suppose Assumption~\ref{ass:approxu} holds and and let $\bz\in\R^{N}$. 
	If $0<\eta<4c_2$ and $p$ is sufficiently large such that $\boldsymbol{C}_p$ constructed from $\boldsymbol{C}$ as in Section~\ref{sec:h2} is positive definite (condition~\eqref{eq:ppos} is sufficient),
	Algorithm~\ref{alg1} called with $\boldsymbol{M}=\boldsymbol{C}_p$ and $0<s<2\norm{\boldsymbol{C}_p}{2}^{-1}$ computes
	the approximation $\ZZ_{k,p}(\bz):=\by\in \R^{N}$ such that 
	 \begin{align*}
	  \frac{|\boldsymbol{C}^{1/2} \bz -\ZZ_{k,p}(\bz)|}{|\bz|}&
	  \leq s^{-1/2}\kappa^{2^k}+\frac{2C_{\rm H}N (\log(p)+1)^{2d-1}\Big(\frac{\eta}{4c_2}\Big)^{p}}{\max\{\lambda_{\rm min}(\boldsymbol{C}),\lambda_{\rm min}(\boldsymbol{C}_p)\}^{1/2}},
	 \end{align*}
	 where $\kappa:=\max\{|1-s\lambda_{\rm max}(\boldsymbol{C}_p)|,|1-s\lambda_{\rm min}(\boldsymbol{C}_p)|\}<1$.
	The algorithm completes in $\mathcal{O}(3^kp^{2d}N)$ arithmetic operations and uses less than $kN$ extra storage.
	The constant $C_{\rm H}$ is defined in Proposition~\ref{prop:h2}.
\end{theorem}
\begin{remark}
 In contrast to Algorithm~\ref{alg1} which needs $\mathcal{O}(|\log_{\kappa}(\eps)|N)$ extra storage (at least in case~(ii)), we see that 
 Algorithm~\ref{alg2} requires only $\mathcal{O}(\log|\log(\eps)|N)$ additional storage for an error request of $\eps>0$.
 \end{remark}

\begin{proof}[Proof of Theorem~\ref{thm:sampleerror}]

First, we prove that \texttt{PartA} and \texttt{PartB} from Algorithm~\ref{alg2} correctly compute $\boldsymbol{A}_k\bz$ and $\boldsymbol{B}_k\bz$ from~\eqref{eq:it2}.
This is done by induction on $k$. First, for $k=0$, the output of \texttt{PartA} is obviously $\boldsymbol{M}\bz=\boldsymbol{A}_0\bz$ and the output of \texttt{PartB} is $\bz=\boldsymbol{B}_0\bz$. 
This confirms the case $k=0$. Assume that \texttt{PartA} and \texttt{PartB} work correctly for $k\in\N$. By substitution of $\texttt{PartA}(\boldsymbol{M},\bz,(\bz^j)_{j=1}^{k-1},k-1)=\boldsymbol{A}_{k-1}\bz$
and $\texttt{PartB}(\boldsymbol{M},\bz^k,(\bz^j)_{j=1}^{k-1},k-1)=\boldsymbol{B}_{k-1}\bz^k$ in \texttt{PartA}, the variable $\bz^k$ before step~(iii) is given by
$\bz^k=\boldsymbol{B}_{k-1}\boldsymbol{A}_{k-1}\bz$. Thus, step~(iii)--(iv) correctly compute $\tfrac12\boldsymbol{B}_{k-1}(3\bz-\boldsymbol{B}_{k-1}\boldsymbol{A}_{k-1}\bz)=\boldsymbol{A}_k\bz$.
During the execution of $\texttt{PartA}(\cdot,\cdot,\cdot,k)$, extra storage vector $\bz^k$ is not accessed by other instances of the subroutines (the function calls to $\texttt{PartA}(\cdot,\cdot,\cdot,k-1)$ and
$\texttt{PartB}(\cdot,\cdot,\cdot,k-1)$ access only $(\bz^k)_{j=1}^{k-1}$). This ensures that the correct value of $\bz^k$ is used at each point of the execution.
Analogously, we argue that $\texttt{PartB}$ works correctly and thus conclude the induction.

{For the computational cost estimate, we prove by induction that each subroutine $\texttt{PartA}(\cdot,\cdot,\cdot,k)$ and $\texttt{PartB}(\cdot,\cdot,\cdot,k)$
requires less than 
\begin{align}\label{eq:tmpcost}
 C(3^kp^{2d}N+2N\sum_{j=0}^{k-1}3^j)
\end{align}
operations for some universal constant $C\geq 1$ and all $k\in\N$.
For $k=0$, subroutine $\texttt{PartA}$ performs an $H^2$-matrix-vector multiplication which, according to Proposition~\ref{prop:h2}, 
 costs less than $\mathcal{O}(p^{2d}N)$.
 Subroutine $\texttt{PartB}$ just returns the vector $\bz$. This shows~\eqref{eq:tmpcost} for $k=0$ for both subroutines. Assume that~\eqref{eq:tmpcost} is correct for
 both subroutines for some $k>0$.
 The fact that each subroutine $\texttt{PartA}(\cdot,\cdot,\cdot,k+1)$ and $\texttt{PartB}(\cdot,\cdot,\cdot,k+1)$
 performs one scalar-vector multiplication and one vector addition as well as three calls to $\texttt{PartA}(\cdot,\cdot,\cdot,k)$ or $\texttt{PartB}(\cdot,\cdot,\cdot,k)$ shows that
 the cost of each subroutine  $\texttt{PartA}(\cdot,\cdot,\cdot,k+1)$ and $\texttt{PartB}(\cdot,\cdot,\cdot,k+1)$ is bounded by
 \begin{align*}
  3C(3^kp^{2d}N+2N\sum_{j=0}^{k-1}3^j) + 2N= C(3^{k+1}p^{2d}N+2N\sum_{j=1}^{k}3^j) + 2N\leq C(3^{k+1}p^{2d}N+2N\sum_{j=0}^{k}3^j).
 \end{align*}
This concludes the proof of~\eqref{eq:tmpcost}, which proves the cost estimate since
\begin{align*}
 C(3^kp^{2d}N+2N\sum_{j=0}^{k-1}3^j)\leq C3^kp^{2d}N + C3^kN\leq 2C3^kp^{2d}N.
\end{align*}} 
To see the error estimate, we use~\eqref{eq:split} and note that
	Algorithm~\ref{alg2} is nothing else than a recursive version of the iteration~\eqref{eq:it2}.  The scaling $s<2\norm{\boldsymbol{C}_p}{2}^{-1}$ ensures $\kappa<1$, since $\lambda\in\{\lambda_{\rm min}(\boldsymbol{C}_p),\lambda_{\rm max}(\boldsymbol{C}_p)\}$ satisfies  $1-s\lambda<1$ (since $s,\lambda>0$) as well as
 $s\lambda-1\leq s\norm{\boldsymbol{C}_p}{2}-1< 2-1= 1$. Thus, 
 Lemma~\ref{lem:it} shows
	\begin{align*}
	 \frac{|\boldsymbol{C}_p^{1/2} \bz -\ZZ_{k,p}(\bz)|}{|\bz|}&
	  \leq s^{-1/2}\big(\max\{|1-s\lambda_{\rm max}(\boldsymbol{C}_p)|,|1-s\lambda_{\rm min}(\boldsymbol{C}_p)|\}\big)^{2^k}=s^{-1/2}\kappa^{2^k}.
	 \end{align*}
	We conclude the proof with the aid of~\eqref{eq:triangle} and Proposition~\ref{prop:h2}.
\end{proof}

\section{Numerical experiments}\label{sec:numerics}
All numerical experiments where computed in Matlab, by use of a Matlab-$H^2$-matrix library which can be downloaded under \texttt{software.michaelfeischl.net}.
The authors are well aware that the Matlab implementation prohibits high-end performance. However, we wanted to demonstrate the feasibility of our algorithms and
show the correct convergence rates, for which purpose the Matlab implementation is sufficient.  

For the first example, we consider a covariance function of the form~\eqref{eq:nonstat} with 
\begin{align}\label{eq:excov}
 \boldsymbol{\Sigma}_\bx:=|\bx|^2\boldsymbol{I}\quad\text{and}\quad \boldsymbol{\Sigma}_\by:=|\by|^2\boldsymbol{I}.
\end{align}
We use Algorithm~\ref{alg1} to generate six samples on the unit square $D=[0,1]^2$ of the corresponding normal random field $\ZZ$ shown in Figure~\ref{fig:samplesnonstat}.
Figure~\ref{fig:samplesexp}--\ref{fig:samplesgauss} show samples of the covariance functions from~\eqref{eq:matern} with different parameters.
\begin{figure}
 \includegraphics[ clip, width=0.3\textwidth]{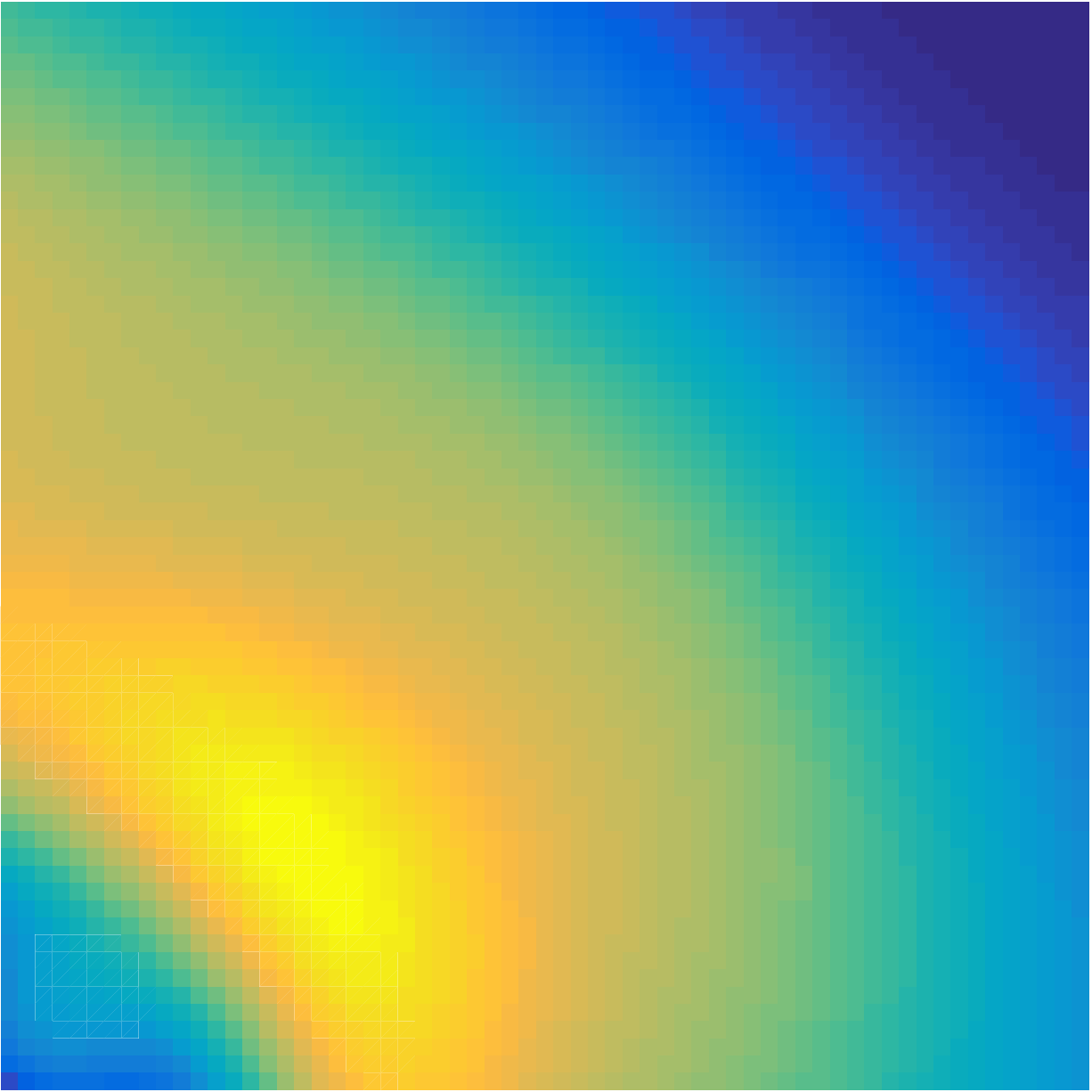}\hspace{5mm}%
\includegraphics[clip, width=0.3\textwidth]{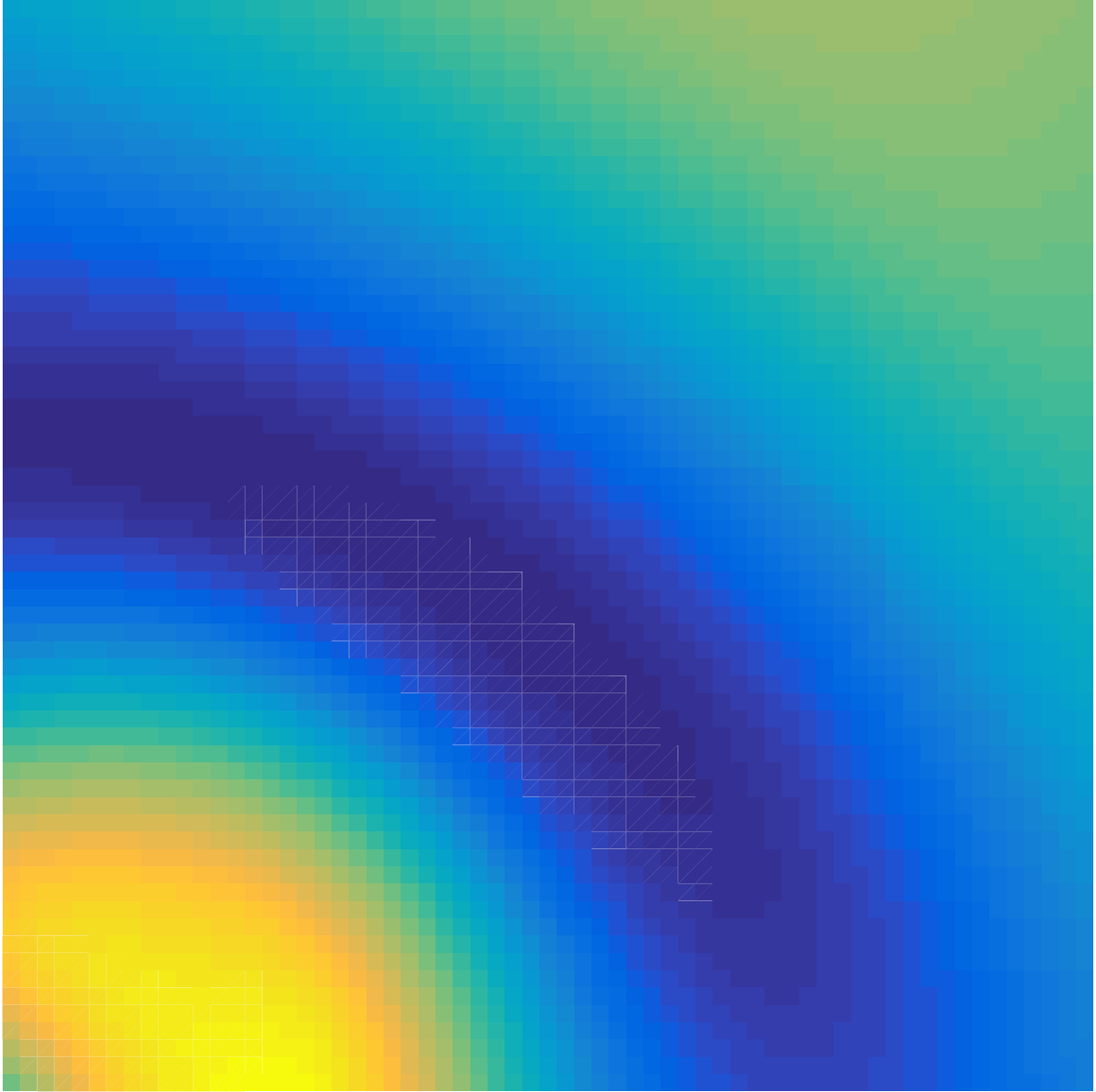}\hspace{5mm}%
\includegraphics[clip, width=0.3\textwidth]{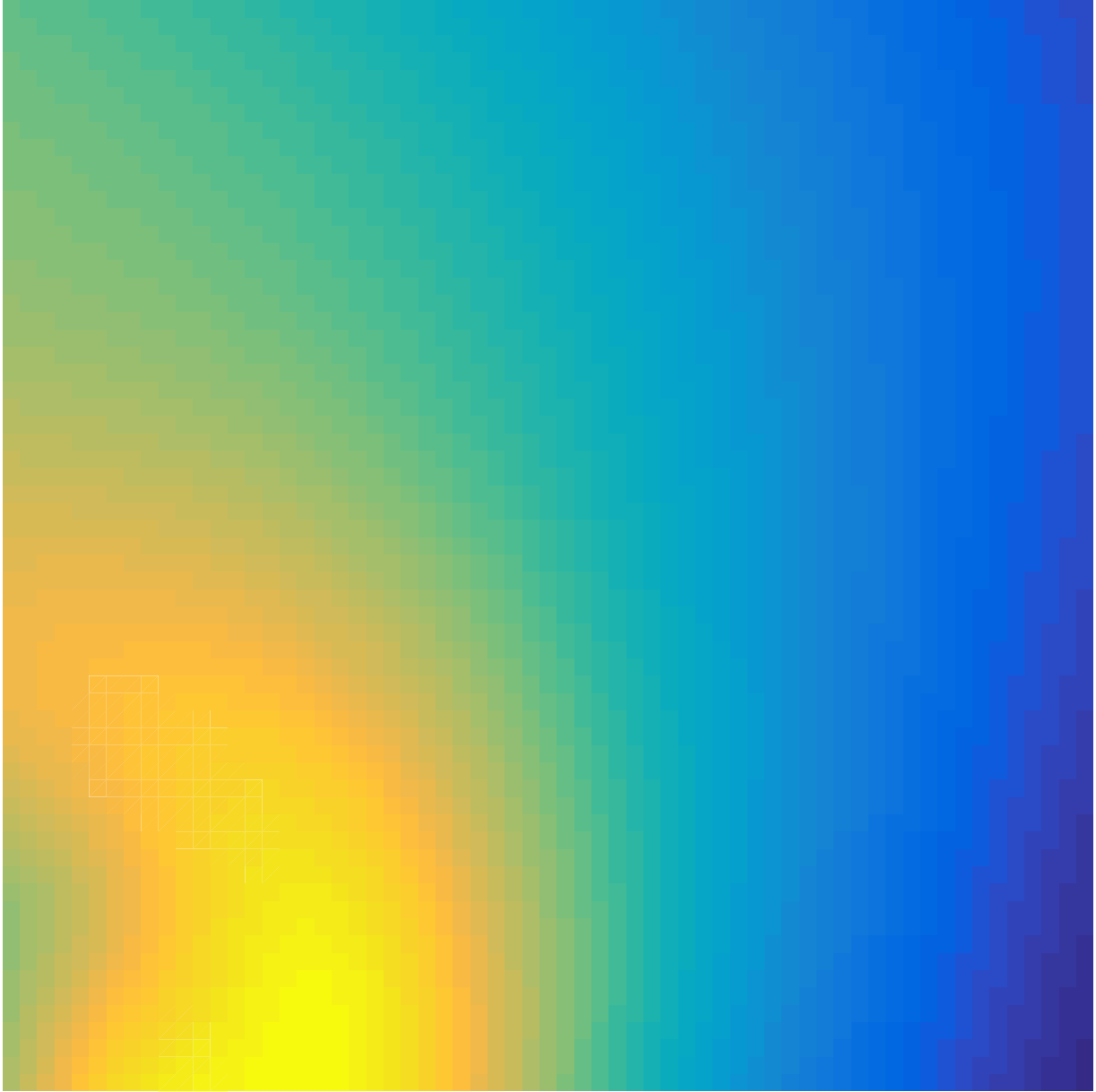}\\\vspace{5mm}%
\includegraphics[clip, width=0.3\textwidth]{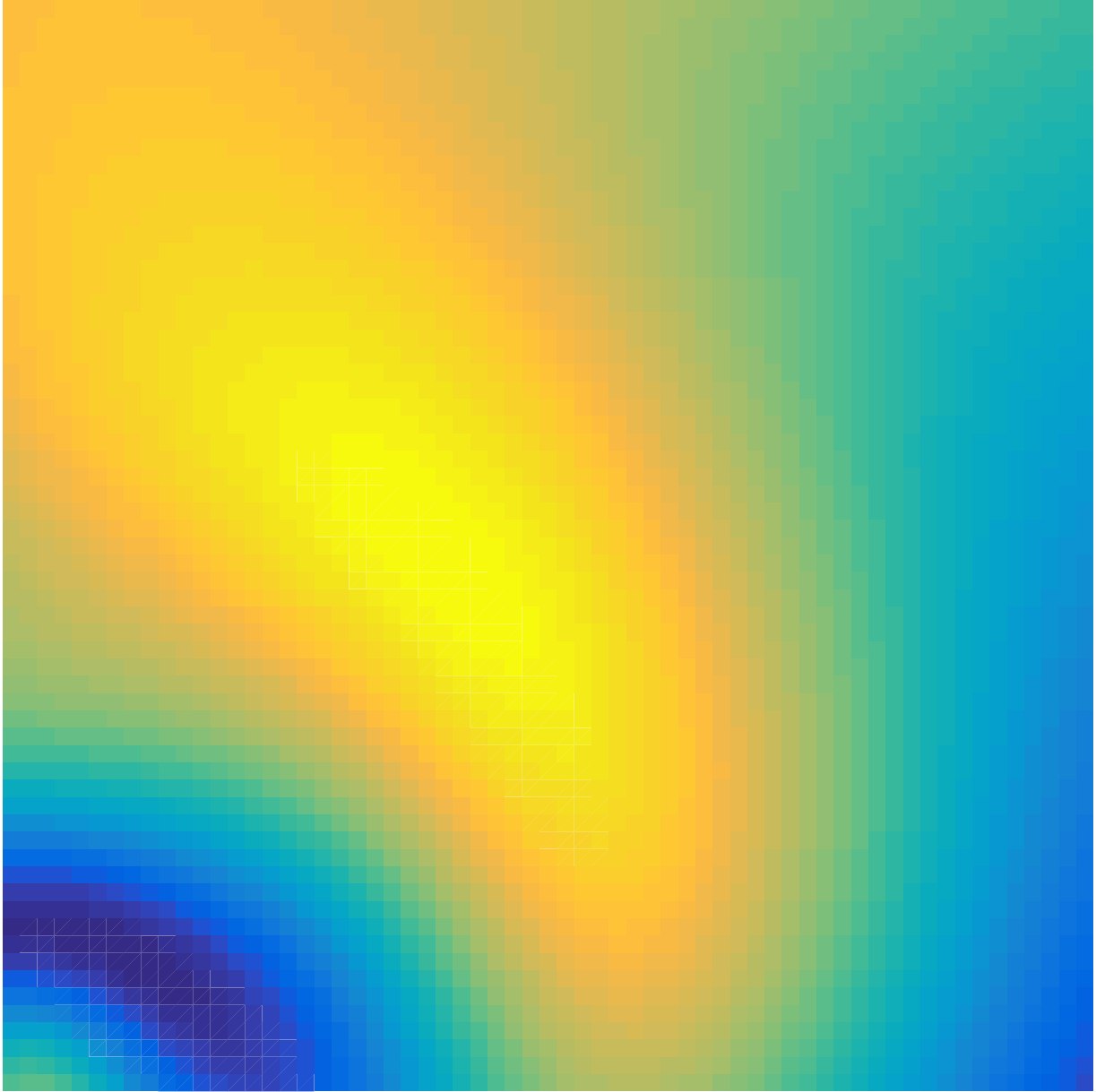}\hspace{5mm}%
\includegraphics[clip, width=0.3\textwidth]{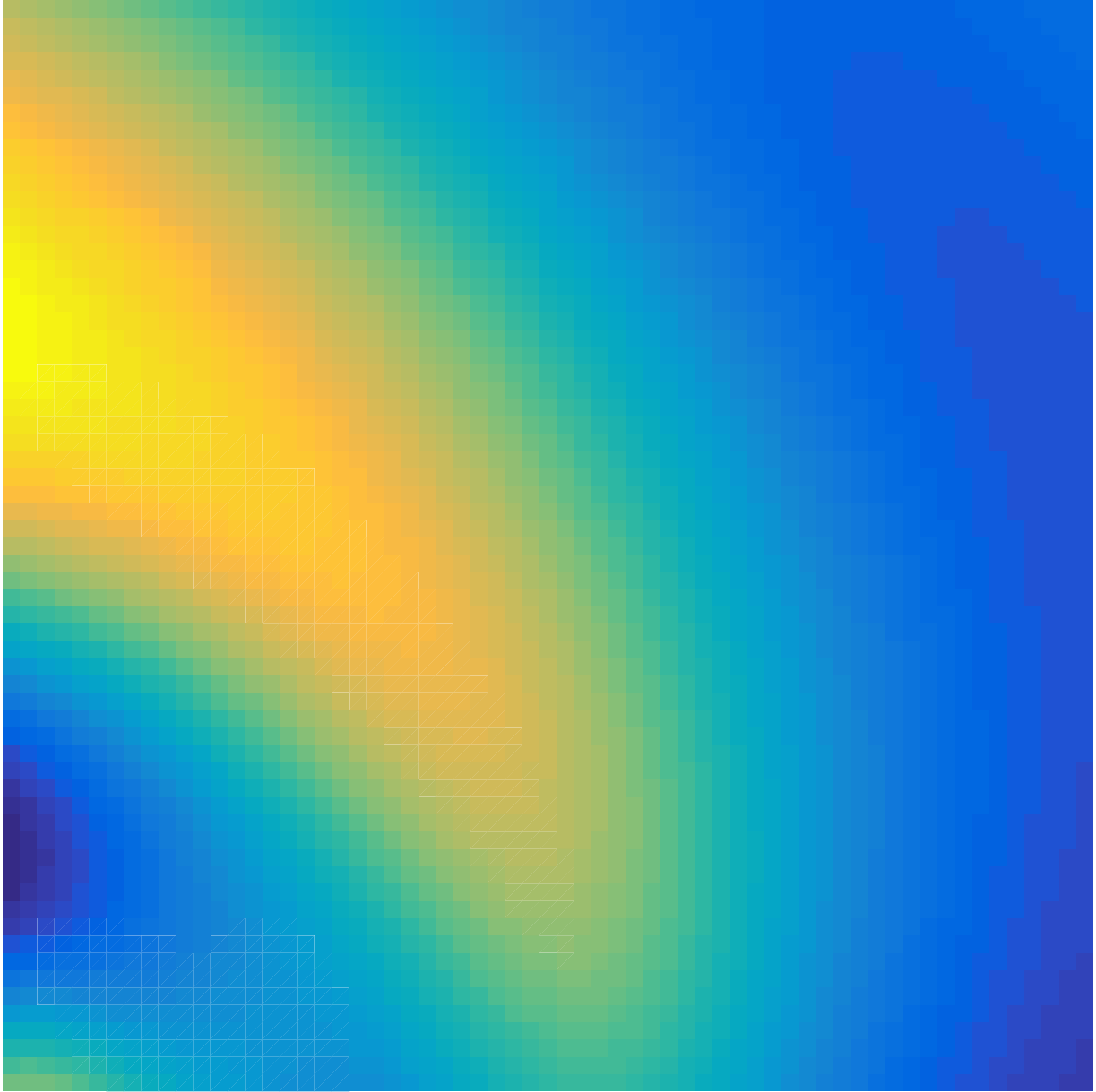}\hspace{5mm}%
\includegraphics[clip, width=0.3\textwidth]{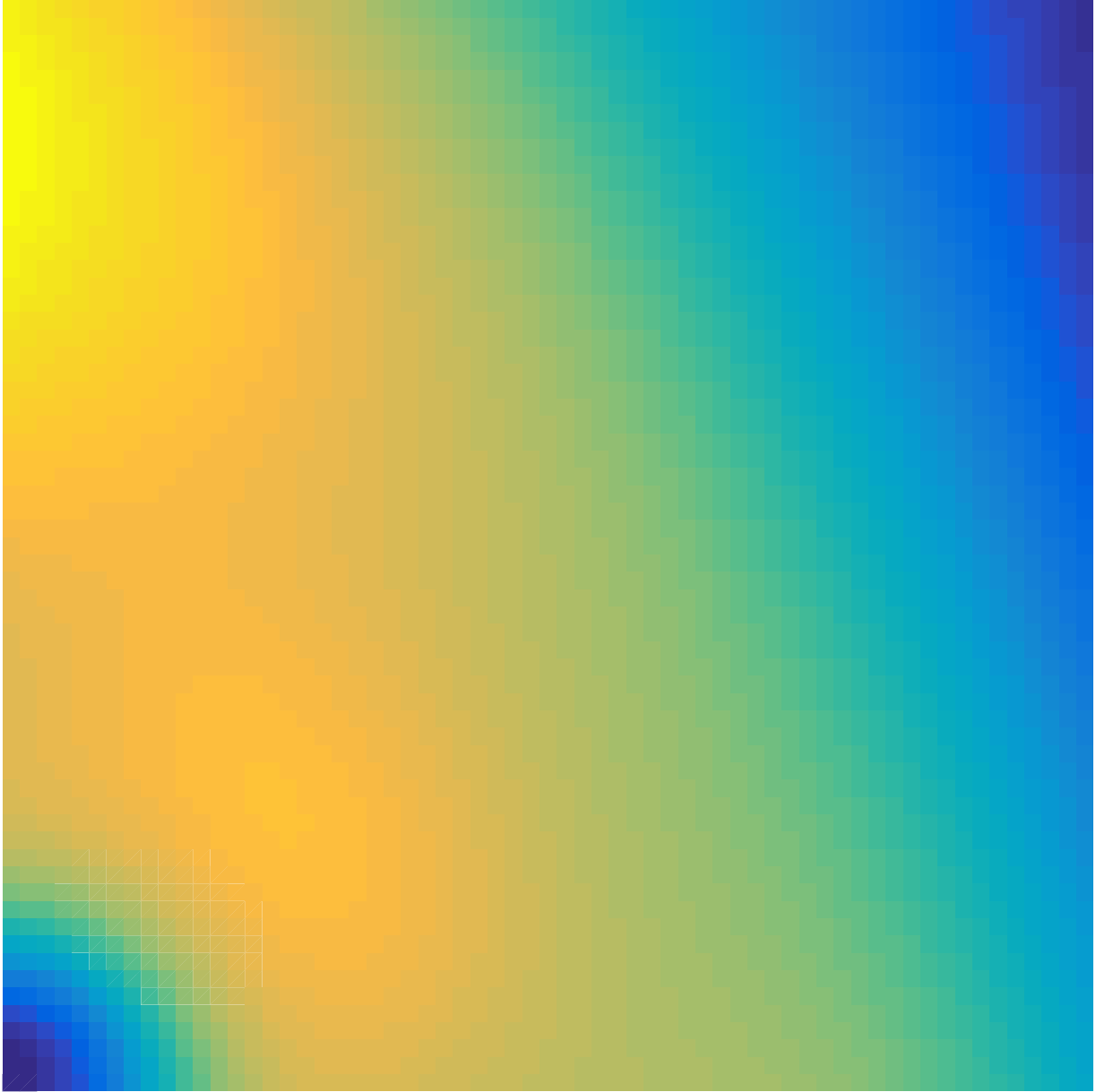}
\caption{Samples of $\ZZ$ with a non-stationary covariance function. We clearly observe the shorter covariance length (more variation) near the bottom left corner.}
\label{fig:samplesnonstat}
\end{figure}
\begin{figure}
\includegraphics[ clip, width=0.3\textwidth]{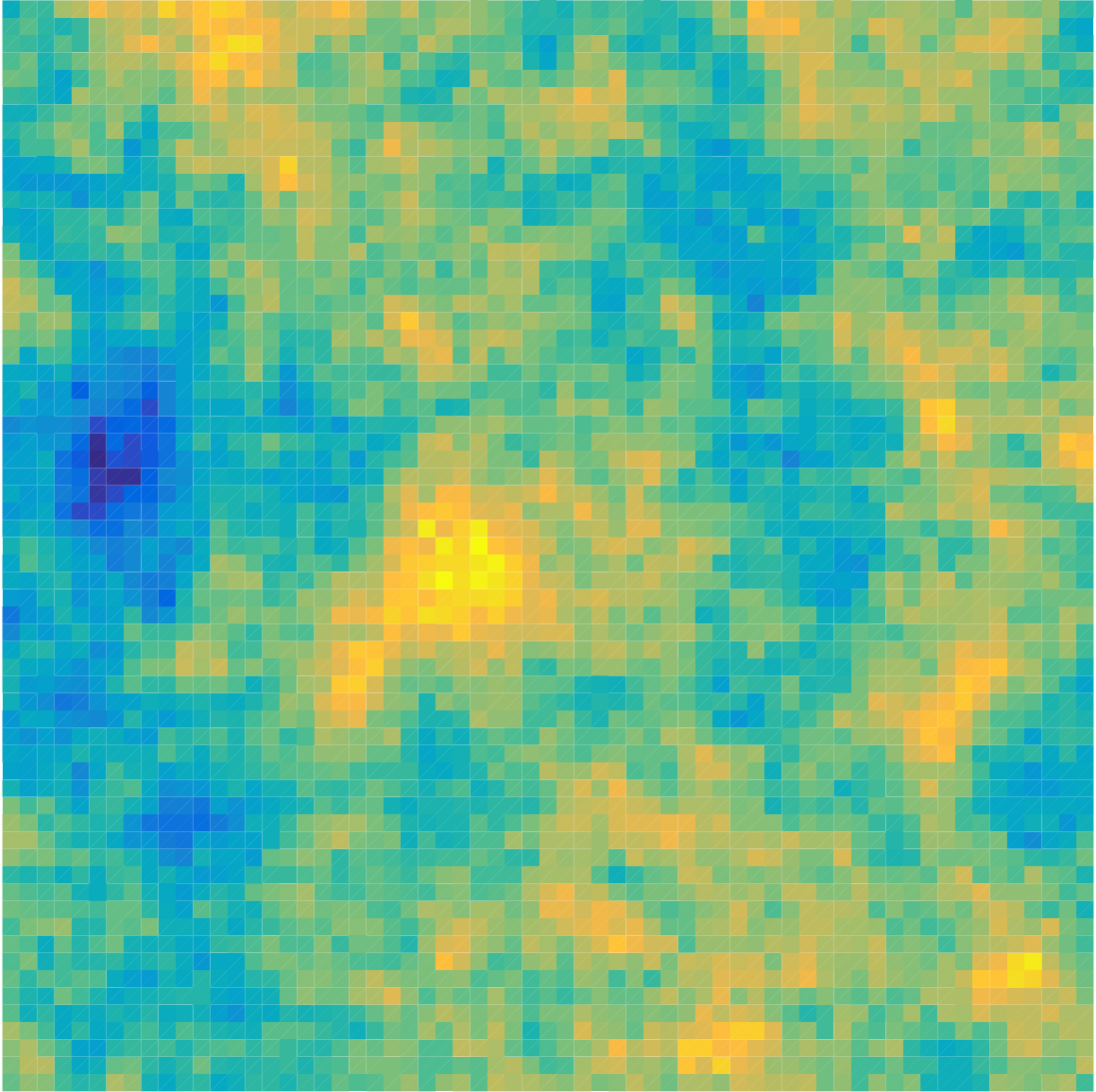}\hspace{5mm}%
\includegraphics[clip, width=0.3\textwidth]{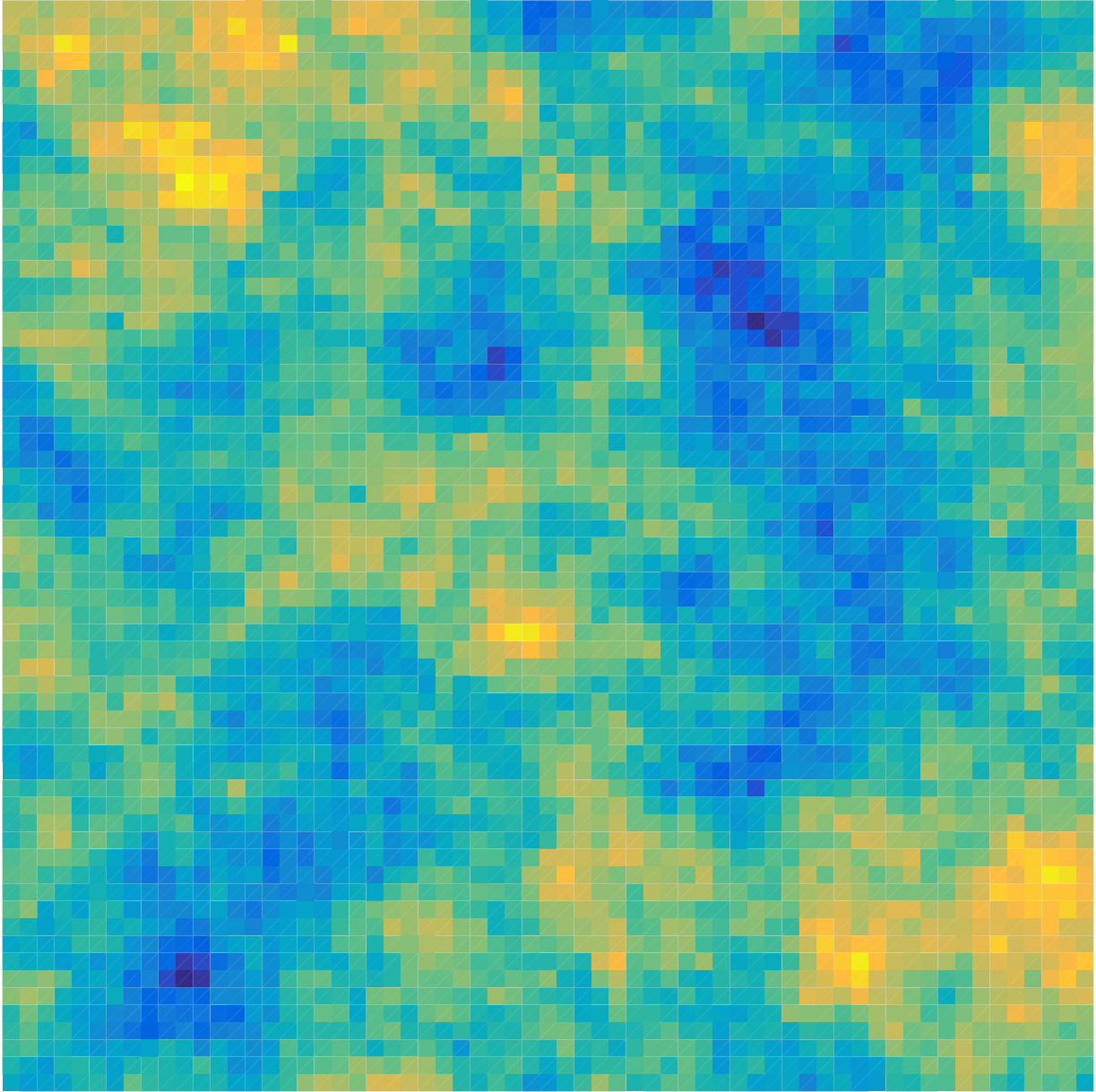}\hspace{5mm}%
\includegraphics[clip, width=0.3\textwidth]{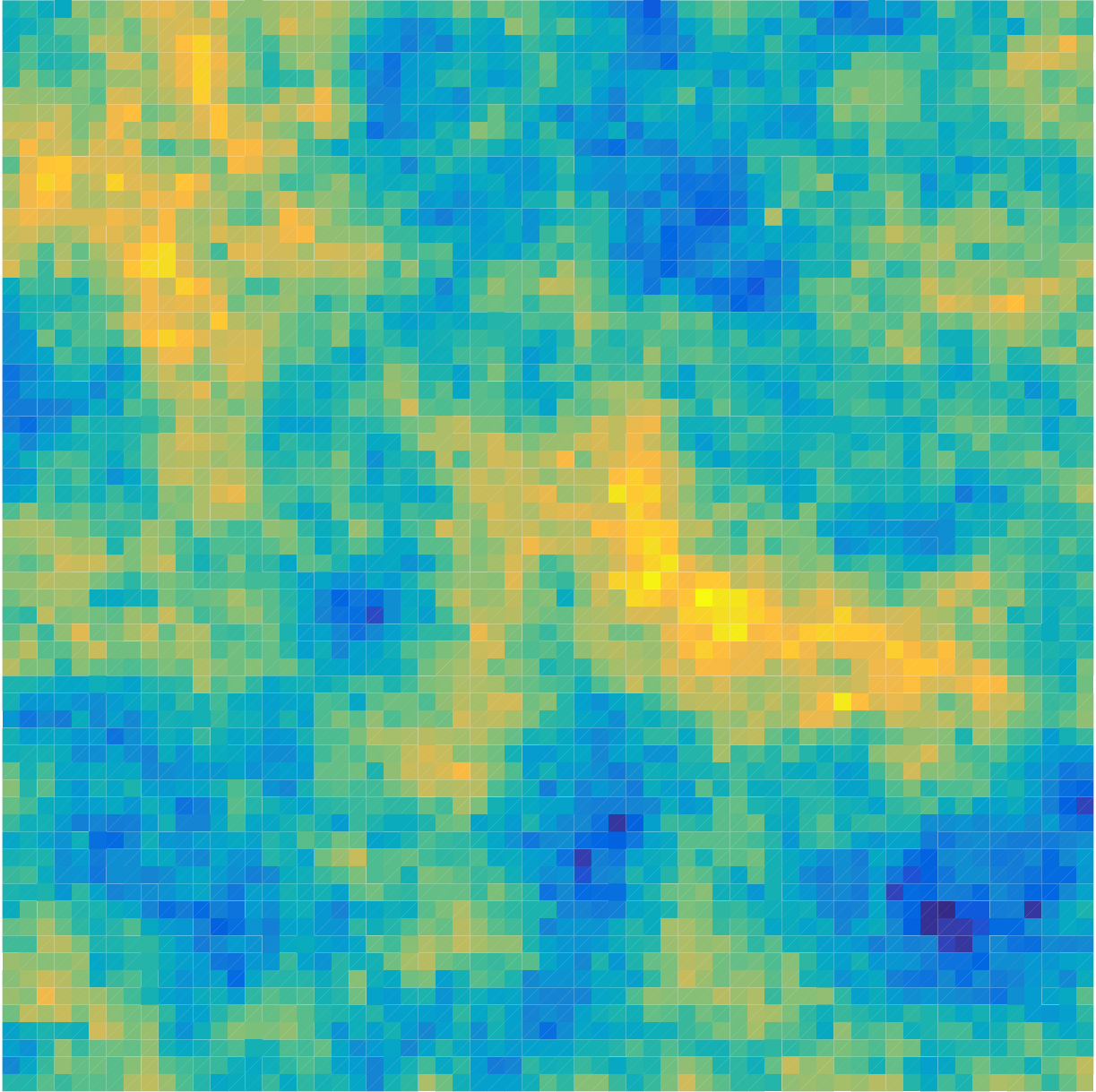}\\\vspace{5mm}%
\includegraphics[clip, width=0.3\textwidth]{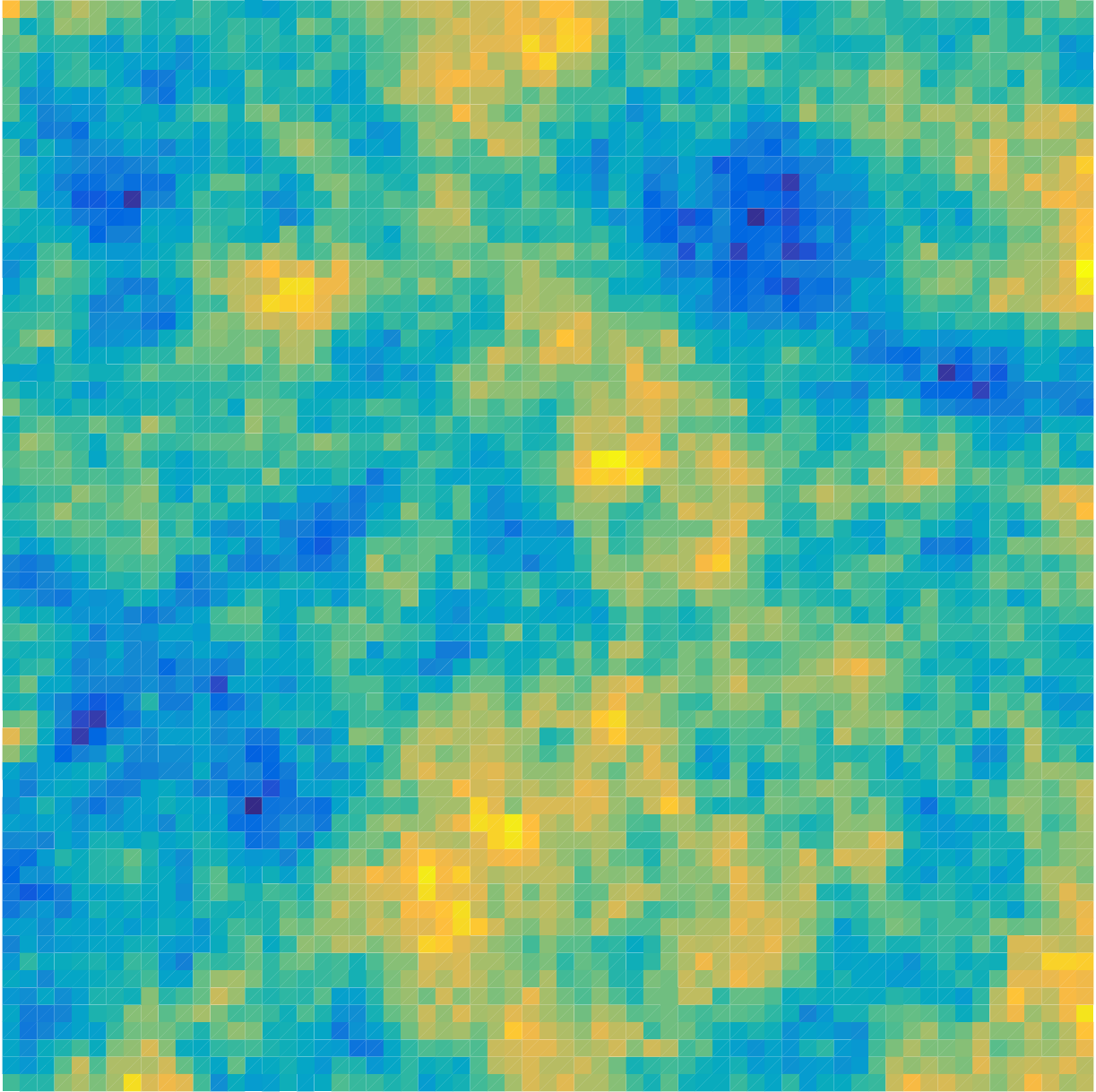}\hspace{5mm}%
\includegraphics[clip, width=0.3\textwidth]{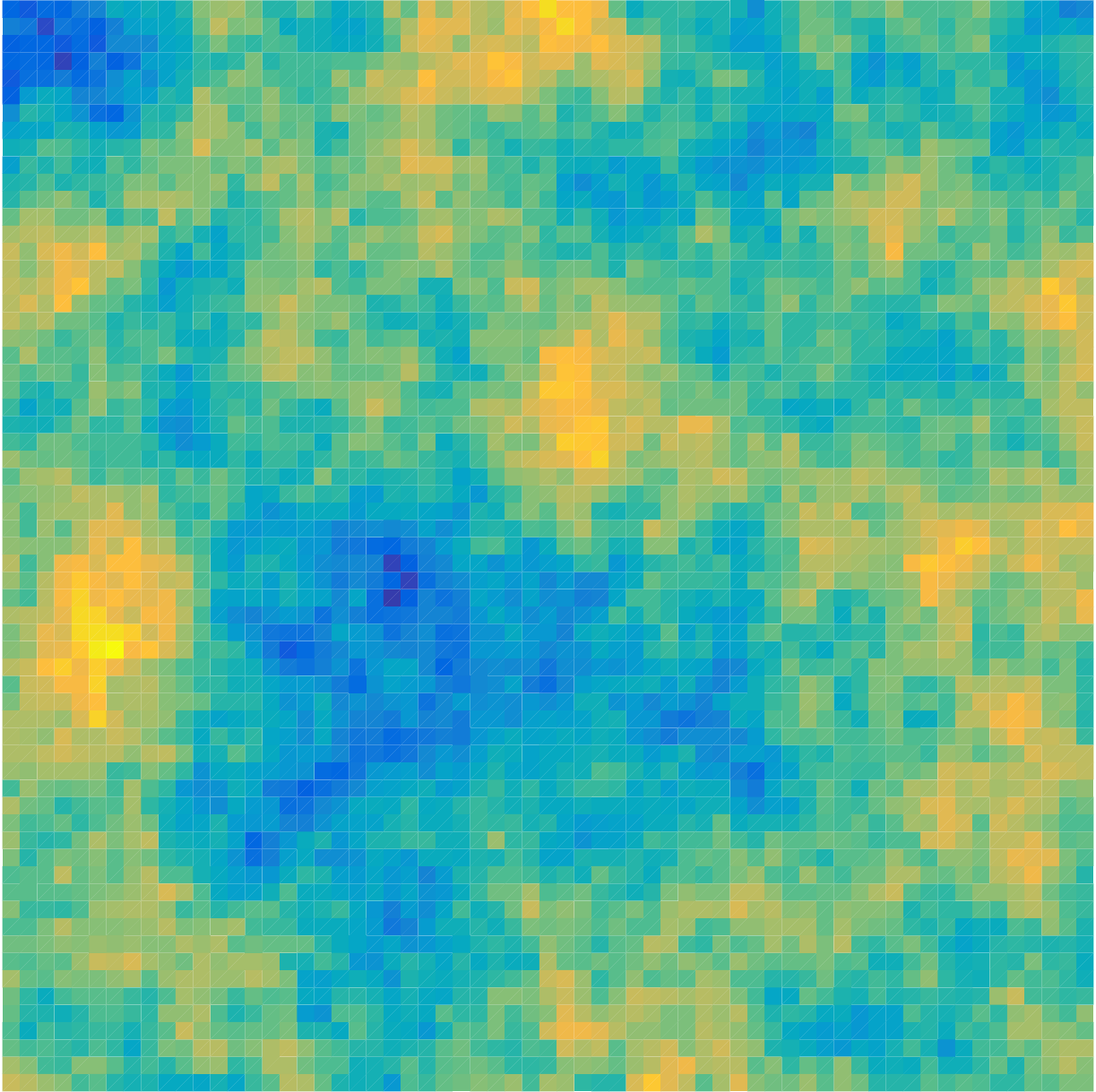}\hspace{5mm}%
\includegraphics[clip, width=0.3\textwidth]{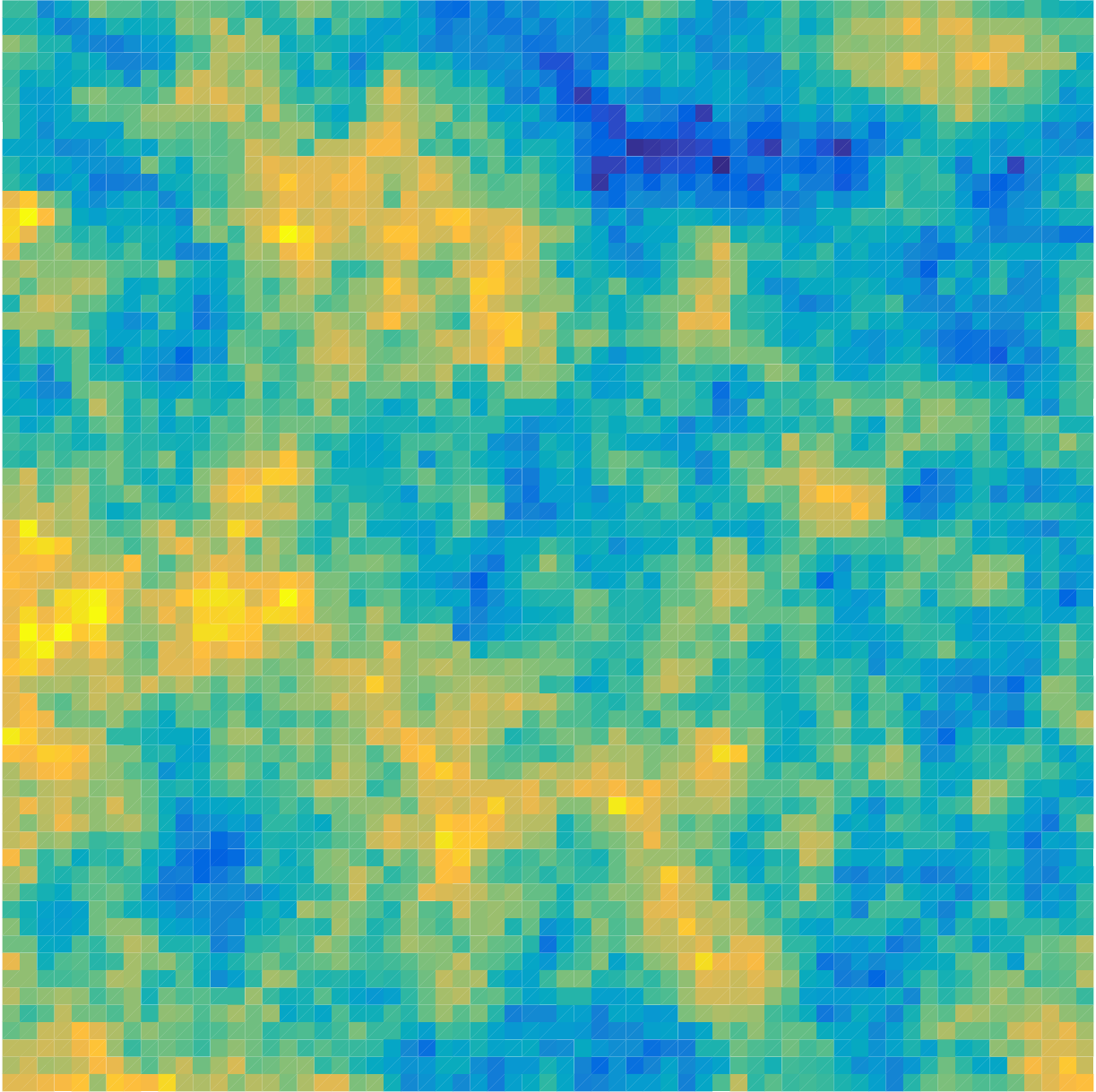}%
\caption{Samples of $\ZZ$ with a stationary covariance function from~\eqref{eq:matern} with $p=2$ and $\mu=1/2$.}
\label{fig:samplesexp}
\end{figure}
\begin{figure}
\includegraphics[ clip, width=0.3\textwidth]{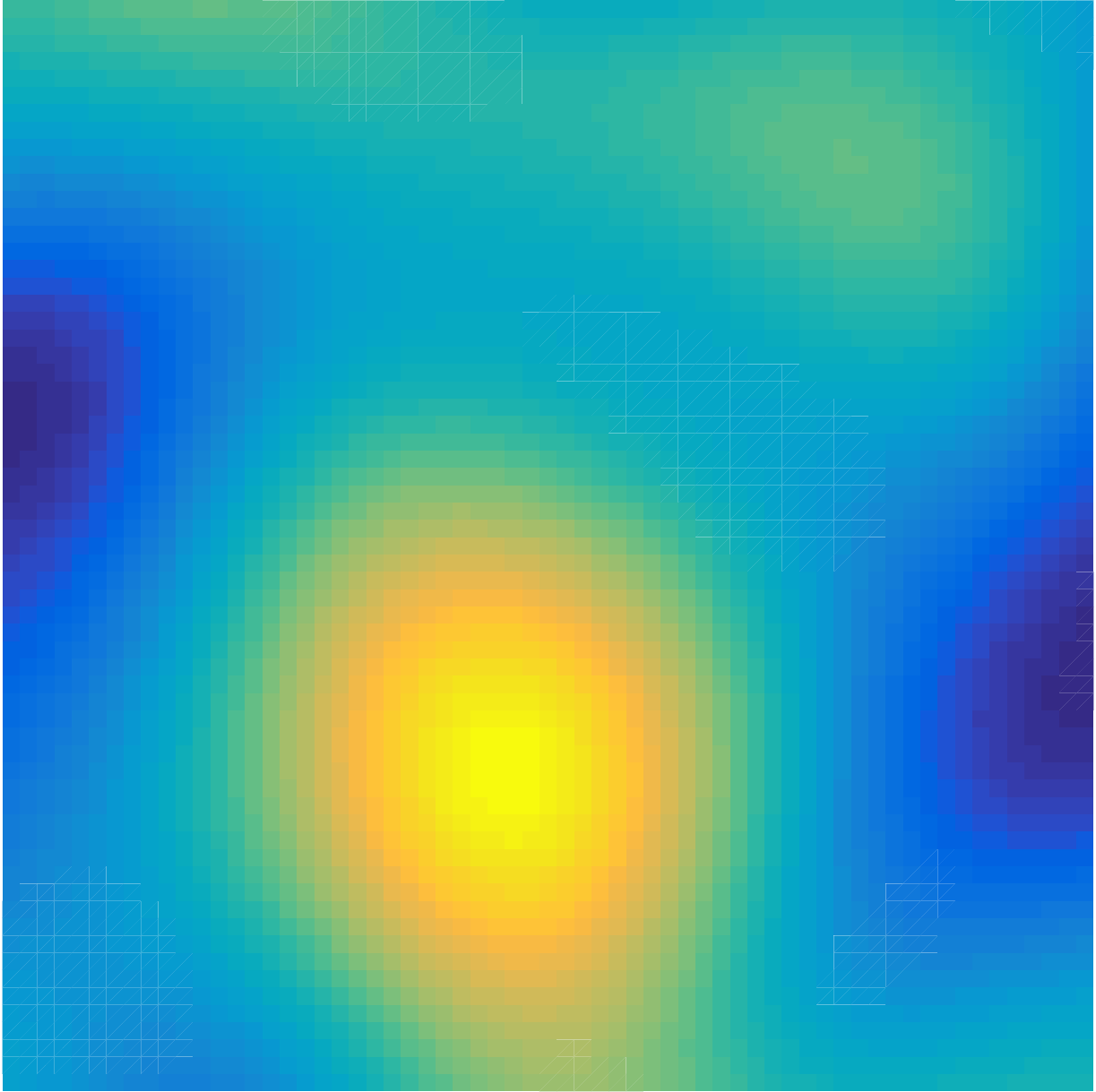}\hspace{5mm}%
\includegraphics[clip, width=0.3\textwidth]{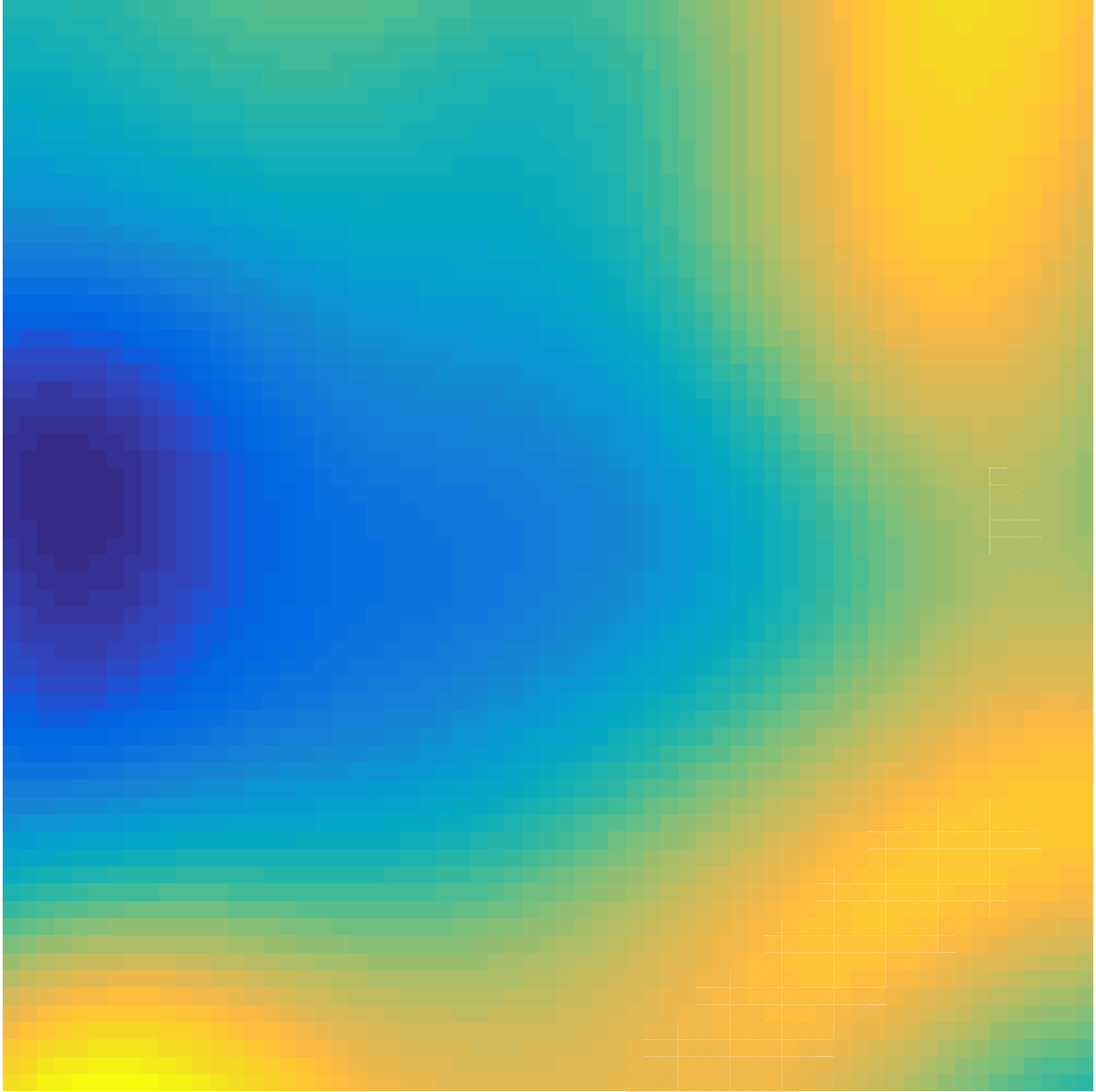}\hspace{5mm}%
\includegraphics[clip, width=0.3\textwidth]{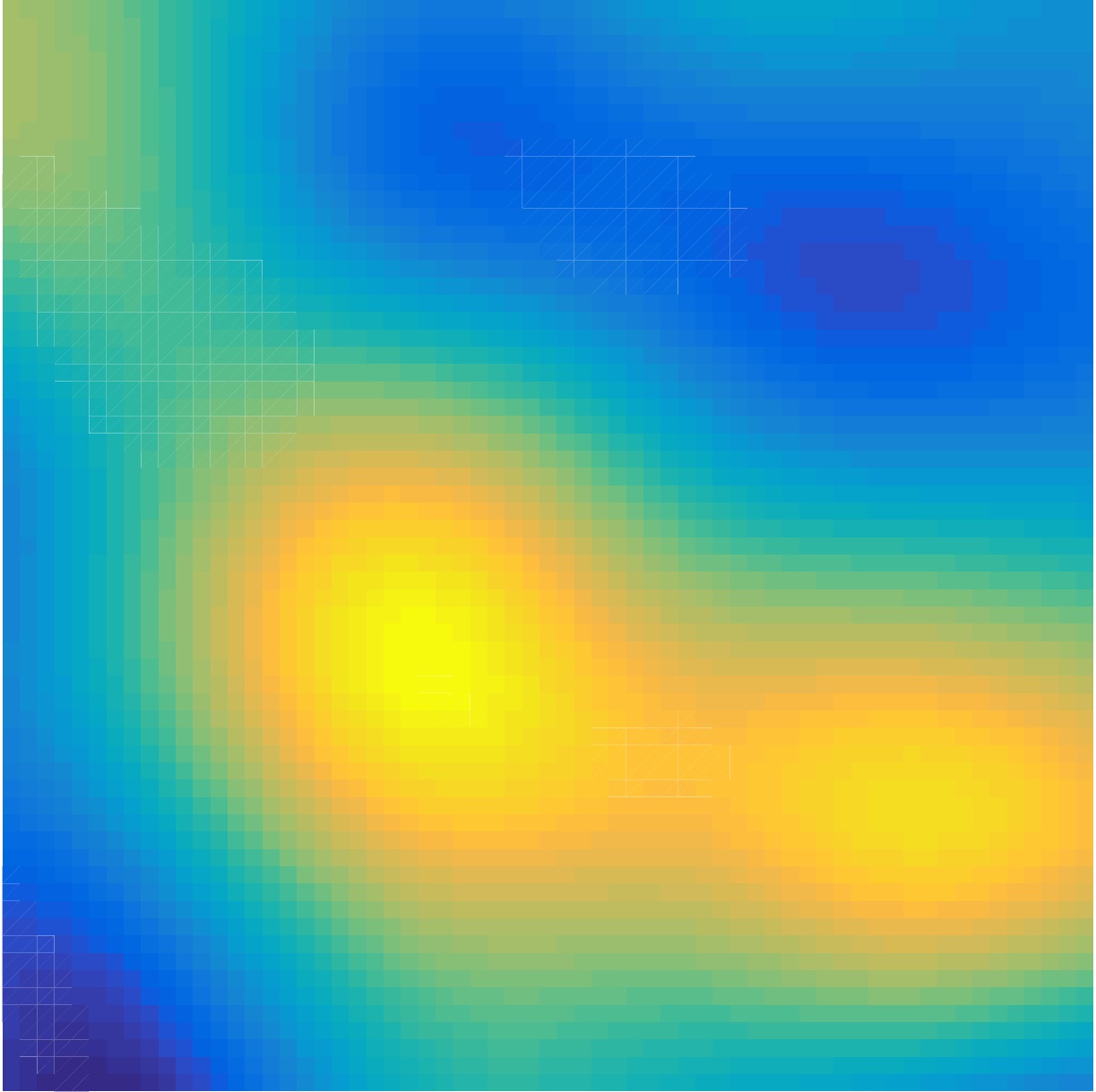}\\\vspace{5mm}%
\includegraphics[clip, width=0.3\textwidth]{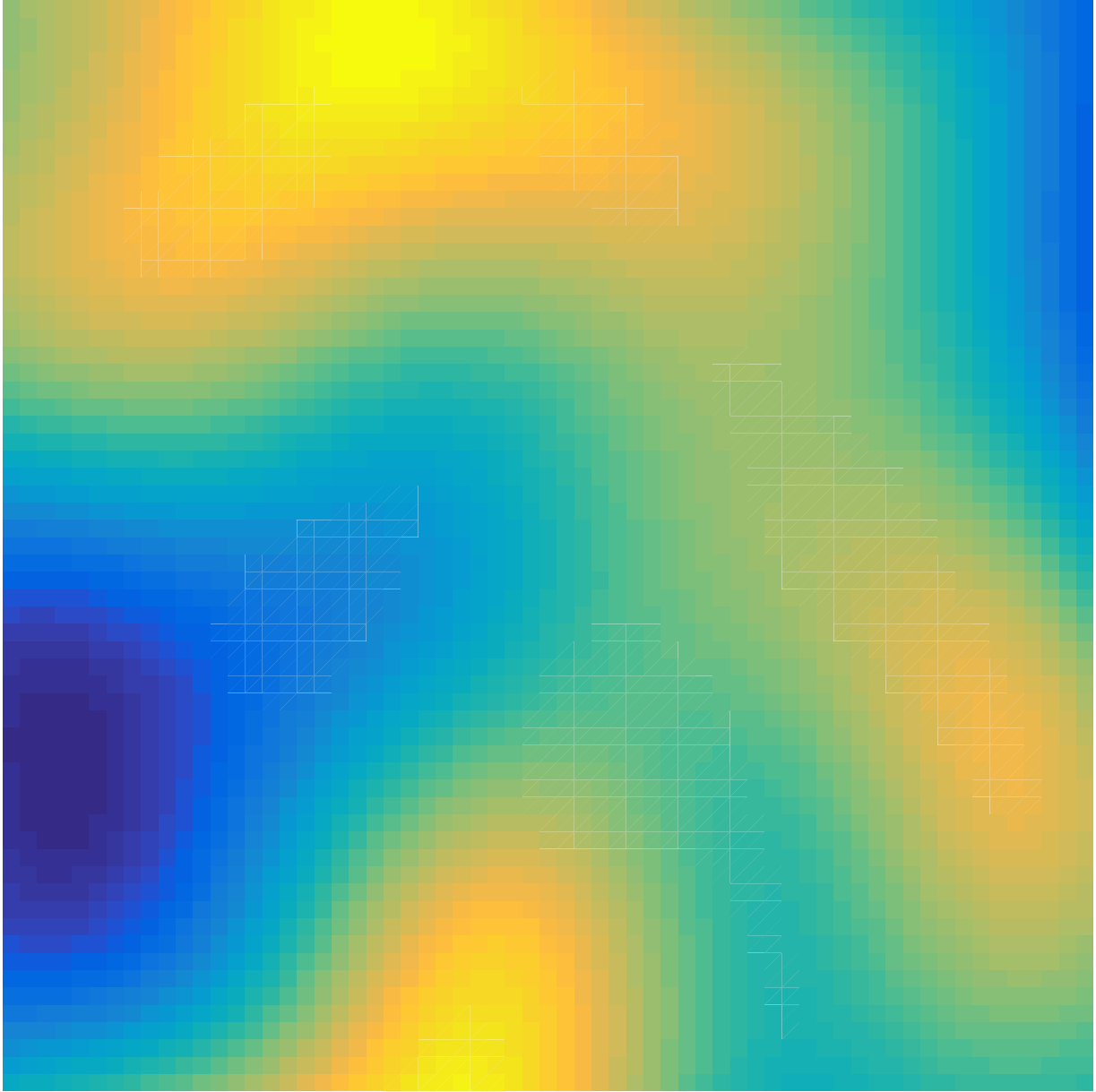}\hspace{5mm}%
\includegraphics[clip, width=0.3\textwidth]{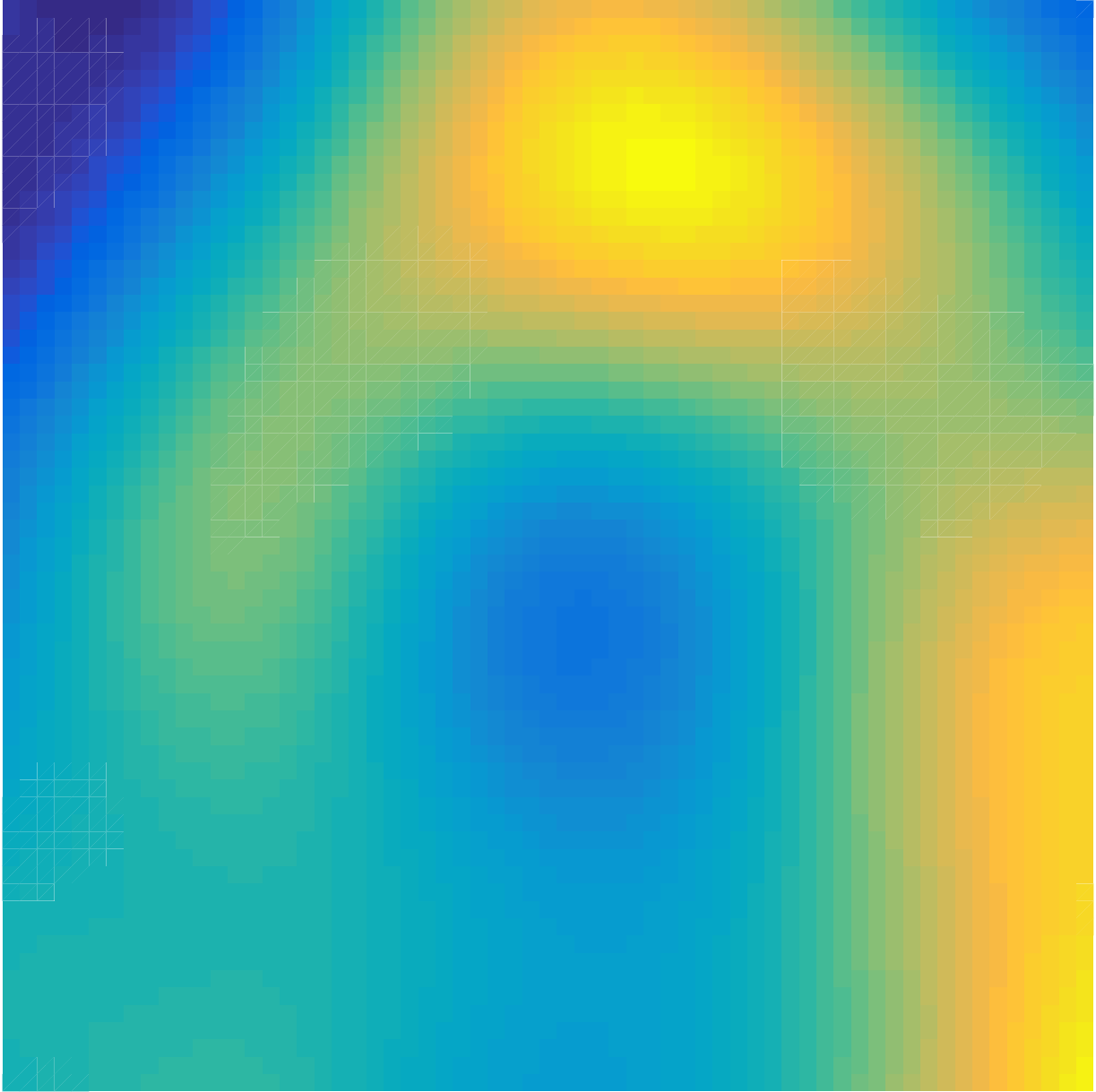}\hspace{5mm}%
\includegraphics[clip, width=0.3\textwidth]{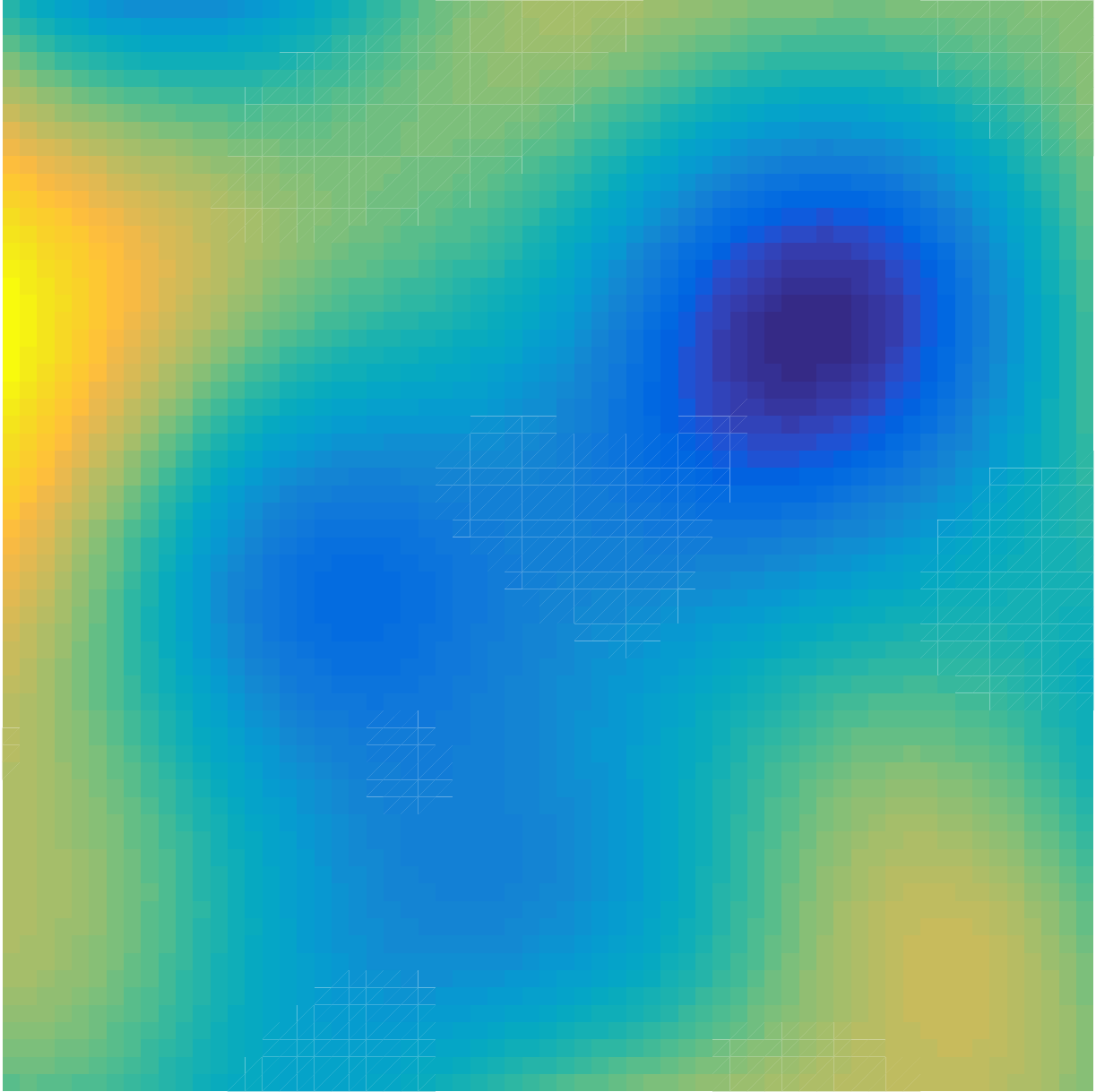}%
\caption{Samples of $\ZZ$ with a stationary covariance function from~\eqref{eq:matern} with $p=2$ and $\mu=\infty$.}
\label{fig:samplesgauss}
\end{figure}

To illustrate the challenging nature of handling these covariance matrices, Table~\ref{tab:cond} shows condition numbers of $\boldsymbol{C}$ for different problem sizes and the Mat\'ern covariance function~\eqref{eq:matern}.
\begin{table}
\begin{tabular}{l|cccccc}
$m=$ &  5 & 6 & 7 & 8 & 9 & 10\\ \hline
 $\lambda=1$ &2.0e+09 & 6.1e+16 & 8.6e+17 & 2.6e+19 & 1.8e+20 & 1.4e+20 \\ 
$\lambda=10^{-1}$ & 3.9e+07 & 5.5e+14 & 1.8e+17 & 8.4e+18 & 4.8e+20 & 4.6e+20 \\ 
$\lambda=10^{-2}$ & 6.5e+06 & 2.6e+12 & 2.7e+17 & 1.2e+19 & 3.3e+19 & 2.8e+20 \\ 
$\lambda=10^{-3}$ & 4.2e+06 & 9.4e+11 & 6.1e+17 & 4.2e+18 & 2.6e+19 & 1.1e+20
\end{tabular}
\caption{Condition numbers of $\boldsymbol{C}$ for the covariance function from~\eqref{eq:matern} with $\NN$ being a Sobol point set with $2^m$ points.}
\label{tab:cond}
\end{table}

For a performance comparison of Algorithm~\ref{alg1} and Algorithm~1\ref{alg2}, we consider the covariance function of the form~\eqref{eq:matern} with $p=2$, $\sigma=1$, and varying $\mu\in\{1/2,\infty\}$,
$\lambda\in\{1,10^{-1},10^{-2},10^{-3}\}$. We compute samples of $\ZZ(\bx,\omega)$ on a Sobol pointset with $2^{10}$ points. The results are plotted in Figure~\ref{fig:compare}
where we see the relative approximation error versus the computation time in seconds.  We observe that with respect to computational time, Algorithm~\ref{alg1} is superior in almost
all cases (particularly for smooth fields). However, keep in mind that according to Theorem~\ref{thm:sampleerror2}, Algorithm~\ref{alg1} needs up to $\mathcal{O}(\log_\kappa(\eps)N)$ extra storage,
while Algorithm~\ref{alg2} uses only $\mathcal{O}(\log(\log(\eps))N)$ extra storage units. (See Theorem~\ref{thm:sampleerror}, {where the quadratic convergence shows that $k\simeq \log(\log(\eps))$ is
sufficient to reach a given accuracy $\eps>0$. However, we have to mention that $k$ iterations of Algorithm~\ref{alg2} require $\mathcal{O}(3^k)$ arithmetic operations.)}

\begin{figure}
\psfrag{algorithm1}{\tiny Alg.~\ref{alg2}}
\psfrag{algorithm2}{\tiny Alg.~\ref{alg1}}
    \centering
     \includegraphics{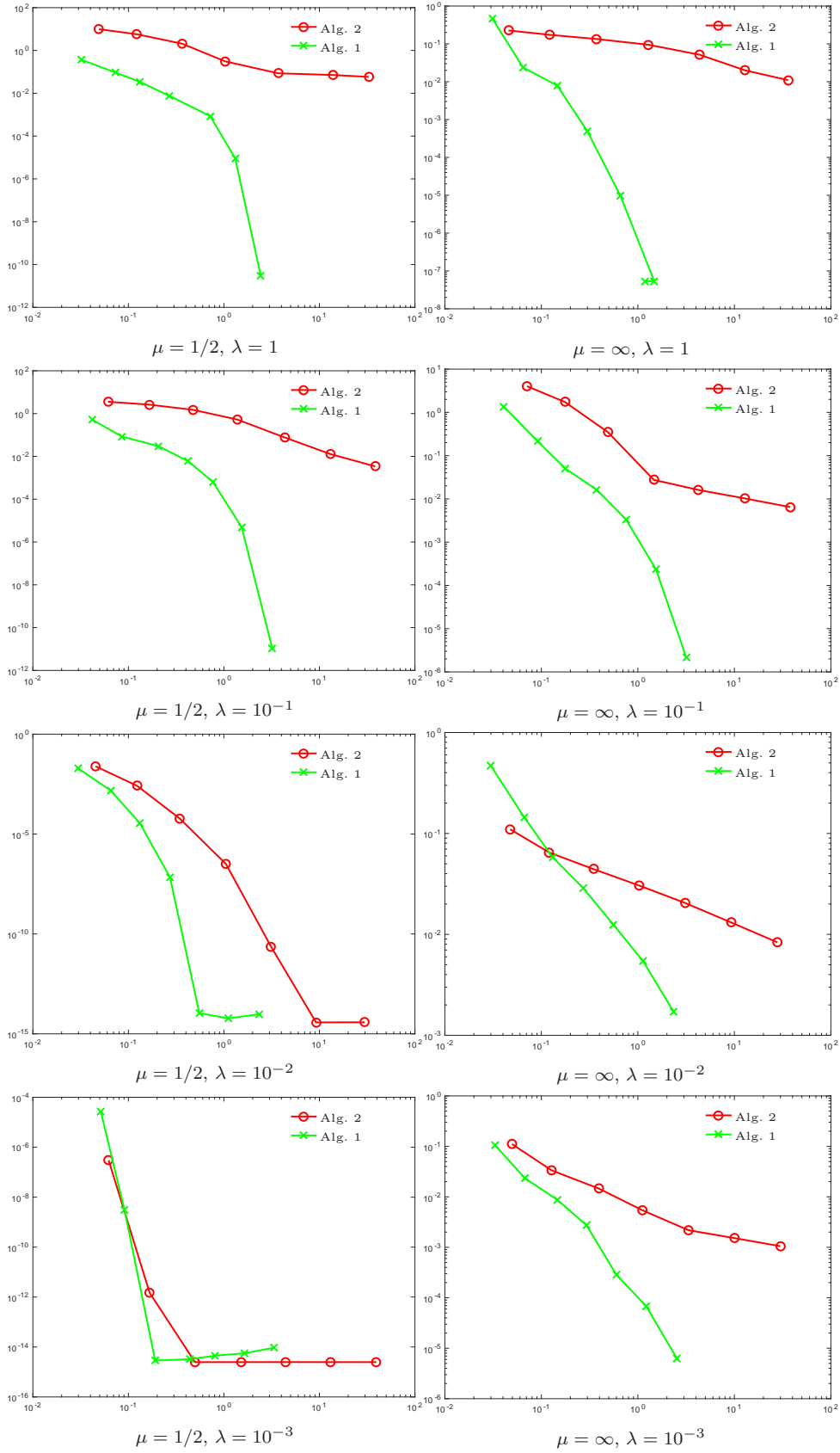}
 \caption{Comparison of Algorithm~\ref{alg2} and Algorithm~\ref{alg1}. We plot the relative error $|\ZZ_{k,p}(\bz)-\boldsymbol{C}^{1/2}\bz|/|\bz|$ versus  computation time in seconds.}
\label{fig:compare} 
\end{figure}
Figure~\ref{fig:compareTime} compares the two algorithms with the direct matrix square root provided by Matlab. We evaluate $\ZZ(\bx,\omega)$ on a Sobol pointset with size $2^m$ for $m\in\{1,\ldots,14\}$.
The number of iterations in both algorithm is set such that the relative error is smaller than $10^{-10}$ for the example from above with $p=2$, and varying $\mu\in\{1/2,\infty\}$,
$\lambda\in\{1,10^{-1},10^{-2},10^{-3}\}$.
We see that both, Algorithm~\ref{alg1}--\ref{alg2}, perform in linear time, whereas the direct approach comes closer to $\mathcal{O}(N^3)$. Even though our $H^2$-matrix library is programmed
entirely in Matlab (and thus nowhere near optimal performance), the breakthrough point at around $N=10^3$ shows that also small problems benefit from the speed up.
\begin{figure}
\psfrag{algorithm1}{\tiny Alg.~\ref{alg2}}
\psfrag{algorithm2}{\tiny Alg.~\ref{alg1}}
\psfrag{sqrtm}{\tiny direct}
\psfrag{time}[cc][cc]{\tiny time in seconds}
\psfrag{N}[cc][cc]{\tiny $N=|\NN|$}
\psfrag{OO}[cc][cc][1][25]{\tiny $\mathcal{O}(N)$}
    \centering
     \includegraphics[ clip, width=0.5\textwidth]{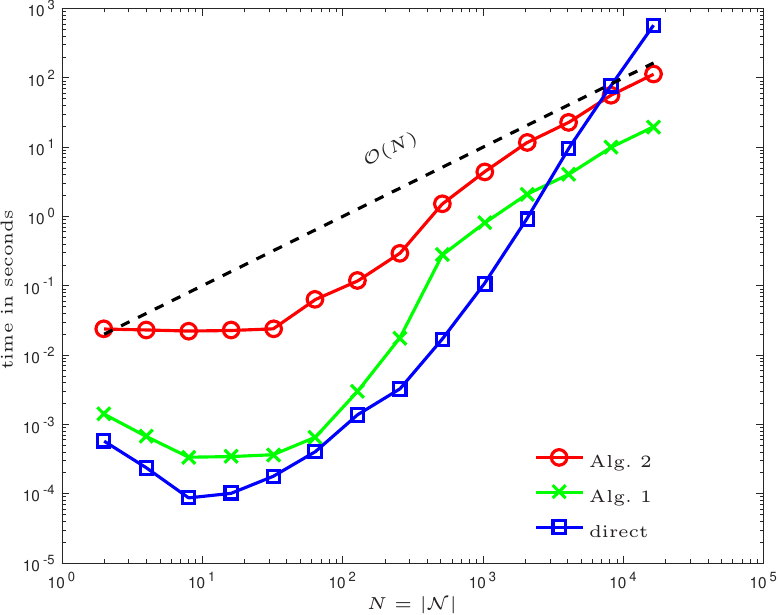}
     \caption{Computational time in seconds versus the number of evaluation points $N$. The direct approach uses Matlab's $\texttt{sqrtm}$ function.}
     \label{fig:compareTime}
\end{figure}

\FloatBarrier
\section{Lemmas for the proof of Theorem~\ref{thm:sampleerror2}}\label{section:method1}
First, we state a slight generalization of a well-known result.
\begin{lemma} \label{lem:sqrtm}
	Let $\boldsymbol{A},\boldsymbol{B}\in\R^{N\times N}$ be symmetric positive definite. Then, there holds
	\begin{align}\label{eq:sqstandard}
	\norm{\boldsymbol{A}^{1/2}-\boldsymbol{B}^{1/2}}{2}\leq (\lambda_{\rm min}(\boldsymbol{A})+\lambda_{\rm min}(\boldsymbol{B}))^{-1/2}\norm{\boldsymbol{A}-\boldsymbol{B}}{2},
	\end{align}
	as well as
	\begin{align}\label{eq:sqimproved}
	\norm{\boldsymbol{A}^{1/2}-\boldsymbol{B}^{1/2}}{2}\leq 3\norm{\boldsymbol{A}-\boldsymbol{B}}{2}^{1/2}.
	\end{align}%
\end{lemma}
\begin{proof}
 The estimate~\eqref{eq:sqstandard} is proved in~\cite[Lemma~2.2]{matrixsq}. To obtain~\eqref{eq:sqimproved}, let $\boldsymbol{U}\in\R^{N\times N}$ denote
 the orthonormal matrix that diagonalizes $\boldsymbol{A}$, i.e., $\boldsymbol{U}^T\boldsymbol{A}\boldsymbol{U}=\boldsymbol{D}$ for a {positive} diagonal matrix $\boldsymbol{D}\in\R^{N\times N}$.
 With $\boldsymbol{U}\boldsymbol{D}^{1/2}\boldsymbol{U}^T=\boldsymbol{A}^{1/2}$ and $\boldsymbol{U}\boldsymbol{U}^T=\boldsymbol{I}$, there holds for arbitrary $\alpha\geq 0$ 
 \begin{align*}
  \norm{\boldsymbol{A}^{1/2}-(\boldsymbol{A}+\alpha\boldsymbol{I})^{1/2}}{2}&=
  \norm{\boldsymbol{U}\boldsymbol{D}^{1/2}\boldsymbol{U}^T-\boldsymbol{U}(\boldsymbol{D}+\alpha\boldsymbol{I})^{1/2}\boldsymbol{U}^T}{2}\\
  &=\norm{\boldsymbol{D}^{1/2}-(\boldsymbol{D}+\alpha\boldsymbol{I})^{1/2}}{2}=\max_{1\leq i\leq N}\big|\sqrt{\boldsymbol{D}_{ii}+\alpha}-\sqrt{\boldsymbol{D}_{ii}}\big|\leq \sqrt{\alpha},
 \end{align*}
{where we used $x+y\leq (\sqrt{x}+\sqrt{y})^2$ and hence $\sqrt{x+y}\leq \sqrt{x}+\sqrt{y}$ for $x,y\geq 0$ in the last estimate.}
With $\alpha:=\norm{\boldsymbol{A}-\boldsymbol{B}}{2}$,~\eqref{eq:sqstandard} shows
\begin{align*}
 \norm{(\boldsymbol{A}+\alpha\boldsymbol{I})^{1/2}-\boldsymbol{B}^{1/2}}{2}\leq  2\alpha
 (\lambda_{\rm min}(\boldsymbol{A})+\lambda_{\rm min}(\boldsymbol{B})+\alpha)^{-1/2}\leq 2\sqrt{\alpha}.
\end{align*}
The combination of the last two estimates concludes the proof of~\eqref{eq:sqimproved}.
\end{proof}

 \begin{lemma}\label{lem:vandermondeplusqr}
  Let $\boldsymbol{M}\in\R^{N\times N}$ be symmetric positive definite and assume
that $0<\kappa<1$ and $C_\kappa>0$ are such
	that the sequence of all distinct eigenvalues $\lambda_1> \ldots>\lambda_M>0\in\R$ (for some $M\leq N$) of $\boldsymbol{M}$ satisfies
	$|\lambda_{i}-\lambda_{j}|\leq \lambda_1C_\kappa \kappa^{\min\{i,j\}}$ for all $1\leq i,j\leq M$.
	Given $1\leq k\leq M$ and $\bz\in\R^N$, define $\boldsymbol{Z}\in\R^{N\times k}$ by
	\begin{align}\label{eq:defZmatrix}
	\boldsymbol{Z}:=(\bz,\lambda_1^{-1}\boldsymbol{M}\bz,\lambda_1^{-2}\boldsymbol{M}^2\bz,\ldots,\lambda_1^{-(k-1)}\boldsymbol{M}^{k-1} \bz).
	\end{align}
	Consider the $QR$-factorization $\boldsymbol{Z}=\boldsymbol{Q}\boldsymbol{R}$, with $\boldsymbol{Q}\in
\R^{N\times k}$ satisfying $\boldsymbol{Q}^T \boldsymbol{Q} =\boldsymbol{I}_k$ and $\boldsymbol{R}\in\R^{k\times
k}$ upper triangular with non-negative diagonal entries (note that if $\boldsymbol{Z}$ has full rank, this ensures uniqueness of $\boldsymbol{Q}$ and $\boldsymbol{R}$). 
Then the diagonal entries of $\boldsymbol{R}$ satisfy
	\begin{align}\label{eq:qr}
	\boldsymbol{R}_{nn}\leq |\bz|C_\kappa^{n-1}\kappa^{(n-1)n/2}\quad\text{for all }1\leq n\leq k.
	\end{align}
 \end{lemma}
\begin{proof}
Let $\bq^i$, $i=1,\ldots,k$ denote the orthonormal columns of $\boldsymbol{Q}$.
 By definition of the $QR$-factorization, there holds for $1\leq n\leq k$
 \begin{align*}
  \lambda_1^{-(n-1)}\boldsymbol{M}^{n-1}\bz = \sum_{i=1}^n \boldsymbol{R}_{in}\bq^i.
 \end{align*}
 Since the $\bq^i$ are orthogonal, the best approximation (with respect to $|\cdot|$) of $\lambda_1^{-(n-1)}\boldsymbol{M}^{n-1}\bz$ in ${\rm span}\{\bq^1,\ldots,\bq^{n-1}\}$ is given by
 $\sum_{i=1}^{n-1} \boldsymbol{R}_{in}\bq^i$ for all $n\geq 2$. Therefore, we obtain
\begin{align*}
 \boldsymbol{R}_{nn} = \Big|\lambda_1^{-(n-1)}\boldsymbol{M}^{n-1}\bz - \sum_{i=1}^{n-1} \boldsymbol{R}_{in}\bq^i\Big|\leq
 \min_{\bv\in {\rm span}\{\bz,\ldots,\boldsymbol{M}^{n-2}\bz\}} |\lambda_1^{-(n-1)}\boldsymbol{M}^{n-1}\bz-\bv|,
\end{align*}
where we used ${\rm span}\{\bz,\ldots,\boldsymbol{M}^{n-2}\bz\}\subseteq {\rm span}\{\bq^1,\ldots,\bq^{n-1}\}$ by definition of the $QR$-factorization (see also Remark~\ref{rem:qr}).
We may choose $\bv=p(\boldsymbol{M})\bz$, where $p(x)$ is the polynomial of degree $n-2$ interpolating $f(x):=(x/\lambda_1)^{n-1}$ at the points $x=\lambda_1,\ldots,\lambda_{n-1}$.
Since $\boldsymbol{M}$ is symmetric and positive definite, we may diagonalize it with an orthogonal matrix $\boldsymbol{U}\in\R^{N\times N}$, i.e., $\boldsymbol{U}^T\boldsymbol{M}\boldsymbol{U}=\boldsymbol{D}$ 
with a diagonal matrix $\boldsymbol{D}\in\R^{N\times N}$ containing the eigenvalues of $\boldsymbol{M}$. This allows us to conclude
\begin{align*} 
\boldsymbol{R}_{nn} &\leq \norm{f(\boldsymbol{M})-p(\boldsymbol{M})}{2}|\bz|=\norm{\boldsymbol{U}^T(f(\boldsymbol{D})-p(\boldsymbol{D}))\boldsymbol{U}}{2}|\bz|=
\norm{f(\boldsymbol{D})-p(\boldsymbol{D})}{2}|\bz|\\
&\leq \max_{x\in\{\lambda_1,\ldots,\lambda_M\}}|f(x)-p(x)||\bz| =\max_{x\in\{\lambda_{n},\ldots,\lambda_M\}}|f(x)-p(x)||\bz|.
\end{align*}
The function $f(x)-p(x)$ is a polynomial of degree $n-1$ with known zeros $\lambda_1,\ldots,\lambda_{n-1}$ and thus reads
\begin{align*}
 f(x)-p(x)=\alpha(x-\lambda_1)\cdots(x-\lambda_{n-1})
\end{align*}
for some leading coefficient $\alpha\in\R$. Differentiation reveals $\alpha (n-1)! = f^{(n-1)}(x) = (n-1)!\lambda_1^{-(n-1)}$ and hence $\alpha=\lambda_1^{-(n-1)}$.
This shows
\begin{align*}
  \boldsymbol{R}_{nn}\leq |\bz|\max_{n\leq i\leq M}\prod_{j=1}^{n-1}\frac{|\lambda_i-\lambda_j|}{\lambda_1}=|\bz|\prod_{j=1}^{n-1}\frac{|\lambda_M-\lambda_j|}{\lambda_1}.
\end{align*}
By the decay assumption on the $\lambda_i$ it
follows that
	\begin{align}\label{eq:max3}
	 \boldsymbol{R}_{nn}\leq |\bz|\prod_{j=1}^{n-1}(C_\kappa \kappa^j) = |\bz|C_{\kappa}^{n-1} \kappa^{n(n-1)/2}.
	\end{align}
	This concludes the proof.
\end{proof}

The next lemma shows that the matrices $\boldsymbol{Q}_j$ from Algorithm~\ref{alg1} are strongly tied to the matrices $\boldsymbol{Z}=\boldsymbol{Q}\boldsymbol{R}$ defined in Lemma~\ref{lem:vandermondeplusqr}.
\begin{lemma}\label{lem:qq}
 Given $\bz\in\R^N$ and let $\boldsymbol{M}\in\R^{N\times N}$ be symmetric positive definite.
 Call Algorithm~\ref{alg1} with $\boldsymbol{M}$, $\bz$, and $k\in\N$ to compute $k_0\leq k$ and $\boldsymbol{R}_j$, $\boldsymbol{Q}_j$ for all $1\leq j\leq k_0$.
 Define $\boldsymbol{Z},\boldsymbol{Q},\boldsymbol{R}$ satisfying $\boldsymbol{Z}=\boldsymbol{Q}\boldsymbol{R}$ as in Lemma~\ref{lem:vandermondeplusqr}.
 Then, $\boldsymbol{Q}_j$ (as defined in Algorithm~\ref{alg1}) for $1\leq j\leq k_0$ satisfies
 $\boldsymbol{Q}_j=\boldsymbol{Q}|_{\{1,\ldots,N\}\times\{1,\ldots,j\}}$, i.e., the first $j$ columns coincide and
\begin{align}\label{eq:rank2}
{\rm range}(\boldsymbol{Q}_{j})={\rm span}\{\bz,\ldots,\boldsymbol{M}^{j-1}\bz\}={\rm range}(\boldsymbol{Z}|_{\{1,\ldots,N\}\times\{1,\ldots,j\}})
\end{align}
for all $1\leq j\leq k_0$. Moreover, $\boldsymbol{Z}$ has full rank if and only if $k_0=k$.
\end{lemma}
\begin{proof}
Let $\bq^j$ denote the $j$-th column of $\boldsymbol{Q}_j$ and note that by definition of Algorithm~\ref{alg1} we have
\begin{align}\label{eq:diag}
(\boldsymbol{R}_j)_{jj}> 0\quad\text{for all }1\leq j\leq k_0. 
\end{align}
In order to prove~\eqref{eq:rank2}, we first show
\begin{align}\label{eq:rank22}
{\rm range}(\boldsymbol{Q}_{j})={\rm span}\{\bz,\ldots,\boldsymbol{M}^{j-1}\bz\}
\end{align}
for all $1\leq j\leq k_0$ by induction. To that end, note that $\boldsymbol{Q}_1=\bq^1=\bz/|\bz|$ and consequently~\eqref{eq:rank22} holds for $j=1$.
Assume~\eqref{eq:rank22} holds for all $1\leq j< j_0\leq k_0$.
By construction of the matrices in Algorithm~\ref{alg1}, we have
\begin{align}\label{eq:rank}
 (\boldsymbol{Q}_{j_0-1},\boldsymbol{M}\bq^{j_0-1})=\boldsymbol{Q}_{j_0}\boldsymbol{R}_{j_0}.
\end{align}
By the induction assumption, $\bq^{j_0-1}\in {\rm span}\{\bz,\ldots,\boldsymbol{M}^{j_0-2}\bz\}$. Thus,~\eqref{eq:rank} and the fact that $\boldsymbol{R}_{j_0}$
is regular (by~\eqref{eq:diag}) imply
\begin{align*}
	{\rm range}(\boldsymbol{Q}_{j_0})={\rm span}\{{\rm range}(\boldsymbol{Q}_{j_0-1}), \boldsymbol{M}\boldsymbol{q}^{j_0-1}\}\subseteq {\rm span}\{\bz,\ldots,\boldsymbol{M}^{j_0-1}\bz\}.
\end{align*}
The fact that $\boldsymbol{Q}_{j_0}$ is orthogonal (and hence its range is $j_0$ dimensional) shows even equality, that is
\begin{align}\label{eq:iid1}
	{\rm range}(\boldsymbol{Q}_{j_0})= {\rm span}\{\bz,\ldots,\boldsymbol{M}^{j_0-1}\bz\}.
\end{align}
This concludes the induction, and proves~\eqref{eq:rank22} for all $1\leq j\leq k_0$.
The second equation in~\eqref{eq:rank2} follows by definition of $\boldsymbol{Z}$.

To see the remainder of the statement, we first assume $k_0=k$ and proceed to prove that $\boldsymbol{Z}$ has full rank.
To that end, we apply~\eqref{eq:rank2} with $j=k$ to see that ${\rm range}(\boldsymbol{Z})={\rm range}(\boldsymbol{Q}_k)$ is $k$-dimensional and therefore $\boldsymbol{Z}$ has full rank.


For the converse implication, assume that $\boldsymbol{Z}$ has full rank.
 We prove $k_0=k$ by induction.
By construction, we have $(\boldsymbol{R}_1)_{11}=1$ and thus $k_0\geq 1$.
Assume $k_0\geq j_0$ for some $j_0<k$. Then, since $(\boldsymbol{R}_j)_{jj}\neq 0$ for all $1\leq j< j_0$, the identity~\eqref{eq:rank2} shows ${\rm range}(\boldsymbol{Z}|_{\{1,\ldots,N\}\times\{1,\ldots,j\}})={\rm range}(\boldsymbol{Q}_{j})$ for all $j<j_0$.
From this, we argue that 
\begin{align*}
\bq^{j_0-1}\in {\rm range}(\boldsymbol{Z}|_{\{1,\ldots,N\}\times\{1,\ldots,j_0-1\}})\setminus{\rm range}(\boldsymbol{Z}|_{\{1,\ldots,N\}\times\{1,\ldots,j_0-2\}}),
\end{align*}
{which, by definition of $\boldsymbol{Z}=(\bz,\lambda_1^{-1}\boldsymbol{M}\bz,\ldots,\lambda_1^{-(k-1)}\boldsymbol{M}^{k-1}\bz)$, shows that $\bq^{j_0-1}=\sum_{i=0}^{j_0-2}\alpha_i\boldsymbol{M}^i\bz$ for some $\alpha_i\in\R$ with $\alpha_{j_0-2}\neq 0$.
Consequently, we obtain $\boldsymbol{M}\bq^{j_0-1}=\sum_{i=0}^{j_0-2}\alpha_i\boldsymbol{M}^{i+1}\bz\in 
{\rm range}(\boldsymbol{Z}|_{\{1,\ldots,N\}\times\{1,\ldots,j_0\}})\setminus{\rm range}(\boldsymbol{Z}|_{\{1,\ldots,N\}\times\{1,\ldots,j_0-1\}})$.
Since ${\rm range}(\boldsymbol{Z}|_{\{1,\ldots,N\}\times\{1,\ldots,j_0-1\}})={\rm range}(\boldsymbol{Q}_{j_0-1})$, this implies the identity}
${\rm range}((\boldsymbol{Q}_{j_0-1},\boldsymbol{M}\bq^{j_0-1}))={\rm range}(\boldsymbol{Z}|_{\{1,\ldots,N\}\times\{1,\ldots,j_0\}}) $
and therefore the matrix $(\boldsymbol{Q}_{j_0-1},\boldsymbol{M}\bq^{j_0-1})$ has full rank.
Hence,~\eqref{eq:rank} implies that $\boldsymbol{R}_{j_0}$ has full rank, which in particular implies $(\boldsymbol{R}_{j_0})_{j_0j_0}\neq 0$ and thus $k_0\geq j_0+1$.
This concludes the induction and shows $k_0=k$.
\end{proof}

The following result proves that if Algorithm~\ref{alg1} terminates in less than $k$ steps (due to the criterion in step~1(c)), the quantity $\boldsymbol{M}^{1/2}\bz$ is computed exactly.
\begin{lemma}\label{lem:qr1.4}
 Let $\bz\in\R^N$ and let $\boldsymbol{M}\in\R^{N\times N}$ be symmetric positive definite. 
  Call Algorithm~\ref{alg1} with $\boldsymbol{M}$, $\bz$, and $k\in\N$ to compute $k_0\leq k$ as well as $\boldsymbol{Q}_j$ for all $1\leq j\leq k_0$.
	Define $\boldsymbol{U}_{k_0}=\boldsymbol{Q}_{k_0}^T\boldsymbol{M}\boldsymbol{Q}_{k_0}$ as in Algorithm~\ref{alg1}. If $k_0<k$, there holds
	\begin{align*}
	 \boldsymbol{M}^{1/2}\bz=\boldsymbol{Q}_{k_0}\boldsymbol{U}_{k_0}^{1/2}\boldsymbol{Q}_{k_0}^T\bz.
	\end{align*}
\end{lemma}
\begin{proof}
 If $k_0<k$ then Lemma~\ref{lem:qq} shows that $\boldsymbol{Z}$ as defined in Lemma~\ref{lem:vandermondeplusqr} does not have full rank. Moreover, the identity ~\eqref{eq:rank2}
 shows
 that  $\boldsymbol{Z}|_{\{1,\ldots,N\}\times\{1,\ldots,k_0\}}$ has full rank. 
 By definition of $\boldsymbol{Z}$, this implies ${\rm range}(\boldsymbol{Z}|_{\{1,\ldots,N\}\times\{1,\ldots,k_0\}})={\rm range}(\boldsymbol{Z})$.
Therefore,~\eqref{eq:rank2} shows 
\begin{align}\label{eq:rr}
\begin{split}
{\rm range}(\boldsymbol{M}\boldsymbol{Q}_{k_0})&= {\rm range}(\boldsymbol{M}\boldsymbol{Z}|_{\{1,\ldots,N\}\times\{1,\ldots,k_0\}})\\
&\subseteq {\rm range}(\boldsymbol{Z})= {\rm range}(\boldsymbol{Z}|_{\{1,\ldots,N\}\times\{1,\ldots,k_0\}})={\rm range}(\boldsymbol{Q}_{k_0}).
\end{split}
\end{align}
Let $\overline{\boldsymbol{Q}}\in\R^{N\times N}$ be an orthonormal matrix such that its first $k_0$ columns coincide with $\boldsymbol{Q}_{k_0}$, i.e., $\overline{\boldsymbol{Q}}=(\boldsymbol{Q}_{k_0},\boldsymbol{Q}_{\perp})$ 
for some orthonormal $\boldsymbol{Q}_\perp\in\R^{N\times (N-k_0)}$.
We obtain
\begin{align}\label{eq:Mrr}
 \boldsymbol{M}^{1/2}=\overline{\boldsymbol{Q}}\,\overline{\boldsymbol{Q}}^T\boldsymbol{M}^{1/2}\overline{\boldsymbol{Q}}\,\overline{\boldsymbol{Q}}^T=\overline{\boldsymbol{Q}}\,(\overline{\boldsymbol{Q}}^T\boldsymbol{M}\overline{\boldsymbol{Q}})^{1/2}\,\overline{\boldsymbol{Q}}^T.
\end{align}
There holds
\begin{align*}
 \overline{\boldsymbol{Q}}^T\boldsymbol{M}\overline{\boldsymbol{Q}}=\begin{pmatrix}\boldsymbol{Q}_{k_0}^T\boldsymbol{M}\boldsymbol{Q}_{k_0} & \boldsymbol{Q}_{k_0}^T\boldsymbol{M}\boldsymbol{Q}_{\perp}\\
                                \boldsymbol{Q}_{\perp}^T\boldsymbol{M}\boldsymbol{Q}_{k_0} & \boldsymbol{Q}_{\perp}^T\boldsymbol{M}\boldsymbol{Q}_{\perp}
                               \end{pmatrix}.
\end{align*}
The invariance property~\eqref{eq:rr} shows $\boldsymbol{Q}_{\perp}^T\boldsymbol{M}\boldsymbol{Q}_{k_0}=\boldsymbol{0}$, and by symmetry also $\boldsymbol{Q}_{k_0}^T\boldsymbol{M}\boldsymbol{Q}_{\perp}=\boldsymbol{0}$. Therefore, we have
\begin{align*}
  \overline{(\boldsymbol{Q}}^T\boldsymbol{M}\overline{\boldsymbol{Q}})^{1/2}=\begin{pmatrix}\boldsymbol{U}_{k_0}^{1/2} & \boldsymbol{0}\\
                                \boldsymbol{0} & (\boldsymbol{Q}_{\perp}^T\boldsymbol{M}\boldsymbol{Q}_{\perp})^{1/2}
                               \end{pmatrix}.
\end{align*}
This and~\eqref{eq:Mrr}, together with $\bz\in{\rm range}(\boldsymbol{Q}_{k_0})$, show $\boldsymbol{M}^{1/2}\bz = \overline{\boldsymbol{Q}}\,(\overline{\boldsymbol{Q}}^T\boldsymbol{M}\overline{\boldsymbol{Q}})^{1/2}\,\overline{\boldsymbol{Q}}^T\bz=\boldsymbol{Q}_{k_0}\boldsymbol{U}_{k_0}^{1/2}\boldsymbol{Q}_{k_0}^T\bz$ and conclude the proof.
\end{proof}

The following result is the main tool to prove Theorem~\ref{thm:sampleerror2}~(i).
\begin{lemma}\label{lem:qr1.5}
	Let $\bz\in\R^N$ and let $\boldsymbol{M}\in\R^{N\times N}$ be symmetric positive definite. 
	Call Algorithm~\ref{alg1} with $\boldsymbol{M}$, $\bz$, and $k\in\N$ to compute $k_0\leq k$ as well as $\boldsymbol{Q}_j$ for all $1\leq j\leq k_0$.
	Let $\boldsymbol{U}_{k_0}=\boldsymbol{Q}_{k_0}^T\boldsymbol{M}\boldsymbol{Q}_{k_0}$ be defined as in Algorithm~\ref{alg1}.
	Then, there holds
	\begin{align*}
	 &\frac{|\boldsymbol{M}^{1/2}\bz-\boldsymbol{Q}_{k_0}\boldsymbol{U}_{k_0}^{1/2}\boldsymbol{Q}_{k_0}^T\bz|}{|\bz|} \leq
	 \begin{cases} \displaystyle
	 {\sqrt{2\norm{\boldsymbol{M}}{2}}}\frac{4r^{{2}}}{r-1}r^{-k} & \text{ if }k_0=k,\\
	\displaystyle 0 & \text{ if } k_0<k,
	\end{cases}
	\end{align*}
	where
	 \begin{align}\label{eq:r}
 r:= \frac{\lambda_{\rm max}(\boldsymbol{M})+\lambda_{\rm min}(\boldsymbol{M})}{\lambda_{\rm max}(\boldsymbol{M})-\lambda_{\rm min}(\boldsymbol{M})}>1.
\end{align}
\end{lemma}
\begin{proof}
The case $k_0<k$ is covered in Lemma~\ref{lem:qr1.4}. Assume $k_0=k$.
Note that $\boldsymbol{Q}_{k}\boldsymbol{Q}_{k}^T$ is the identity on ${\rm range}(\boldsymbol{Q}_{k})$. Lemma~\ref{lem:qq} shows
that $\boldsymbol{M}^j\bz\in {\rm range}(\boldsymbol{Q}_{k})$ for all $0\leq j\leq k-1$. Moreover, $\bz\in {\rm range}(\boldsymbol{Q}_k)$ by construction. Hence, we have
\begin{align*}
\boldsymbol{M}^j\bz= \boldsymbol{Q}_{k}\boldsymbol{Q}_{k}^T\boldsymbol{M}^j\bz= \boldsymbol{Q}_{k}\boldsymbol{Q}_{k}^T\boldsymbol{M}^j\boldsymbol{Q}_{k}\boldsymbol{Q}_{k}^T\bz=  \boldsymbol{Q}_{k}(\boldsymbol{Q}_{k}^T\boldsymbol{M}\boldsymbol{Q}_{k})^j\boldsymbol{Q}_{k}^T\bz\quad\text{for all }1\leq j\leq k-1.
\end{align*}
Thus, any polynomial $p\in\PP^{k-1}$ of degree $k-1$ satisfies
\begin{align*}
 p(\boldsymbol{M})\bz=\boldsymbol{Q}_{k}\boldsymbol{Q}_{k}^Tp(\boldsymbol{M})\boldsymbol{Q}_{k}\boldsymbol{Q}_{k}^T\bz=\boldsymbol{Q}_{k}p(\boldsymbol{Q}_{k}^T\boldsymbol{M}\boldsymbol{Q}_{k})\boldsymbol{Q}_{k}^T\bz=\boldsymbol{Q}_{k}p(\boldsymbol{U}_{k})\boldsymbol{Q}_{k}^T\bz.
\end{align*}
This implies for all $p\in\PP^{k-1}$
\begin{align}\label{eq:polyest}
\begin{split}
 |\boldsymbol{M}^{1/2}\bz&-\boldsymbol{Q}_{k}\boldsymbol{U}_{k}^{1/2}\boldsymbol{Q}_{k}^Tz|\\
 &\leq |\boldsymbol{M}^{1/2}\bz-\boldsymbol{Q}_{k}p(\boldsymbol{U}_{k})\boldsymbol{Q}_{k}^Tz|+|\boldsymbol{Q}_{k}p(\boldsymbol{U}_{k})\boldsymbol{Q}_{k}^Tz-\boldsymbol{Q}_{k}\boldsymbol{U}_{k}^{1/2}\boldsymbol{Q}_{k}^Tz|\\
 &\leq |\boldsymbol{M}^{1/2}\bz-p(\boldsymbol{M})\bz|+|\boldsymbol{Q}_{k}(p(\boldsymbol{U}_{k})-\boldsymbol{U}_{k}^{1/2})\boldsymbol{Q}_{k}^T\bz|\\
 &\leq \big(\norm{\boldsymbol{M}^{1/2}-p(\boldsymbol{M})}{2}+\norm{p(\boldsymbol{U}_{k})-\boldsymbol{U}_{k}^{1/2}}{2}\big)|\bz|.
\end{split}
 \end{align}
With $f(x):=\sqrt{(x+1)(\lambda_{\rm max}(\boldsymbol{M}) -\lambda_{\rm min}(\boldsymbol{M}))/2 + \lambda_{\rm min}(\boldsymbol{M})}$, the result~\cite[Lemma~4.14]{h2mat} proves
\begin{align*}
 \min_{p\in \PP^{k-1}}\norm{f-p}{L^\infty([-1,1])}\leq  \frac{2r^2}{r-1}r^{-k}\sup_{x\in\C_r}|f(x)|
\end{align*}
with $r>1$ from~\eqref{eq:r} and
\begin{align*}
 \C_r:=\set{x\in\C}{\Big(\frac{2\,{\rm real}(x)}{r+1/r}\Big)^2 + \Big(\frac{2\,{\rm imag}(x)}{r-1/r}\Big)^2\leq 1}.
\end{align*}
Since $x\in\C_r$ implies $|x|\leq r$, straightforward calculations show 
{\begin{align*}
\sup_{x\in\C_r}|f(x)|\leq \sup_{|x|\leq r}|f(x)|&= \sup_{|x|\leq r}\sqrt{|(x+1)(\lambda_{\rm max}(\boldsymbol{M}) -\lambda_{\rm min}(\boldsymbol{M}))/2 + \lambda_{\rm min}(\boldsymbol{M})|}\\
&\leq \sup_{|x|\leq r}\sqrt{(|x|+1)(\lambda_{\rm max}(\boldsymbol{M}) -\lambda_{\rm min}(\boldsymbol{M}))/2 + \lambda_{\rm min}(\boldsymbol{M})}\\
&\leq \sqrt{\lambda_{\rm max}(\boldsymbol{M}) +\lambda_{\rm min}(\boldsymbol{M})}\leq \sqrt{2\norm{\boldsymbol{M}}{2}},
\end{align*}}
which implies the estimate $\min_{p\in \PP^{k-1}}\norm{f-p}{L^\infty([-1,1])}\leq  {\sqrt{2\norm{\boldsymbol{M}}{2}}}\frac{2r^{{2}}}{r-1}r^{-k}$.
Hence, we obtain for $g(x):=\sqrt{x}$ (note that 
$x\mapsto (x+1)(\lambda_{\rm max}(\boldsymbol{M}) -\lambda_{\rm min}(\boldsymbol{M}))/2 + \lambda_{\rm min}(\boldsymbol{M})$ maps $[-1,1]$ onto $[\lambda_{\rm min}(\boldsymbol{M}),\lambda_{\rm max}(\boldsymbol{M})]$)
also
\begin{align}\label{eq:eigapprox}
 \min_{p\in \PP^{k-1}}\norm{g-p}{L^\infty([\lambda_{\rm min}(\boldsymbol{M}),\lambda_{\rm max}(\boldsymbol{M})])}\leq  {\sqrt{2\norm{\boldsymbol{M}}{2}}}\frac{2r^{{2}}}{r-1}r^{-k}.
\end{align}
Let $\boldsymbol{U}\in\R^{N\times N}$ denote the orthonormal matrix ($\boldsymbol{U}\boldsymbol{U}^T= \boldsymbol{I}$) that diagonalizes $\boldsymbol{M}$, i.e., 
$\boldsymbol{M}=\boldsymbol{U}\boldsymbol{D}\boldsymbol{U}^T$ with $\boldsymbol{D}\in\R^{N\times N}$ being the diagonal matrix containing the eigenvalues of $\boldsymbol{M}$.
There holds $\boldsymbol{M}^{1/2}=\boldsymbol{U}\boldsymbol{D}^{1/2}\boldsymbol{U}^T$ as well as $p(\boldsymbol{M})=\boldsymbol{U}p(\boldsymbol{D})\boldsymbol{U}^T$. This,~\eqref{eq:eigapprox}, and
invariance of the spectral norm $\norm{\cdot}{2}=\norm{\boldsymbol{U}(\cdot)\boldsymbol{U}^T}{2}$ show
\begin{align*}
 \min_{p\in \PP^{k-1}}\norm{\boldsymbol{M}^{1/2}-p(\boldsymbol{M})}{2}&= \min_{p\in \PP^{k-1}}\norm{\boldsymbol{D}^{1/2}-p(\boldsymbol{D})}{2}\\
 &= \min_{p\in \PP^{k-1}}\max_{1\leq i\leq N}|g(\boldsymbol{D}_{ii})-p(\boldsymbol{D}_{ii})|\leq {\sqrt{2\norm{\boldsymbol{M}}{2}}}\frac{2r^{{2}}}{r-1}r^{-k}.
\end{align*}%
Since $\boldsymbol{U}_k$ is an orthogonal projection of $\boldsymbol{M}$, we have $\lambda_{\rm min}(\boldsymbol{M})\leq \lambda_{\rm min}(\boldsymbol{U}_{k})\leq \lambda_{\rm max}(\boldsymbol{U}_{k})\leq \lambda_{\rm max}(\boldsymbol{M})$.
Thus, repeating the above argument for $\boldsymbol{U}_k$ instead of $\boldsymbol{M}$ yields
\begin{align*}
 \min_{p\in \PP^{k-1}}\Big(\norm{\boldsymbol{M}^{1/2}-p(\boldsymbol{M})}{2}+\norm{p(\boldsymbol{U}_{k})-\boldsymbol{U}_{k}^{1/2}}{2}\Big)
 \leq {\sqrt{2\norm{\boldsymbol{M}}{2}}}\frac{4r^{{2}}}{r-1}r^{-k}.
\end{align*}
This in combination with~\eqref{eq:polyest} and Lemma~\ref{lem:qq} conclude the proof.
\end{proof}

The next result quantifies the distance of ${\rm range}(\boldsymbol{Q}_j)$ to ${\rm range}(\boldsymbol{M}\boldsymbol{Q}_j)$ in terms of the projection 
$\boldsymbol{Q}_j\boldsymbol{Q}_j^T$ onto ${\rm range}(\boldsymbol{Q}_j)$. 

\begin{lemma}\label{lem:qr2}
{Assume the requirements of Lemma~\ref{lem:vandermondeplusqr}.} 
	Call Algorithm~\ref{alg1} with $\boldsymbol{M}$, $\bz$, and $k\in\N$ to compute $k_0\leq k$ as well as $\boldsymbol{Q}_j$ for all $1\leq j\leq k_0$.
	Let $\bq^j$ be the last column of $\boldsymbol{Q}_j$ for all $1\leq j\leq k_0$. There holds for all $1\leq j< k_0$ 
	\begin{align}\label{eq:errest}
	\norm{\boldsymbol{M}\boldsymbol{Q}_j-\boldsymbol{Q}_j\boldsymbol{Q}_j^T\boldsymbol{M}\boldsymbol{Q}_j}{2}&=|(\bq^{j+1})^T\boldsymbol{M}\bq^j|
	\end{align}
	as well as
	\begin{align*}
	\min_{1\leq i\leq j}\norm{\boldsymbol{M}\boldsymbol{Q}_i-\boldsymbol{Q}_i\boldsymbol{Q}_i^T\boldsymbol{M}\boldsymbol{Q}_i}{2}\leq \lambda_{\rm max}(\boldsymbol{M})C_\kappa\kappa^{(j+1)/2}.
	\end{align*}
\end{lemma}
\begin{proof}
Recall $\boldsymbol{Z},\boldsymbol{Q},\boldsymbol{R}$ satisfying $\boldsymbol{Z}=\boldsymbol{Q}\boldsymbol{R}$ from Lemma~\ref{lem:vandermondeplusqr} with $k$ replaced by $k_0$ in the call to Algorithm~\ref{alg1}.
	By Lemma~\ref{lem:qq}, $\bq^j$ coincides with the $j$-th column of $\boldsymbol{Q}$ for all $1\leq j\leq k_0$.
	Moreover, let $\br^j$ be the $j$-th column of $\boldsymbol{R}$ and let $\widetilde{\br}^j$ be the $j$-th column of $\boldsymbol{R}^{-1}$ from Lemma~\ref{lem:vandermondeplusqr}.
	All quantities are well-defined since $\boldsymbol{Z}$ has maximal rank $k_0$ by Lemma~\ref{lem:qq}.
For the first statement~\eqref{eq:errest}, note that ${\rm range}(\boldsymbol{M}\boldsymbol{Q}_j)\subseteq {\rm range}(\boldsymbol{Q}_{j+1})$ implies $\boldsymbol{M}\boldsymbol{Q}_j=\boldsymbol{Q}_{j+1}\boldsymbol{Q}_{j+1}^T\boldsymbol{M}\boldsymbol{Q}_j$.
Moreover,  due to Lemma~\ref{lem:qq}, we have $\boldsymbol{Q}_{j+1}=(\boldsymbol{Q}_j,\bq^{j+1})$, and hence $\boldsymbol{Q}_{j+1}\boldsymbol{Q}_{j+1}^T=\bq^{j+1}(\bq^{j+1})^T+\boldsymbol{Q}_{j}\boldsymbol{Q}_{j}^T$. Altogether, this shows
\begin{align*}
	\norm{\boldsymbol{M}\boldsymbol{Q}_j-\boldsymbol{Q}_j\boldsymbol{Q}_j^T\boldsymbol{M}\boldsymbol{Q}_j}{2}&=\norm{\boldsymbol{Q}_{j+1}\boldsymbol{Q}_{j+1}^T\boldsymbol{M}\boldsymbol{Q}_j-\boldsymbol{Q}_j\boldsymbol{Q}_j^T\boldsymbol{M}\boldsymbol{Q}_j}{2}\\
	&=\norm{(\bq^{j+1}(\bq^{j+1})^T+\boldsymbol{Q}_{j}\boldsymbol{Q}_{j}^T)\boldsymbol{M}\boldsymbol{Q}_j-\boldsymbol{Q}_j\boldsymbol{Q}_j^T\boldsymbol{M}\boldsymbol{Q}_j}{2}\\
	&= \norm{\bq^{j+1}(\bq^{j+1})^T\boldsymbol{M}\boldsymbol{Q}_j}{2}=|(\bq^{j+1})^T\boldsymbol{M}\bq^j|,
	\end{align*}	
	where the last step follows because $\bq^{j+1}$ is orthogonal to $\boldsymbol{M}\bq^i$, $i=1,\ldots,j-1$.
	This proves~\eqref{eq:errest}.
	
	To see the remaining statement, note that the definition of $\boldsymbol{Z}$ in~\eqref{eq:defZmatrix} implies
	\begin{align*}
	(\boldsymbol{M}\boldsymbol{Z})|_{\{1,\ldots,N\}\times \{j\}}=\lambda_1 \boldsymbol{Z}|_{\{1,\ldots,N\}\times \{j+1\}}=\lambda_1(\boldsymbol{Q}\boldsymbol{R})|_{\{1,\ldots,N\}\times \{j+1\}}
	=\lambda_1\boldsymbol{Q}\br^{j+1}
	\end{align*}
	as well as
	\begin{align*}
	 \bq^j = (\boldsymbol{Z}\boldsymbol{R}^{-1})|_{\{1,\ldots,N\}\times \{j\}} = \boldsymbol{Z}\widetilde{\br}^j.
	\end{align*}
	The last two identities, and the fact that $(\bq^{j+1})^T(\boldsymbol{M}\boldsymbol{Z})|_{\{1,\ldots,N\}\times \{i\}}=0$ for all $1\leq i\leq j-1$, imply
	\begin{align*}
	(\bq^{j+1})^T\boldsymbol{M}\bq^j&=(\bq^{j+1})^T\boldsymbol{M}\boldsymbol{Z}\widetilde{\br}^j=(\bq^{j+1})^T(\boldsymbol{M}\boldsymbol{Z})|_{\{1,\ldots,N\}\times \{j\}}
	(\widetilde{\br}^j)_j\\
	&=\lambda_1(\bq^{j+1})^T\boldsymbol{Q}\br^{j+1}
	(\widetilde{\br}^j)_j=\lambda_1(\br^{j+1})_{j+1}(\widetilde{\br}^j)_j.
	\end{align*}
	The triangular structure of $\boldsymbol{R}$ implies $(\boldsymbol{R}^{-1})_{jj}=1/\boldsymbol{R}_{jj}$ and hence $(\widetilde \br^j)_j=1/\boldsymbol{R}_{jj}$ (where $\boldsymbol{R}_{jj}\neq 0$  by assumption). This shows
	\begin{align}\label{eq:qr2}
	(\bq^{j+1})^T \boldsymbol{M}\bq^j=\lambda_1\frac{\boldsymbol{R}_{(j+1)(j+1)}}{\boldsymbol{R}_{jj}}.
	\end{align} 
	With Lemma~\ref{lem:vandermondeplusqr}, we have
	\begin{align}\label{eq:factors}
	\begin{split}
	\frac{\boldsymbol{R}_{(j+1)(j+1)}}{\boldsymbol{R}_{jj}}\frac{\boldsymbol{R}_{jj}}{\boldsymbol{R}_{(j-1)(j-1)}}\cdots \frac{\boldsymbol{R}_{22}}{\boldsymbol{R}_{11}}\boldsymbol{R}_{11}&=
	\boldsymbol{R}_{(j+1)(j+1)}\leq |\bz|C_\kappa^{j}\kappa^{(j+1)j/2}.
	\end{split}
	\end{align}
	Moreover, we know $\boldsymbol{R}_{11}=|\bq^1\boldsymbol{R}_{11}|=|\bz|$.
	This implies that at least one of the fractions on the left-hand side of~\eqref{eq:factors} must be smaller than the $j$-th root of the right hand side of~\eqref{eq:factors} divided by $|\bz|$ and hence
	\begin{align*}
	\min_{1\leq i\leq j}\frac{\boldsymbol{R}_{(i+1)(i+1)}}{\boldsymbol{R}_{ii}}\leq  C_\kappa\kappa^{(j+1)/2}.
	\end{align*}
	With this,~\eqref{eq:qr2}, and~\eqref{eq:errest}, 
	we obtain
	\begin{align*}
	\min_{1\leq i\leq j}\norm{\boldsymbol{M}\boldsymbol{Q}_i-\boldsymbol{Q}_i\boldsymbol{Q}_i^T\boldsymbol{M}\boldsymbol{Q}_i}{2}\leq \lambda_1C_\kappa\kappa^{(j+1)/2}.
	\end{align*}
	This concludes the proof.
\end{proof}

The following proposition is the main tool to prove Theorem~\ref{thm:sampleerror2}~(ii).
\begin{proposition}\label{prop:alg2}
Let $\bz\in\R^N$ and let $\boldsymbol{M}\in\R^{N\times N}$ be symmetric positive definite. 
	Call Algorithm~\ref{alg1} with $\boldsymbol{M}$, $\bz$, and $k\in\N$ to compute $k_0\leq k$ as well as $\boldsymbol{Q}_j$ for all $1\leq j\leq k_0$.
	Then, $\boldsymbol{U}_j:=\boldsymbol{Q}_j^T\boldsymbol{M}\boldsymbol{Q}_j$ satisfies the error bound
	\begin{align*}
	\frac{|\boldsymbol{M}^{1/2}\bz-\boldsymbol{Q}_j\boldsymbol{U}_j^{1/2}\boldsymbol{Q}_j^T\bz|}{|\bz|}\leq\begin{cases} \displaystyle
	\min\Big\{\frac{|(\bq^{j+1})^T \boldsymbol{M}\bq^j|}{\sqrt{\lambda_{\rm min}(\boldsymbol{M})}},3\sqrt{|(\bq^{j+1})^T \boldsymbol{M}\bq^j|}\Big\}&1\leq j<k_0,\\
	                                                                            0 &j=k_0\text{ and }k_0<k
	                                                                           \end{cases}
	\end{align*}
	and we have the a~priori estimate
	\begin{align*}
	\min_{1\leq i\leq j}|(\bq^{i+1})^T \boldsymbol{M}\bq^i|\leq \lambda_1C_\kappa\kappa^{(j+1)/2}
	\end{align*}
	for all $1\leq j< k_0$.
\end{proposition}
\begin{proof}
The case $k_0<k$ and $j=k_0$ is trivially covered in Lemma~\ref{lem:qr1.4}. For the other cases, let $\overline{\boldsymbol{Q}}\in\R^{N\times N}$ be orthonormal such that the first $j$ columns coincide with $\boldsymbol{Q}_j$, i.e., $\overline{\boldsymbol{Q}}=(\boldsymbol{Q}_j,\boldsymbol{Q}_\perp)$ for 
	some orthonormal $\boldsymbol{Q}_\perp\in\R^{N\times(N-j)}$. Then, we write
	\begin{align*}
	\overline{\boldsymbol{Q}}^T\boldsymbol{M}\overline{\boldsymbol{Q}}=\begin{pmatrix}
	\boldsymbol{U}_j & \boldsymbol{S}^T\\ \boldsymbol{S} &\boldsymbol{T}
	\end{pmatrix}
	\end{align*}
	for matrices $\boldsymbol{S}=\boldsymbol{Q}_\perp^T\boldsymbol{M}\boldsymbol{Q}_j\in\R^{(N-j)\times j}$, $\boldsymbol{T}\in\R^{(N-j)\times (N-j)}$. 
	This means that
	\begin{align*}
	\norm[\Big]{\overline{\boldsymbol{Q}}^T\boldsymbol{M}\overline{\boldsymbol{Q}}-\begin{pmatrix}
		\boldsymbol{U}_j & \boldsymbol{0}\\ \boldsymbol{0}  &\boldsymbol{T}
		\end{pmatrix}}{2}\leq \norm{\boldsymbol{S}}{2}.
	\end{align*}
	Lemma~\ref{lem:sqrtm} then implies
	\begin{align}\label{eq:key1}
	\norm[\Big]{(\overline{\boldsymbol{Q}}^T\boldsymbol{M}\overline{\boldsymbol{Q}})^{1/2}-\begin{pmatrix}
		\boldsymbol{U}_j^{1/2} & \boldsymbol{0}\\ \boldsymbol{0}  &\boldsymbol{T}^{1/2}
		\end{pmatrix}}{2}\leq \min\Big\{\lambda_{\rm min}(\boldsymbol{M})^{-1/2}\,\norm{\boldsymbol{S}}{2},3\sqrt{\norm{\boldsymbol{S}}{2}}\Big\}.
	\end{align}
	Since $\boldsymbol{I}-\boldsymbol{Q}_j\boldsymbol{Q}_j^T = \boldsymbol{Q}_\perp\boldsymbol{Q}_\perp^T$, we have
	\begin{align*}
	\norm{\boldsymbol{S}}{2}&=\norm{\boldsymbol{Q}_\perp^T\boldsymbol{M}\boldsymbol{Q}_j}{2}=\norm{\boldsymbol{Q}_\perp\boldsymbol{Q}_\perp^T\boldsymbol{M}\boldsymbol{Q}_j}{2}
	=\norm{\boldsymbol{M}\boldsymbol{Q}_j-\boldsymbol{Q}_j\boldsymbol{Q}_j^T\boldsymbol{M}\boldsymbol{Q}_j}{2}.
	\end{align*}
	With $(\overline{\boldsymbol{Q}}^T\boldsymbol{M}\overline{\boldsymbol{Q}})^{1/2}=\overline{\boldsymbol{Q}}^T\boldsymbol{M}^{1/2}\overline{\boldsymbol{Q}}$ and since the ranges of $\boldsymbol{Q}_j$ and $\boldsymbol{Q}_\perp$ are orthogonal, 
	we have $\boldsymbol{Q}_\perp^T\boldsymbol{Q}_j\boldsymbol{Q}_j^T=0$ and	
	\begin{align}\label{eq:key2}
	\begin{split}
	\norm{\boldsymbol{M}^{1/2}\boldsymbol{Q}_j\boldsymbol{Q}_j^T&-\boldsymbol{Q}_j(\boldsymbol{U}_j^{1/2})\boldsymbol{Q}_j^T}{2}\\
	&=
	\norm{\boldsymbol{M}^{1/2}\boldsymbol{Q}_j\boldsymbol{Q}_j^T-\boldsymbol{Q}_j(\boldsymbol{U}_j^{1/2})\boldsymbol{Q}_j^T\boldsymbol{Q}_j\boldsymbol{Q}_j^T-\boldsymbol{Q}_\perp(\boldsymbol{T}^{1/2})\boldsymbol{Q}_\perp^T\boldsymbol{Q}_j\boldsymbol{Q}_j^T}{2}\\
	&\leq 
	\norm{\boldsymbol{M}^{1/2}-\boldsymbol{Q}_j(\boldsymbol{U}_j^{1/2})\boldsymbol{Q}_j^T-\boldsymbol{Q}_\perp(\boldsymbol{T}^{1/2})\boldsymbol{Q}_\perp^T}{2}\\
	&=\norm{\overline{\boldsymbol{Q}}^T\Big(\boldsymbol{M}^{1/2}-\boldsymbol{Q}_j(\boldsymbol{U}_j^{1/2})\boldsymbol{Q}_j^T-\boldsymbol{Q}_\perp(\boldsymbol{T}^{1/2})\boldsymbol{Q}_\perp^T\Big)\overline{\boldsymbol{Q}}}{2}\\
	&=\norm[\Big]{(\overline{\boldsymbol{Q}}^T\boldsymbol{M}\overline{\boldsymbol{Q}})^{1/2}-\begin{pmatrix}
		\boldsymbol{U}_j^{1/2} & \boldsymbol{0}\\ \boldsymbol{0}  &\boldsymbol{T}^{1/2}
		\end{pmatrix}}{2}.
		\end{split}
	\end{align}
	The combination of~\eqref{eq:key1} and~\eqref{eq:key2} shows
	\begin{align*}
	\norm{\boldsymbol{M}^{1/2}\boldsymbol{Q}_j\boldsymbol{Q}_j^T-\boldsymbol{Q}_j(\boldsymbol{U}_j^{1/2})\boldsymbol{Q}_j^T}{2}&\leq\min\Big\{\lambda_{\rm min}(\boldsymbol{M})^{-1/2}\,
	\norm{\boldsymbol{S}}{2},3\sqrt{\norm{\boldsymbol{S}}{2}}\Big\}.
	\end{align*}
	We conclude the proof with $\bz=\boldsymbol{Q}_j\boldsymbol{Q}_j^T\bz$ due to $\bz\in {\rm range}(\boldsymbol{Q}_j)$ and Lemma~\ref{lem:qr2}.
\end{proof}

 \section{Lemma for the proof of Theorem~\ref{thm:sampleerror}}
 The following lemma is the main tool for the proof of Theorem~\ref{thm:sampleerror}.
\begin{lemma}\label{lem:it}
Let $\boldsymbol{M}\in \R^{N\times N}$ be symmetric positive definite. 
Then, the iteration~\eqref{eq:it2} with initial values $\boldsymbol{A}_0= s\boldsymbol{M}$ and 
$\boldsymbol{B}_0=\boldsymbol{I}$ satisfies
\begin{align}\label{eq:conv}
 \norm{\boldsymbol{M}^{1/2}-s^{-1/2}\boldsymbol{A}_k}{2}\leq s^{-1/2}\, (\max\{|1-s\lambda_{\rm max}(\boldsymbol{M})|,|1-s\lambda_{\rm min}(\boldsymbol{M})|\})^{2^k}
\end{align}
for all $k\in\N$ and all $s>0$.
The minimum bound is attained at $s= 2/(\lambda_{\rm min}(\boldsymbol{M})+\lambda_{\rm max}(\boldsymbol{M}))$ such that 
$\max\{|1-s\lambda_{\rm max}(\boldsymbol{M})|,|1-s\lambda_{\rm min}(\boldsymbol{M})|\}=1-2\lambda_{\rm min}(\boldsymbol{M})/(\lambda_{\rm min}(\boldsymbol{M})+\lambda_{\rm max}(\boldsymbol{M}))$.
\end{lemma}
\begin{proof}
Straightforward calculations show
\begin{align*}
 \max\big\{\norm{\boldsymbol{M}^{1/2}-\boldsymbol{A}_k}{2},\norm{\boldsymbol{M}^{-1/2}-\boldsymbol{B}_k}{2}\Big\}&= 
\norm[\Big]{\begin{pmatrix} 0 & \boldsymbol{M}^{1/2}\\ \boldsymbol{M}^{-1/2} & 0\end{pmatrix} - \begin{pmatrix} 
 0 & \boldsymbol{A}_k\\ \boldsymbol{B}_k & 0\end{pmatrix}}{2}.
\end{align*}
The result~\cite[Theorem~5.2]{sqit2} shows $\norm{\boldsymbol{I}-\boldsymbol{X}_{n}^2}{2}<\norm{\boldsymbol{I}-\boldsymbol{X}_{0}^2}{2}^{(e_1+e_2+1)^n}$ for all $n\in\N$, where
$\boldsymbol{X}_{n+1}=-\boldsymbol{X}_{n}P_{e_1e_2}(\boldsymbol{I}-\boldsymbol{X}_{n}^2)Q_{e_1e_2}^{-1}(\boldsymbol{I}-\boldsymbol{X}_{n}^2)$ and $\boldsymbol{X}_{0}$ has no purely imaginary
eigenvalues. Here $P_{e_1e_2}/Q_{e_1e_2}$ is the $(e_1/e_2)$-Pad\'e approximant to $(1-x)^{-1/2}$. We obtain from~\cite[Table~1]{sqit2} that for $e_1=1$ and $e_2=0$, $\boldsymbol{X}_n$ satisfies
the Schultz iteration~\eqref{eq:schultz} and thus we may use the result with 
\begin{align*}
 \boldsymbol{X}_0=\begin{pmatrix} 0 & \boldsymbol{M}\\ \boldsymbol{I} & 0\end{pmatrix}
\end{align*}
 to show
\begin{align*}
\norm[\Big]{\begin{pmatrix} 0 & \boldsymbol{M}^{1/2}\\ \boldsymbol{M}^{-1/2} & 0\end{pmatrix} - \begin{pmatrix} 
 0 & \boldsymbol{A}_k\\ \boldsymbol{B}_k & 0\end{pmatrix}}{2}
 <
 \norm[\Big]{\begin{pmatrix} \boldsymbol{I}-\boldsymbol{M} & 0\\ 0 & \boldsymbol{I}-\boldsymbol{M}\end{pmatrix}}{2}^{2^k}=\norm{\boldsymbol{I}-\boldsymbol{M}}{2}^{2^k}
\end{align*}
for all $k\in\N$.
By scaling of $\boldsymbol{M}$, we may minimize the right-hand side. To that end, we observe that the spectrum satisfies $\sigma(\boldsymbol{I}-s\boldsymbol{M})\subset [1-s\lambda_{\rm max}(\boldsymbol{M}),1-s\lambda_{\rm min}(\boldsymbol{M})]$.
The fact $\norm{\boldsymbol{I}-s\boldsymbol{M}}{2}\leq \max\{|1-s\lambda_{\rm max}(\boldsymbol{M})|,|1-s\lambda_{\rm min}(\boldsymbol{M})|\}$ proves~\eqref{eq:conv}.
A straightforward optimization of $s>0$ concludes the proof.
\end{proof}

%

\appendix
\section{Proof of Lemma~\ref{lem:cov}}\label{app:lemma1}

The 
following lemma is an elementary statement on holomorphic functions
\begin{lemma}\label{lem:holo}
 Let $f\colon O\to \C$ be a continuous function on the domain $O\subset \C^n$ which is 
holomorphic in $O$ in all variables $\bx_i$, $i\in\{1,\ldots,n\}$, i.e.,
 \begin{align*}
  \bx_i\mapsto f(\bx_1,\ldots,\bx_i,\ldots,\bx_n)
 \end{align*}
is holomorphic in $\set{\bx_i\in \C}{(\bx_1,\ldots,\bx_i,\ldots,\bx_n)\in O}$ for all 
$\bx_1,\ldots,\bx_{i-1},\bx_{i+1},\ldots,\bx_n\in\C$. Then, for all multi-indices 
$\alpha\in\N_0^n$, the function
$\partial_\bx^\alpha f$ is holomorphic in $O$ in all variables $\bx_i$, 
$i\in\{1,\ldots,n\}$ as defined above.
\end{lemma}
\begin{proof}
 The result is proved by induction on $|\alpha|_1$.
 Obviously, for $|\alpha|_1=0$, $\partial_\bx^\alpha f=f$ and the statement is true.
 Assume the statement holds for all $|\alpha|_1\leq k$ and choose some 
$\alpha\in\N_0^n$ with $|\alpha|_1=k+1$. Then, we have for some 
$i\in\{1,\ldots,n\}$ and some $\alpha_0\in\N_0^n$ with $|\alpha_0|_1=k$ that
 \begin{align*}
  \partial_\bx^\alpha f =\partial_{\bx_i}\partial_\bx^{\alpha_0}f.
 \end{align*}
Since, $\partial_\bx^{\alpha_0}f$ is holomorphic in $O$ in all variables by the 
induction hypothesis, obviously $\partial_\bx^\alpha f$ is holomorphic in $O$ at 
least in $\bx_i$ (derivatives of holomorphic functions are holomorphic).
To prove the statement for all other variables,
we may employ Cauchy's integral formula to obtain
\begin{align*}
  \partial_\bx^\alpha f(\bx)=\partial_{\bx_i}\partial_\bx^{\alpha_0}f = \frac{1}{2\pi i}\int_{\partial B_\eps(\bx_i)} \frac{ 
\partial_\bx^{\alpha_0}f(\bx_1,\ldots,\bx_{i-1},\bz,\bx_{i+1},\ldots,\bx_n)}{(\bz-\bx_i)^2}\,{\rm d}\bz,
\end{align*}
for some $\eps>0$ with $B_\eps(\bx_i)\subset \C$ being the ball with radius 
$\eps$. The integrand is holomorphic in all variables $\bx_j$, $j\neq i$. Hence, 
we conclude that $\partial_\bx^\alpha f(\bx)$ is holomorphic
in all variables and prove the assertion. 
\end{proof}

The following result is elementary but technical.
\begin{lemma}\label{lem:man}
For $n,p\in\N$, define the set $M:=\set{\bx\in\C^{n}}{{\rm real}(\sum_{i=1}^n \bx_i^p) \leq 0}$. Then, there holds $(\R^n)_+:=\set{\bx\in \R^n\setminus\{0\}}{\bx_i\geq 0}\cap M=\emptyset$ and
\begin{align*}
 {\rm dist}(M,\bx)\geq |\sin(\frac{\pi}{2p})||\bx|\quad\text{for all }\bx\in(\R^n)_+.
\end{align*}
\end{lemma}
\begin{proof}
 Let $\bx\in(\R^n)_+$, then we have $\sum_{i=1}^n \bx_i^p>0$ and hence $\bx\notin M$. It is easy to see that the cone 
 $C_p:=\set{r\exp(i\phi)}{r>0,\,\phi\in(-\frac{\pi}{2p},\frac{\pi}{2p})}\subset \C$
 satisfies ${\rm real}(x^p)>0$ for all $x\in C_p$. Thus, we have that 
 \begin{align*}
C_p^n:=\Big(\prod_{i=1}^n(\{0\}\cup C_p)\Big)\setminus \{0\}\subset \C^n
\end{align*}
satisfies $C_p^n\cap M=\emptyset$. 
\begin{figure}
\psfrag{y}{\tiny $\partial C_p$}
\psfrag{x}{\tiny $x$}
\psfrag{0}{\tiny $0$}
\psfrag{the}[cc][cc]{\tiny $\frac{\pi}{2p}$}
\centering
\includegraphics[clip,width=0.5\textwidth]{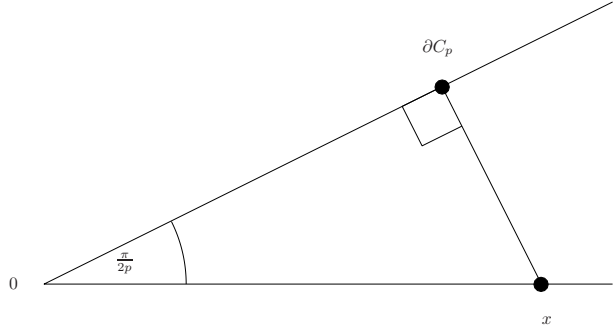} 
\caption{The situation of the proof of Lemma~\ref{lem:man}. The distance between $\partial C_p$ and $x$ is $x\sin(\pi/(2p))$.}
\label{fig:trig}
\end{figure}
Moreover, a simple geometric argument {(see Figure~\ref{fig:trig})} shows that all $x>0$ satisfy
\begin{align*}
 {\rm dist}(x,\partial C_p) = x \sin(\pi/(2p)).
\end{align*}
Since $(\R^n)_+\subseteq C_p^n$, this implies
 \begin{align*}
 {\rm dist}(M,\bx)\geq {\rm dist}(\partial C_p^n,\bx)=\Big(\sum_{i=1}^n \bx_i^2\sin(\pi/(2p))^2\Big)^{1/2}=|\sin(\pi/(2p))||\bx|.
 \end{align*}
This concludes the proof.
\end{proof}

Products of asymptotically smooth functions are again asymptotically smooth. This is shown in the next lemma.
\begin{lemma}\label{lem:prod}
 Given two functions $f,g\colon D\times D\to \R$ which are asymptotically smooth~\eqref{eq:as}. Then, also their product $fg$ satisfies~\eqref{eq:as}.
\end{lemma}
\begin{proof} 
To simplify the notation, we consider $f,g$ as functions of one variable $\bz=(\bx,\by)\in D\times D\subset \R^{2d}$.
For multi-indices $\alpha,\beta\in\N^{2d}$, define
\begin{align*}
 \binom{\alpha}{\beta}:=\prod_{i=1}^{2d}\binom{\alpha_i}{\beta_i}.
\end{align*}
Note that there holds $\binom{\alpha}{\beta}\leq \binom{|\alpha|_1}{|\beta|_1}$. This follows from the basic combinatorial fact that the number of possible choices of $\beta_i$ elements out of a set of $\alpha_i$ elements
for all $i=1,\ldots,2d$
is smaller than the number of choices of $|\beta|_1$ elements out of a set of $|\alpha|_1$ elements.

 The Leibniz formula together with the definition of asymptotically smooth function~\eqref{eq:as} show for $\alpha\in\N^{2d}$
 \begin{align*}
  |\partial_\bz^\alpha (fg)|(\bz)&\leq \sum_{\beta\in\N^{2d}_0\atop \beta\leq \alpha}\binom{\alpha}{\beta}|\partial_\bz^{\beta}f|(\bz)|\partial_\bz^{\alpha-\beta} g|(\bz)\\
  &\leq \sum_{\beta\in\N^{2d}_0\atop \beta\leq \alpha}\binom{|\alpha|_1}{|\beta|_1} c_1(c_2|\bx-\by|)^{-|\beta|_1}|\beta|_1!c_1(c_2|\bx-\by|)^{{-|\alpha|_1+|\beta|_1}}(|\alpha|_1-|\beta|_1)!\\
  &\leq\sum_{\beta\in\N^{2d}_0\atop \beta\leq \alpha} c_1^2(c_2|\bx-\by|)^{-|\alpha|_1}|\alpha|_1!\\
  &\leq (|\alpha|_1+1)^{2d} c_1^2(c_2|\bx-\by|)^{-|\alpha|_1}|\alpha|_1!\\
  &\lesssim c_1^2(\widetilde c_2|\bx-\by|)^{-|\alpha|_1}|\alpha|_1!,
 \end{align*}
 where we used $(|\alpha|_1+1)^{2d}\leq (2d\exp(2d))^{|\alpha|_1}$ and $\widetilde c_2=c_2/(2d\exp(2d))$.
This concludes the proof.
\end{proof}

The final lemma of this section proves the concatenations of certain asymptotically smooth functions are asymptotically smooth.
\begin{lemma}\label{lem:circ}
 Let  $g\colon D\times D\to \R$ be asymptotically smooth~\eqref{eq:as} with constants $c_1,c_2>0$.
 \begin{itemize}
  \item[(i)] If $c_g:=\sup_{\bx\in D\times D}g(\bx)<\infty$. Then, $\exp\circ g$ satisfies~\eqref{eq:as} 
 with constants $\widetilde c_1:= \exp(c_g)$ and $\widetilde c_2:= c_2/(2\max\{1,c_1\})$. 
 \item[(ii)] If $g$ satisfies $\partial_\bx^\alpha \partial_\by^\alpha g(\bx,\by)\leq C_g$ for all $\alpha,\beta\in\N_0^d$ and some $C_g<\infty$
 as well as $g(\bx,\by)\geq C_g^{-1}|\bx-\by|$,
 then, $g^{1/q}$ satisfies~\eqref{eq:as}  with $\widetilde \varrho_1=1/2$ and $\widetilde \varrho_2 =C_g^{-1}$ for all $q\in\N$.
 \item[(iii)] If $g$ satisfies the assumptions from {\rm (ii)} and additionally $g(\bx,\by)\geq c_0>0$ for all $\bx,\by\in D$, then $g^{-1/q}$ satisfies~\eqref{eq:as} for all $q\in\N$.
 \end{itemize}
\end{lemma}
\begin{proof}
 To simplify the notation, we consider $g$ as a function of one variable $\bz=(\bx,\by)\in D\times D\subset \R^{2d}$.
 Define the set of all partitions of $\{1,\ldots,n\}$ as 
 \begin{align*}
 \Pi(n):=\set{P\subseteq 2^{\{1,\ldots,n\}}}{{S\cap S^\prime=\emptyset\text{ or }S=S^\prime\text{ for all }}S,S^\prime\in P,\, \bigcup_{S\in P} S=\{1,\ldots,n\}}.
 \end{align*}
 {For a multi-index $\alpha\in\N^{2d}$, we define $\widetilde \alpha\in \{1,\ldots,2d\}^n$ by $\widetilde \alpha_i=j$ for all
 $1+\sum_{k=1}^{j-1}\alpha_k\leq i\leq \sum_{k=1}^j\alpha_k$ and all $1\leq j\leq 2d$ (e.g., $\alpha=(2,3,1,1)$ yields $\widetilde \alpha=(1,1,2,2,2,3,4)$). With $n=|\alpha|_1$ and some $S\in P\in \Pi(n)$, we define 
 \begin{align*}
  \partial_{\bz}^S g(\bz)=\Big(\prod_{i\in S}\partial_{\bz_{\widetilde \alpha_i}}\Big)g(\bz).
 \end{align*}
(the definition implies $\partial_{\bz}^{\{1,\ldots,n\}}g(\bz)=\partial_\bz^\alpha g(\bz)$.)}
 With those definitions and given a function $f\colon \R\to\R$, Fa\`a di Bruno's formula reads for a multi-index $\alpha\in\N^{2d}$
 \begin{align}\label{eq:faa}
  \partial_{\bz}^\alpha(f\circ g)(\bz) = \sum_{P\in \Pi(|\alpha|_1)} (\partial_x^{|P|} f)\circ g(\bz)\prod_{S\in P}{\partial_{\bz}^S} g(\bz).
 \end{align} 
 For~(i), Fa\`a di Bruno's formula~\eqref{eq:faa} {and $\partial_x^{|P|}\exp =\exp $} show for all multi indices $\alpha\in\N^{2d}$ with $n=|\alpha|_1$ that
 \begin{align*}
  \partial_\bx^\alpha (\exp\circ g)(\bz){ = \sum_{P\in \Pi(n)} \exp\circ g(\bz)\,\prod_{S\in P}\partial_{\bz}^Sg(\bz)}.
 \end{align*}
 The definition of asymptotically smooth~\eqref{eq:as} and {$\norm{g}{L^\infty(D\times D)}=c_g$} imply
\begin{align*}
 |\partial_\bx^\alpha (\exp\circ g(\bz))|&\leq \exp(c_g)\sum_{P\in \Pi(n)}\prod_{S\in P}c_1(c_2|\bx-\by|)^{-|S|}|S|!\\
 &{\leq \exp(c_g)\sum_{P\in \Pi(n)}(c_2|\bx-\by|)^{-\sum_{S\in P}|S|}\,c_1^{|P|}\prod_{S\in P}|S|!}\\
 &\leq \exp(c_g)
 \max\{1,c_1\}^n(c_2|\bx-\by|)^{-n}\sum_{P\in \Pi(n)}\prod_{S\in P}|S|!.
\end{align*}
{With $f(x):=(1-x)^{-1}$, $x\in\R\setminus\{1\}$, we have $\partial_x^k f(x) = k!(1-x)^{-1-k}$.} Hence, the last factor can be written, using Fa\`a di Bruno's formula again, as
\begin{align*}
 \sum_{P\in \Pi(n)}\prod_{S\in P}|S|!= \sum_{P\in \Pi(n)}\exp\circ f(0)\prod_{S\in P}\partial_\bx^{|S|} f(0) = \partial_x^n(\exp\circ f)(0).
\end{align*}
As the function $h(x):=\exp((1-x)^{-1})$, $x\in\C$ is holomorphic at least for $|x|<1$, Cauchy's integral formula shows
\begin{align*}
 |\partial_x^n h(0)|=\frac{n!}{2\pi}\Big|\int_{|z|=1/2} \frac{h(z)}{z^{n+1}}\,{\rm d}\bz\Big|\leq n! 2^{n} \exp(2).
\end{align*}
Altogether, we conclude the proof of~(i) by
\begin{align*}
 |\partial_\bz^\alpha (\exp\circ g(\bz))|&\leq \exp(c_g) \Big(\frac{c_2}{2\max\{1,c_1\}}\,|\bx-\by|\Big)^{-n}n!.
\end{align*}

For~(ii), Fa\`a di Bruno's formula~\eqref{eq:faa} shows again for $q>1$
\begin{align*}
 |\partial_\bz^\alpha ( g^{1/q})(\bz)| &\leq  \sum_{P\in \Pi(n)} |P|!|g(\bz)|^{1/q-|P|}\,\prod_{S\in P}C_g\leq C_g^n |\bx-\by|^{-|n|}\sum_{P\in \Pi(n)} |P|!,
\end{align*}
where we used $f(x):=x^{1/q}$ and $|\partial_x^{|P|}f(x)|=|(1/q)(1/q-1)(1/q-2)\cdots (1/q-|P|+1)||x|^{1/q-|P|}\leq |P|!|x|^{1/q-|P|}$ as well as the boundedness assumption on the derivatives of $g$ from~(ii).
With $r(x):=\exp(x)-1$ {and $f(x):=(1-x)^{-1}$}, $x\in\R$, the last factor satisfies
\begin{align*}
 \sum_{P\in \Pi(n)} |P|! = \sum_{P\in \Pi(n)} (\partial_x^{|P|} f)\circ r(0)\prod_{S\in P}(\partial_x^{|S|}r)(0)=\partial_x^n(f\circ r)(0).
\end{align*}
The function $h(x):=f\circ r(x)= (2-\exp(x))^{-1}$, $x\in\C$ is holomorphic at least for $|x|\leq 1/2$. As above, this implies
\begin{align*}
 \partial_x^n(f\circ r)(0)\leq n!2^n
\end{align*}
and thus concludes the proof of~(ii).

For~(iii), we conclude the proof as for~(ii) by use of the estimate $g(z)^{{-}1/q-|P|}\leq c_0^{-1-n}$.
\end{proof}

At last, we are ready to prove Lemma~\ref{lem:cov} {which states that the covariance functions from~\eqref{eq:matern} and~\eqref{eq:nonstat} are asymptotically smooth~\eqref{eq:as}.}
\begin{proof}[Proof of Lemma~\ref{lem:cov}]
To see~\eqref{eq:as}, consider $\varrho(\cdot,\cdot)$ from~\eqref{eq:matern}. We define for complex variables $\bx_i,\by_i\in\C$
\begin{align*}
 d(\bx-\by)= \Big(\sum_{i=1}^d (\bx_i-\by_i)^p\Big)^{1/p}\in\C,
\end{align*}
{whenever $(\cdot)^{1/p}$ is defined in $\C$.}
and consider $\widetilde \varrho(\bx,\by)$ which is $\varrho(\bx,\by)$ from~\eqref{eq:matern} but with $d(\bx-\by)$ instead of $|\bx-\by|_p$.
With the notation of Lemma~\ref{lem:man}, the above sum has positive real part in $O:=\set{(\bx,\by)\in\C^{2d}}{\bx-\by\notin M}$. Thus, the function
$(\bx,\by)\mapsto d(\bx-\by)$ is holomorphic in each variable in $O$. Since for $a>0$, $\bx\mapsto \bx^\mu K_\mu(ax)$ is a holomorphic function on $\C\setminus(\R_-\cup \{0\})$,
and $d(\bx-\by)$ has positive real part, we deduce that $(\bx,\by)\mapsto  \widetilde \varrho(\bx,\by)$ is holomorphic in each variable in $O$.
Thus, Lemma~\ref{lem:holo} 
proves that $\partial_\bx^\alpha \partial_\by^\beta \widetilde \varrho(\bx,\by)$ is holomorphic in $O$ in 
all variables $\bx_i$ and $\by_i$. Therefore, Cauchy's integral formula applied in all 
variables shows
\begin{align*}
 \partial_\bx^\alpha \partial_\by^\beta \widetilde \varrho(\bx,\by)
 &=\frac{\prod_{i=1}^d\alpha_i ! 
\beta_i!}{(2\pi i)^{2d}} 
 \int_{\partial B_{\bx,1}}\ldots \int_{\partial B_{\bx,d}}\int_{\partial 
B_{\by,1}}\ldots \int_{\partial 
B_{\by,d}}\frac{\widetilde \varrho(s,t)}{\prod_{i=1}^d(s_i-\bx_i)^{\alpha_i+1}(t_i-\by_i)^{\beta_i+1}}\,{\rm d}t\,{\rm d}s.
\end{align*}
The balls $B_{\bx,i}$ and $B_{\by,i}$ have to be chosen such that $\prod_{i=1}^d 
B_{\bx,i}\times \prod_{i=1}^d B_{\by,i}\subset O$. With Lemma~\ref{lem:man}, and for $(\bx,\by)\in \R^{2d}$ such that $\bx-\by\in (\R^n)_+$ (note that Lemma~\ref{lem:man} implies $(\bx,\by)\in O$), this can be achieved by setting
$B_{\bx,i}:= B_\eps(\bx_i)$ and 
$B_{\by,i}:=B_\eps(\by_i)$ with $\eps:=\sin(\pi/(2p))|\bx-\by|/(2d+1)$.
From this, we obtain the estimate
\begin{align}\label{eq:firstest}
|\partial_\bx^\alpha \partial_\by^\beta \varrho(\bx,\by)|= |\partial_\bx^\alpha \partial_\by^\beta \widetilde \varrho(\bx,\by)|\lesssim \frac{\alpha!\beta! 
(2d+1)^{|\alpha|_1+|\beta|_1}}{|\bx-\by|^{|\alpha|_1+|\beta|_1}}\max_{(s,t)\in D\times D} 
|\widetilde\varrho(s,t)|
\end{align}
for all $(\bx,\by)\in \R^{2d}$ such that $\bx-\by\in (\R^n)_+$, where the first equality follows from $d(\bx-\by)=|\bx-\by|_p$ for all $\bx-\by\in (\R^n)_+$. 
To remove the restriction $\bx-\by\in(\R^n)_+$, consider $b\in \{0,1\}^d$ and define the function
\begin{align*}
 F_b(\bx,\by):=((-1)^{b_1}\bx_1,\ldots,(-1)^{b_d}\bx_d,(-1)^{b_1}\by_1,\ldots,(-1)^{b_d}\by_d).
\end{align*}
Since we consider $\varrho(\cdot,\cdot)$ from~\eqref{eq:matern}, there holds $\varrho\circ F_b=\varrho$. Since for all $\bx,\by\in \R^{d}$ with $\bx\neq \by$, there exists some $b\in \{0,1\}^d$ such that
$(\bx_b,\by_b):=F_b(\bx,\by)$ satisfies $\bx_b-\by_b\in (\R^n)_+$, we prove~\eqref{eq:firstest} for all $\bx,\by\in \R^d$ with $\bx\neq \by$.
Finally, the fact $\alpha!\beta! \leq |\alpha+\beta|_1!$, proves that $\varrho(\cdot,\cdot)$ from~\eqref{eq:matern} is asymptotically smooth~\eqref{eq:as}.

Next, consider the covariance function $\varrho(\cdot,\cdot)$ from~\eqref{eq:nonstat}. By definition $\boldsymbol{\Sigma}_\bx$ is continuous on $\overline{D}$. Hence, ${\rm det}(\boldsymbol{\Sigma}_\bx)\geq c_0>0$ for all $\bx\in D$.
The assumption~\eqref{eq:boundderiv} implies that also ${\rm det}(\boldsymbol{\Sigma}_\bx)$ has bounded derivatives in the sense of~\eqref{eq:boundderiv} (since ${\rm det}(\boldsymbol{\Sigma}_\bx)$ is a polynomial in the
matrix entries of $\boldsymbol{\Sigma}_\bx$). Thus, Lemma~\ref{lem:circ}
shows that the functions $(\bx,\by)\mapsto {\rm det}(\boldsymbol{\Sigma}_\bx)^{1/4}$, $(\bx,\by)\mapsto {\rm det}(\boldsymbol{\Sigma}_\by)^{1/4}$, and $(\bx,\by)\mapsto {\rm det}(\boldsymbol{\Sigma}_\bx+\boldsymbol{\Sigma}_\by)^{-q}$, $q\in\{1/2,1\}$ satisfy~\eqref{eq:as}.
With $\boldsymbol{\Sigma}_\bx$, also all functions $\widetilde {\boldsymbol{\Sigma}}_\bx$ defined by considering only sub-matrices of $\boldsymbol{\Sigma}_\bx$ satisfy~\eqref{eq:boundderiv}. Thus, Cramer's rule and Lemma~\ref{lem:prod} show
that the map $(\bx,\by)\mapsto ((\boldsymbol{\Sigma}_\bx+\boldsymbol{\Sigma}_\by)^{-1})_{i,j}$ for all $i,j\in\{1,\ldots,d\}$ satisfies~\eqref{eq:as}. From this, we conclude (again with Lemma~\ref{lem:prod}), that
$(\bx,\by)\mapsto (\bx-\by)^T(\boldsymbol{\Sigma}_\bx+\boldsymbol{\Sigma}_\by)^{-1}(\bx-\by)$ as sum and product of asymptotically smooth functions is asymptotically smooth~\eqref{eq:as}.  Finally, Lemma~\ref{lem:circ} shows that $\varrho(\bx,\by)$ satisfies~\eqref{eq:as}.
This concludes the proof.
\end{proof}

\section{Proof of Proposition~\ref{prop:h2}}\label{app:proph2}
The following lemmas state facts about the $H^2$-matrix block partitioning, 
which are well-known but cannot be found explicitly in the literature.

\begin{lemma}\label{lem:bb}
Under Assumption~\ref{ass:approxu}, there exists a constant $C_{B}>0$ which depends only 
on $d$, $C_{\rm u}$, $D$, and $B_{X_{\rm root}}$
such that all $X\in \T_{\rm cl}$ satisfy
\begin{subequations}\label{eq:shape}
\begin{align}\label{eq:shape1}
 {\rm diam}(B_X)^d &\leq C_B |B_X|,\\
 {C_B^{-1}N|B_X|-1}&{\leq  |X|\leq 1+C_BN|B_X|},\label{eq:shape2}\\
 |B_X|&=  2^{-{\rm level}(X)}|B_{X_{\rm root}}|.\label{eq:shape3}
\end{align} 
\end{subequations}
Moreover, all $(X,Y)\in\T$ satisfy
\begin{align}\label{eq:eq}
 C_{BB}^{-1} {\rm diam}(B_X)\leq {\rm diam}(B_Y)\leq C_{BB} {\rm diam}(B_X),
\end{align}
where $C_{BB}>0$ depends only on $C_B$, $C_{\rm leaf}$, and $D$.
\end{lemma}
\begin{proof}
The first estimate~\eqref{eq:shape1} follows from the fact that always the 
longest edge of a bounding box is halved. This means that the ratio $L_{\rm max}/L_{\rm min}$
of the maximal and the minimal side length of a bounding box $B_X$ stays bounded in 
terms of the corresponding ratio for $B_{X_{\rm root}}$. Therefore, we have
\begin{align*}
 {\rm diam}(B_X)^d\leq (\sqrt{d} L_{\rm max})^d\lesssim d^{d/2} L_{\rm min}^d\leq d^{d/2}|B_X|.
\end{align*}

{To see the second estimate~\eqref{eq:shape2}, consider a given bounding box $B$ with 
side lengths $L_1,\ldots, L_d$. Due to Assumption~\ref{ass:approxu} the balls $Q_\bx$ with centre $\bx$ and radius $C_{\rm u}^{-1}N^{-1/d}/2$ for all $\bx\in\NN$ do not overlap.
All balls $Q_\bx$ with $\bx\in B$ are containted in a box with sidelengths $L_{\rm max}+C_{\rm u}^{-1}N^{-1/d}$.
Thus, the number $m_B$ of $\bx\in\NN$ contained in $B$ can be bounded by
 \begin{align*}
  m_B\lesssim \frac{(L_{\rm max}+C_{\rm u}^{-1}N^{-1/d})^d}{C_{\rm u}^{-d}/(N2^d)}\leq \frac{dL_{\rm max}^d}{C_{\rm u}^{-d}/(N2^d)}+d2^d.
 \end{align*}
 Since $m_B\leq 1$ if $L_{\rm max}< C_{\rm u}^{-1}N^{-1/d}/2$ and since $L_{\rm max}^d\simeq |B|$, we may improve the estimate to
 \begin{align*}
  m_B\leq 1+ C_B|B|N,
 \end{align*}
where $C_B$ depends only on $d$ and $C_{\rm u}$. On the other hand, Assumption~\ref{ass:approxu} implies that any ball with radius $C_{\rm u}N^{-1/d}$ contains at least one point $\bx\in\NN$.
Since each such ball fits inside a box with sidelength $2C_{\rm u}N^{-1/d}$, we obtain
\begin{align*}
 m_B\gtrsim \Big\lfloor\frac{L_{\rm min}^d}{2^dC_{\rm u}^{-d}/N}\Big\rfloor
\end{align*}
points of $\NN$. This allows us to estimate $m_B\geq C_{B}^{-1}|B|N-1$ and conclude~\eqref{eq:shape2}.
The estimate~\eqref{eq:shape3} follows from the fact $|B_X|=|B_{X^\prime}|/2$ 
for all $X\in{\rm sons}(X^\prime)$.
 For~\eqref{eq:eq}, we observe with~\eqref{eq:shape2} that
 \begin{align*}
  \frac{|X|-1}{N}{\lesssim |B_X|\lesssim  \frac{|X|+1}{N}}
 \end{align*}
for all $X\in\T_{\rm cl}$ with hidden constants depending only on $C_B$. Thus, with~\eqref{eq:shape3}, we have for all $X\in \T_{\rm cl}$ with $X\in{\rm sons}(X^\prime)$
that
\begin{align*}
  2^{-{\rm level}(X)}\geq  2^{-{\rm level}(X^\prime)}/2\simeq |B_{X^\prime}|\gtrsim C_{\rm leaf}/N.
\end{align*}
Moreover, if additionally ${\rm sons}(X)=\emptyset$, we have even $2^{-{\rm level}(X)}\simeq |B_x|\lesssim C_{\rm leaf}/N$.}
By definition of the block-tree $\T$, a level difference between $X$ and $Y$ for $(X,Y)\in\T$ can only happen, if ${\rm sons}(X)=\emptyset$ or ${\rm sons}(Y)=\emptyset$.
Assume ${\rm sons}(X)=\emptyset$. In this case, we have ${\rm level}(Y)\geq {\rm level}(X)$. Then, we have
\begin{align*}
 2^{-{\rm level}(X)}\simeq C_{\rm leaf}/N\lesssim |Y|/N\simeq 2^{-{\rm level}(Y)},
\end{align*}
with hidden constants depending only on $C_B$ and $D$. This implies ${\rm level}(Y)\leq {\rm level}(X) + C$ for some constant $C>0$ which depends only on $C_{\rm leaf}$, $D$, and $C_B$ from~\eqref{eq:shape}.
From this we derive~\eqref{eq:eq} by use of~\eqref{eq:shape}.
\end{proof}

\begin{lemma}\label{lem:distH2}
 Given the definition of $\T_{\rm far}$ in Section~\ref{sec:h2}, there exists a constant $C>0$ such that all 
$(X,Y)\in\T_{\rm far}$ satisfy
\begin{align}\label{eq:dist}
  C^{-1}{\rm diam}(B_{X})\leq {\rm dist}(B_{X},B_Y)\leq C\,{\rm 
diam}(B_{X}).
 \end{align}
\end{lemma}
\begin{proof}
By Lemma~\ref{lem:bb}, we have
\begin{align*}
  \max\{{\rm diam}(B_X),{\rm diam}(B_Y)\}\simeq 2^{-{\rm level}(X)/d}.
\end{align*}
For $(X,Y)\in{\rm sons}(X^\prime,Y^\prime)$, we obtain additionally
\begin{align*}
 {\rm dist}(B_{X^\prime},B_{Y^\prime})+2^{-{\rm level}(X^\prime)/d}\gtrsim {\rm dist}(B_X,B_Y).
\end{align*}
 By definition of the block-partitioning, for $(X,Y)\in\T_{\rm far}$ there holds 
that $B_X,B_Y$ satisfy~\eqref{eq:adm}
 and $B_{X^\prime},B_{Y^\prime}$ do not satisfy~\eqref{eq:adm}. Altogether, this implies
 \begin{align*}
  {\rm dist}(B_X,B_Y)\lesssim {\rm dist}(B_{X^\prime},B_{Y^\prime})+2^{-{\rm level}(X^\prime)/d}\lesssim \Big(\frac1\eta+1\Big) 2^{-{\rm level}(X^\prime)/d}\lesssim \max\{{\rm diam}(B_X),{\rm diam}(B_Y)\},
 \end{align*}
where we used $|{\rm level}(X^\prime)-{\rm level}(X)|\leq 1$. This concludes the proof.
%
%
%
\end{proof}

The following lemma gives some basic facts about tensorial 
{Chebychev}-interpolation (see, e.g.,~\cite[Section 4.4]{h2mat})
\begin{lemma}\label{lem:ceb}
 Let $f\colon B\to \R$ for an axis parallel box $B\subseteq \R^{2d}$ such that 
$\partial_j^k f\in L^\infty(B)$ for all $j=1,\ldots,d$ and all $0\leq k\leq p$.
 Then,  the tensorial {Chebychev}-interpolation operator of order $p$, $I_p\colon 
C(B) \to \PP^p(B)$ satisfies
 \begin{align}\label{eq:Ierr}
  \sup_{\bx\in B}|I_p f(\bx)-f(\bx)|\leq 2d\Lambda_p^{2d-1}4 \frac{4^{-p}}{(p+1)!}{\rm 
diam}(B)^p\sum_{i=1}^{2d}\norm{\partial_{\bx_i}^{p+1} f}{L^\infty(B)},
 \end{align}
 where
\begin{align}\label{eq:lambda}
 \Lambda_p:=\sup_{f\in C([-1,1])}\frac{\norm{I_p^\bx 
f}{L^\infty([-1,1])}}{\norm{f}{L^\infty([-1,1])}}\leq \frac{2}{\pi}\log(p+1)+1
\end{align}
is the operator norm of the one dimensional {Chebychev} interpolation operator
\end{lemma}
\begin{proof}
 It is well-known that the one dimensional {Chebychev} interpolation operator 
$I_p^{\bx}$ satisfies the error estimate for any $f\in C([-1,1])$
\begin{align*}
 \norm{u-I_p^\bx f}{L^\infty([-1,1])}\leq 4 \frac{2^{-p}}{(p+1)!}\norm{\partial^{(p+1)} 
f}{L^\infty([-1,1])}
\end{align*}
with an operator norm given in~\eqref{eq:lambda}.
Consider $B:=[-1,1]^{2d}$. Then, there holds with $I_p^{\bx_i}$ denoting 
interpolation in the $\bx_i$-variable $i\in\{1,\ldots,2d\}$
\begin{align*}
 |f-I_p f|&=|f-I_p^{\bx_1}f + I_p^{\bx_1}f - I_p^{\bx_2}I_p^{\bx_1} f +\ldots - I_p f|\\
 &\leq \sum_{i=1}^{2d} 4 \frac{2^{-p}}{p!}\norm{\partial_{\bx_i}^p 
I_p^{\bx_1}(I_p^{\bx_2}\ldots I_p^{\bx_{i-1}})f}{L^\infty(B)}\leq
 \sum_{i=1}^{2d}\Lambda_p^{i-1}4 \frac{2^{-p}}{p!}\norm{\partial_{\bx_i}^{(p+1)} 
f}{L^\infty(B)}\\
 &\leq 2d\Lambda_p^{2d-1}4 \frac{2^{-p}}{p!}\norm{\partial_{\bx_i}^{(p+1)} 
f}{L^\infty(B)}.
\end{align*}
Since, for any affine transformation $A\colon \R^{2d}\to \R^{2d}$, we have 
$I_p(f\circ A)= I_p(f)\circ A$, a standard scaling argument concludes the 
proof. 
\end{proof}

\begin{proof}[Proof of Proposition~\ref{prop:h2}]

We start by proving that $\lambda_{\rm min}(\boldsymbol{C}_p)>0$ if $p$ satisfies~\eqref{eq:ppos}.
To that end, note
\begin{align*}
 \lambda_{\rm min}(\boldsymbol{C}_p)&=\min_{\bz\in\R^N\setminus\{0\}} \frac{(\boldsymbol{C}_p\bz)^T\bz}{|\bz|}\geq \min_{\bz\in\R^N\setminus\{0\}} \frac{(\boldsymbol{C}\bz)^T\bz}{|\bz|}
 -\sup_{\bz\in\R^N\setminus\{0\}} \frac{((\boldsymbol{C}_p-\boldsymbol{C})\bz)^T\bz}{|\bz|}\\
 &\geq \lambda_{\rm min}(\boldsymbol{C}) - \norm{\boldsymbol{C}-\boldsymbol{C}_p}{2}\geq \lambda_{\rm min}(\boldsymbol{C}) - \norm{\boldsymbol{C}-\boldsymbol{C}_p}{F},
\end{align*}
since the Frobenius norm is an upper bound for the spectral norm. By use of~\eqref{eq:h2err} (which is proved below) and~\eqref{eq:ppos}, we conclude $ \lambda_{\rm min}(\boldsymbol{C}_p)>0$.

To see~\eqref{eq:h2err}, we first estimate the maximal depth of the tree $\T_{\rm cl}$. 
With~\eqref{eq:shape2}--\eqref{eq:shape3}, we obtain $C_{\rm leaf}\leq |X| 
\lesssim 2^{-{\rm level}(X)}$ for all $X\in\T_{\rm cl}$ with ${\rm sons}(X)\neq 
\emptyset$.
 Thus, there holds
\begin{align*}
\max_{X\in\T_{\rm cl}}{\rm level}(X)\lesssim \log(| \NN|).
\end{align*}
Second, we bound the so-called sparsity constant
\begin{align*}
 C_{\rm sparse}&:= \max_{X\in \T_{\rm cl}}\Big(|\set{Y\in \T_{\rm cl}}{(X,Y)\in \T_{\rm 
near}\cup\T_{\rm far}}|\\
&\qquad+|\set{Y\in \T_{\rm cl}}{(Y,X)\in \T_{\rm near}\cup\T_{\rm 
far}}|\Big).
\end{align*}
{The $H$-matrix case can be found in~\cite[Lemma~4.5]{hmatorig}. For the $H^2$-matrix case,} the combination of~\eqref{eq:eq} and~\eqref{eq:dist} (from Lemma~\ref{lem:distH2}) shows that $(X,Y)\in 
\T_{\rm far}$ only if $B_Y$ touches the (hyper-) annulus with center $B_X$ and 
radii $C^{-1}{\rm diam}(B_X)$ and $C{\rm diam}(B_X)$.
By comparing the volumes of this annulus and of $B_Y$ and using the fact that 
all the bounding boxes are disjoint, we see that the number of $Y$ such that 
$(X,Y)\in \T_{\rm far}$ is bounded in terms of $C$ and 
the constants in~\eqref{eq:shape}.

For $Y\in\T_{\rm cl}$ such that $(X,Y)\in\T_{\rm near}$, we have 
with~\eqref{eq:shape}--\eqref{eq:eq}
\begin{align*}
 {\rm diam}(B_X)\simeq \max\{{\rm diam}(B_X),{\rm diam}(B_Y)\}> \eta \,{\rm 
dist}(B_X,B_Y).
\end{align*}
Again, comparing the volumes of the ball with radius ${\rm diam}(B_X)$ and of 
$B_Y$, we see that the number of $Y$ such that $(X,Y)\in\T_{\rm near}$ is 
bounded in terms of the constants in~\eqref{eq:shape}.
Altogether, we bound $C_{\rm sparse}$ uniformly in terms of the constants of 
Lemma~\ref{lem:bb}.
Now,~\cite[Lemma~3.38]{h2mat} proves the estimate for storage requirements 
and~\cite[Theorem~3.42]{h2mat} proves the estimate for matrix-vector 
multiplication.

It remains to prove the error estimate (see also~\cite[Section~4.6]{h2mat} for the integral operator case). To that end, note that since the near 
field $\T_{\rm near}$ is stored exactly, there holds
\begin{align*}
 \norm{\boldsymbol{C}-\boldsymbol{C}_p}{F}^2 = \sum_{(X,Y)\in\T_{\rm far}} \norm{\boldsymbol{C}|_{I(X)\times I(Y)} - 
V^X M^{XY} (W^Y)^T}{F}^2.
\end{align*}
Given, $(i,j)\in I(X)\times I(Y)$, we have with the interpolation operator $I_p$ 
from Lemma~\ref{lem:ceb} and~\eqref{eq:as}
\begin{align*}
 |\boldsymbol{C}_{ij}-(\boldsymbol{C}_p)_{ij}|&=\big|\varrho(\bx_i,\bx_j)-\sum_{n,m=1}^{p^d} \varrho(q_n^X,q_m^Y) 
L_n^X(\bx_i)L_m^Y(\bx_j)\big|=|\varrho(\bx_i,\bx_j)-(I_p c)(\bx_i,\bx_j)|\\
 &\lesssim (\log(p)+1)^{2d-1} \frac{4^{-p}}{(p+1)!}{\rm diam}(B_X\times 
B_Y)^p\sum_{i=1}^d\big(\norm{\partial_{\bx_i}^{(p+1)} 
c}{L^\infty(B)}+\norm{\partial_{\by_i}^{(p+1)} c}{L^\infty(B_X\times B_Y)}\Big)\\
 &\lesssim(\log(p)+1)^{2d-1} \frac{4^{-p}}{(p+1)!}{\rm diam}(B_X\times 
B_Y)^p(c_2{\rm dist}(B_X,B_Y))^{-p}p!.
\end{align*}
With the admissibility condition~\eqref{eq:adm}, we get
\begin{align*}
 {\rm diam}(B_X\times B_Y)\lesssim \max\{{\rm diam}(B_X),{\rm diam}(B_Y)\}\leq 
\eta {\rm dist}(B_X,B_Y)
\end{align*}
 and hence
\begin{align*}
 |\boldsymbol{C}_{ij}-(\boldsymbol{C}_p)_{ij}|&\lesssim (\log(p)+1)^{2d-1}\big(\frac{\eta}{4c_2}\big)^{p}.
\end{align*}
The combination of the above estimates concludes the proof.

\end{proof}

\bibliographystyle{plain}
\bibliography{literature}
\end{document}